\documentclass[a4paper,reqno,11pt]{amsart}

\usepackage{amsmath, amsfonts, amssymb, amsthm, amscd}
\usepackage{graphicx}
\usepackage{psfrag}
\usepackage{perpage}
\usepackage{url}
\usepackage{color}
\usepackage{mathrsfs}
\usepackage{mathabx}
\usepackage{scalerel}
\usepackage{stmaryrd}
\usepackage[normalem]{ulem}

\usepackage{caption,subcaption}
\usepackage{subdepth}

\usepackage{dsfont} 

\usepackage[T1]{fontenc}
\usepackage{microtype}

\usepackage[a4paper,scale={0.72,0.74},marginratio={1:1},footskip=7mm,headsep=10mm]{geometry}

\usepackage{hyperref}

\makeatletter
\def\@secnumfont{\bfseries\scshape}

\def\section{\@startsection{section}{1}
  \z@{.9\linespacing\@plus\linespacing}{.5\linespacing}%
  {\normalfont\large\bfseries\scshape\centering}}

\def\subsection{\@startsection{subsection}{2}%
  \z@{.5\linespacing\@plus.7\linespacing}{-.5em}%
  {\normalfont\bfseries\scshape}}

\def\subsubsection{\@startsection{subsubsection}{3}%
  \z@{.5\linespacing\@plus.7\linespacing}{-.5em}%
  {\normalfont\scshape}}

\def\specialsection{\@startsection{section}{1}%
  \z@{\linespacing\@plus\linespacing}{.5\linespacing}%
  {\normalfont\centering\large\bfseries\scshape}}
\makeatother

\makeatletter

\renewenvironment{proof}[1][\proofname]{\par
\pushQED{\qed}%
\normalfont \topsep4\p@\@plus4\p@\relax
\trivlist
\item[\hskip\labelsep
\bfseries
#1\@addpunct{.}]\ignorespaces
}{%
\popQED\endtrivlist\@endpefalse
}
\makeatother

\makeatletter
\newcommand \Dotfill {\leavevmode \leaders \hb@xt@ 6pt{\hss .\hss }\hfill \kern \z@}
\makeatother

\makeatletter
\def\@tocline#1#2#3#4#5#6#7{\relax
  \ifnum #1>\c@tocdepth 
  \else
    \par \addpenalty\@secpenalty\addvspace{#2}%
    \begingroup \hyphenpenalty\@M
    \@ifempty{#4}{%
      \@tempdima\csname r@tocindent\number#1\endcsname\relax
    }{%
      \@tempdima#4\relax
    }%
    \parindent\z@ \leftskip#3\relax \advance\leftskip\@tempdima\relax
    \rightskip\@pnumwidth plus4em \parfillskip-\@pnumwidth
    #5\leavevmode\hskip-\@tempdima
      \ifcase #1
       \or\or \hskip 1.65em \or \hskip 3.3em \else \hskip 4.95em \fi%
      #6\nobreak\relax
    \Dotfill
    \hbox to\@pnumwidth{\@tocpagenum{#7}}\par
    \nobreak
    \endgroup
  \fi}
\makeatother

\makeatletter
\def\l@section{\@tocline{1}{0pt}{1pc}{}{\scshape}}
\renewcommand{\tocsection}[3]{%
\indentlabel{\@ifnotempty{#2}{\ignorespaces#1 #2.\hskip 0.7em}}#3}
\def\l@subsection{\@tocline{2}{0pt}{1pc}{5pc}{}}

\def\l@subsubsection{\@tocline{3}{0pt}{1pc}{7pc}{}}

\makeatother

\numberwithin{equation}{section}

\newtheoremstyle{mytheorem}{.7\linespacing\@plus.3\linespacing}{.7\linespacing\@plus.3\linespacing}%
     {\itshape}
     {}
     {\bfseries}
     {. }
     {0.3ex}
     {\thmname{{\bfseries #1}}\thmnumber{ {\bfseries #2}}\thmnote{ (#3)}}  

\theoremstyle{mytheorem}

\newtheorem{theorem}{Theorem}[section]
\newtheorem{lemma}[theorem]{Lemma}
\newtheorem{proposition}[theorem]{Proposition}
\newtheorem{corollary}[theorem]{Corollary}
\newtheorem{remark}[theorem]{Remark}
\newtheorem{definition}[theorem]{Definition}

\newtheorem{metatheorem}[theorem]{Metatheorem}

\newcommand{\bbD}{{\ensuremath{\mathbb D}} }
\newcommand{\bbE}{{\ensuremath{\mathbb E}} }

\newcommand{\bbP}{{\ensuremath{\mathbb P}} }

\newcommand{\bbZ}{{\ensuremath{\mathbb Z}} }

\newcommand{\cA}{{\ensuremath{\mathcal A}} }

\newcommand{\cC}{{\ensuremath{\mathcal C}} }
\newcommand{\cD}{{\ensuremath{\mathcal D}} }

\newcommand{\cF}{{\ensuremath{\mathcal F}} }
\newcommand{\cG}{{\ensuremath{\mathcal G}} }
\newcommand{\cH}{{\ensuremath{\mathcal H}} }

\newcommand{\cK}{{\ensuremath{\mathcal K}} }

\newcommand{\cN}{{\ensuremath{\mathcal N}} }

\newcommand{\cU}{{\ensuremath{\mathcal U}} }

\newcommand{\cW}{{\ensuremath{\mathcal W}} }

\DeclareMathSymbol{\leqslant}{\mathalpha}{AMSa}{"36} 
\DeclareMathSymbol{\geqslant}{\mathalpha}{AMSa}{"3E} 
\DeclareMathSymbol{\eset}{\mathalpha}{AMSb}{"3F}     

\newcommand{\sumtwo}[2]{\sum_{\substack{#1 \\ #2}}} 

\newcommand{\be}{\begin{equation}}
\newcommand{\ee}{\end{equation}}
\newcommand\blue{ }

\newcommand{\asto}[1]{\underset{{#1}\to\infty}{\longrightarrow}}

\newcommand{\R}{\mathbb{R}}

\newcommand{\Z}{\mathbb{Z}}
\newcommand{\N}{\mathbb{N}}

\newcommand{\PEfont}{\mathrm}

\newcommand{\p}{\ensuremath{\PEfont P}}

\newcommand{\E}{\ensuremath{\PEfont  E}}
\renewcommand{\P}{\p}

\DeclareMathOperator{\bbvar}{\ensuremath{\mathbb{V}ar}}
\DeclareMathOperator{\bbcov}{\ensuremath{\mathbb{C}ov}}

\newcommand{\ind}{\mathds{1}}

\newcommand{\eps}{\varepsilon}
\renewcommand{\epsilon}{\varepsilon}
\renewcommand{\theta}{\vartheta}
\renewcommand{\rho}{\varrho}

\newenvironment{myenumerate}{%
\renewcommand{\theenumi}{\arabic{enumi}}%
\renewcommand{\labelenumi}{{\rm(\theenumi)}}%
\begin{list}{\labelenumi}
	{%
	\setlength{\itemsep}{0.4em}%
	\setlength{\topsep}{0.5em}%
	\setlength\leftmargin{2.45em}%
	\setlength\labelwidth{2.05em}%
	\setlength{\labelsep}{0.4em}%
	\usecounter{enumi}%
	}%
	}%
{\end{list}
}

\newenvironment{aenumerate}{%
\renewcommand{\theenumi}{\alph{enumi}}%
\renewcommand{\labelenumi}{{\rm(\theenumi)}}%
\begin{list}{\labelenumi}
	{%
	\setlength{\itemsep}{0.4em}%
	\setlength{\topsep}{0.5em}%
	\setlength\leftmargin{2.45em}%
	\setlength\labelwidth{2.05em}%
	\setlength{\labelsep}{0.4em}%
	\usecounter{enumi}%
	}%
	}%
{\end{list}
}

{\end{list}
}

{\end{myenumerate}}

\newenvironment{myitemize}{%
\begin{list}{$\bullet$}%
 	{%
	\setlength{\itemsep}{0.4em}%
	\setlength{\topsep}{0.5em}%
	\setlength\leftmargin{2.65em}%
	\setlength\labelwidth{2.65em}%
	\setlength{\labelsep}{0.4em}%
	}%
	}%
{\end{list}}

\renewenvironment{itemize}{
\begin{myitemize}}%
{\end{myitemize}}

\MakePerPage[2]{footnote} 

\date{\today}

\newcommand\dd{\mathrm{d}}

\newcommand\bulk{\mathrm{bulk}}

\newcommand\sfQ{\mathsf Q}

\newcommand\sfU{\mathsf U}

\newcommand\sfa{\mathsf a}

\newcommand\sfq{\mathsf q}

\newcommand\even{\mathrm{even}}

\newcommand\diff{\mathrm{diff}}

\newcommand\sfc{\mathsf{c}}

\newcommand{\ev}[1]{[\![ #1 ]\!]}
\newcommand\scrC{\mathscr{C}}

\newcommand\rme{\mathrm{e}}

\newcommand\off{\mathrm{off}}
\newcommand\av{\mathrm{av}}

\newcommand{\assign}{:=}

\newcommand{\cdummy}{\cdot}

\newcommand{\mathe}{\mathrm{e}}
\newcommand{\nobracket}{}

\usepackage{scalerel,stackengine}
\newcommand\pig[1]{\scalerel*[5pt]{\big#1}{%
  \ensurestackMath{\addstackgap[1.5pt]{\big#1}}}}
\newcommand\pigl[1]{\mathopen{\pig{#1}}}
\newcommand\pigr[1]{\mathclose{\pig{#1}}}

\begin{document}

\title[Singularity and regularity of the critical 2D SHF]{Singularity and regularity\\
of the critical 2D Stochastic Heat Flow}

\author[F. Caravenna]{Francesco Caravenna}
\address{Dipartimento di Matematica e Applicazioni\\
 Universit\`a degli Studi di Milano-Bicocca\\
 via Cozzi 55, 20125 Milano, Italy}
\email{francesco.caravenna@unimib.it}

\author[R. Sun]{Rongfeng Sun}
\address{Department of Mathematics\\
National University of Singapore\\
10 Lower Kent Ridge Road, 119076 Singapore
}
\email{matsr@nus.edu.sg}

\author[N. Zygouras]{Nikos Zygouras}
\address{Department of Mathematics\\
University of Warwick\\
Coventry CV4 7AL, UK}
\email{N.Zygouras@warwick.ac.uk}

\begin{abstract}
The Critical 2D Stochastic Heat Flow (SHF) provides a natural
candidate solution to the ill-posed 2D Stochastic Heat Equation with multiplicative
space-time white noise. In this paper, we
initiate the investigation of
the spatial properties of the SHF.
We prove that, as a random measure on $\R^2$, it
is a.s.\ singular w.r.t.\ the Lebesgue measure.
This is obtained by probing a ``{\it quasi-critical''} regime and
showing the asymptotic log-normality of the mass
assigned to vanishing balls, as the disorder strength is sent to zero at a suitable rate,
accompanied by similar results for critical 2D directed polymers.
We also describe the regularity of the SHF, showing that it is a.s.\ H\"older $\cC^{-\epsilon}$
for any $\epsilon>0$, implying the absence of atoms, and we establish local convergence
to zero in the long time limit.
\end{abstract}

\maketitle

\vskip -12mm
\begin{figure}[h]
\centering
\hfill \
\begin{minipage}[b]{.45\linewidth}
\centering
\includegraphics[width=\linewidth]{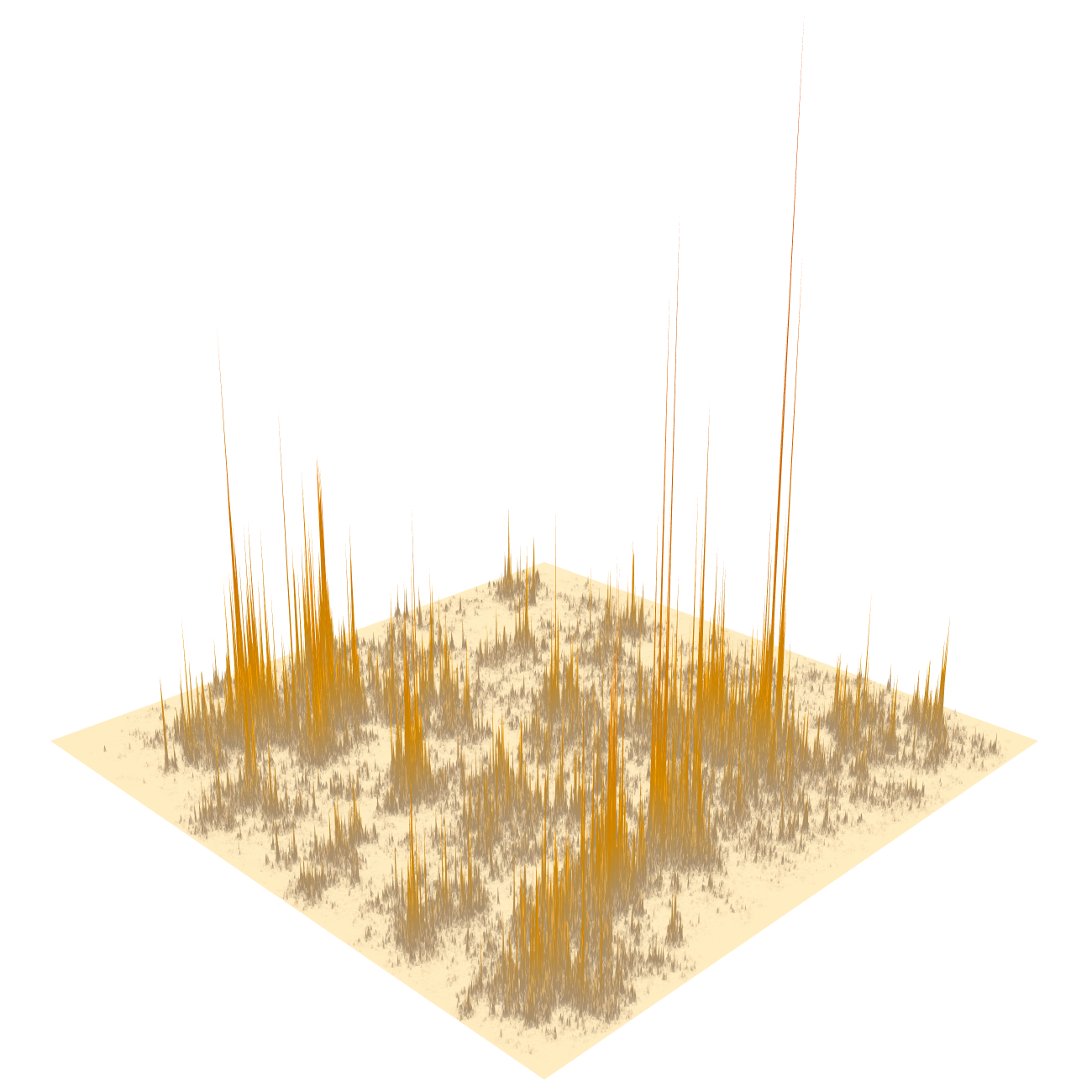}
\end{minipage}%
\begin{minipage}[b]{.43\linewidth}
\centering
\includegraphics[width=\linewidth]{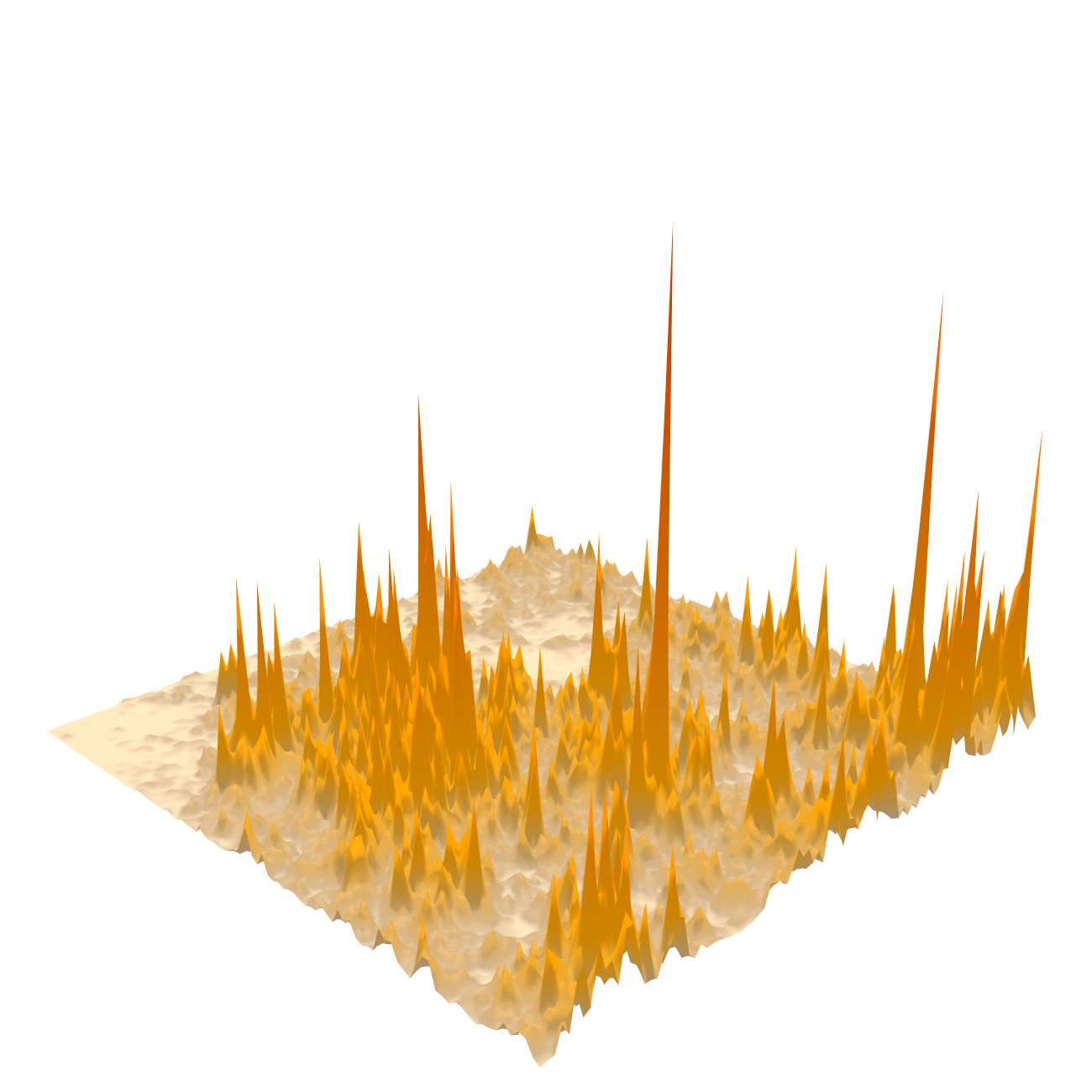}
\end{minipage}
\hfill \
\caption*{\it \footnotesize
The picture on the left is a simulation of the Critical 2D SHF and illustrates its singularity.
The picture on the right is a simulation in the quasi-critical regime,
slightly below the critical window, which is
smoother and will be used to approximate the Critical 2D SHF.
}
\end{figure}

\section{Introduction}

The Critical 2D Stochastic Heat Flow (SHF)
was constructed in \cite{CSZ23a} as a family of measure-valued processes $\mathscr{Z}_{t}^\theta(\dd x)$ with
disorder strength parameter $\theta\in \R$, which give non-trivial solutions to the ill-defined
two-dimensional Stochastic Heat Equation (SHE)
\begin{align}\label{eq:SHE} \tag{SHE}
	\partial_t u(t,x) =\frac{1}{2} \Delta u(t,x) +\beta  \, u(t,x)\, \xi(t,x) \,,\qquad t>0, \ x\in\R^2,
\end{align}
where $\xi(t,x)$ denotes space-time white noise (more recently, an axiomatic characterisation
of the SHF as a continuous measure-valued process was given in \cite{T24}). Dimension $2$ is critical for the SHE as
it is the dimension where the singularity of the noise matches the smoothing effect of the
Laplacian and thus cannot be treated perturbatively.  A comprehensive theory of singular
Stochastic PDEs (SPDEs) below their critical dimension (known as ``subcritical SPDEs'') exists
thanks to the breakthrough theories of regularity structure \cite{H14}, paracontrolled distributions
\cite{GIP15}, renormalisation group theory \cite{K16, D22},
energy solutions \cite{GJ14} and the huge volume of work they have inspired.
The endeavour of treating critical singular SPDEs is only now starting to emerge \cite{CSZ24, CT24}, and the
Critical 2D SHF is the first example describing a non-trivial and non-Gaussian solution to a critical equation
at its phase transition point.

The Critical 2D SHF is an interesting object with a rich structure (see the recent review \cite{CSZ24}). However, its
fine properties have not yet been explored.
The purpose of this paper is to initiate the study of its spatial characteristics. Consider the Critical 2D SHF
$\mathscr{Z}_{t}^\theta(\dd x)$ started from the Lebesgue measure $\mathscr{Z}_{0}^\theta(\dd x)=\dd x$. We will
prove the following results: For every $\theta\in \R$ and $t>0$, almost surely,
\begin{itemize}
\item $\mathscr{Z}_{t}^\theta(\dd x)$ is singular with respect to the Lebesgue measure (Theorem~\ref{th:singularity-SHF});
\item $\mathscr{Z}_{t}^\theta(\dd x)$ barely fails to be a function in the sense that it is in the negative H\"older spaces $C^{-\epsilon}$
for every $\epsilon>0$ (Theorem~\ref{th:regularity-SHF}), and hence contains no atoms.
\end{itemize}
Moreover, we show that
\begin{itemize}
\item $\mathscr{Z}_{t}^\theta(\dd x)$ converges in law to the $0$ measure as $t\to\infty$, in the sense that the mass assigned
to any finite ball converges to $0$ in probability (Theorem~\ref{th:long-time-SHF}).
\end{itemize}
The almost surely singularity of $\mathscr{Z}_{t}^\theta(\dd x)$ is a consequence of the result that:
\begin{itemize}
\item The mass density $\frac{1}{|B(x, \delta)|} \mathscr{Z}_{t}^\theta(B(x, \delta))$ on a ball $B(x, \delta)$ of shrinking
radius $\delta\downarrow 0$ converges to a log-normal limit, if the disorder strength parameter $\theta=\theta(\delta)\to -\infty$ at a suitable
rate (Theorem~\ref{th:log-normality-SHF}).
\end{itemize}
This is obtained by proving an analogous result (Theorem~\ref{th:log-normality}) for the averaged partition function of the directed polymer model in the so-called {\em quasi-critical regime}, which was introduced in \cite{CCR23} as an interpolation between the sub-critical and critical regimes of the 2D directed polymer model. We remark that the directed polymer model is a very interesting and important disordered system on its own \cite{C17, Z24}, and the Critical 2D SHF was first constructed in \cite{CSZ23a} as the unique limit of 2D directed polymer partition functions in the critical regime.

The proof of Theorem~\ref{th:log-normality} constitutes the bulk of this paper and is accomplished via
an approximate \emph{multiplicative,
multi-scale decomposition of the polymer partition function}, see \eqref{eq:decomp}.
Similar decompositions have also been applied in the sub-critical
regime \cite{CD24,CNZ25}. The novelty of our contribution is that we push such a decomposition to the quasi-critical regime,
up to the onset of criticality, setting the foundations for
understanding fine properties of the SHF.
Along the way, we derive a \emph{general hypercontractive bound} on the higher moments of the averaged polymer partition
function in terms of its second moment (Theorem~\ref{th:genmombou}), valid in all regimes up to criticality.

\smallskip

For the rest of the introduction, we will first recall the construction and basic properties of the Critical 2D Stochastic Heat Flow. Our main results for the Critical 2D SHF and the directed polymer model will then be stated in Section~\ref{S:newSHF} and~\ref{sec:DPRM}
respectively. In Section~\ref{S:mom}, we will formulate the hypercontractive moment bound mentioned above. Lastly in Section~\ref{S:ext}, we explain how limiting properties of the Critical 2D SHF as the disorder parameter $\theta\downarrow -\infty$ can always be analysed by studying
the directed polymer partition functions in the quasi-critical regime.

\subsection{Briefing on the Critical 2D SHF} \label{S:Briefing}
To make sense of the two-dimensional stochastic heat equation \eqref{eq:SHE}, we need to first perform a regularisation
on small spatial scales (ultraviolet cutoff) and then take a suitable limit. The regularisation can be accomplished in different
ways. One way is to consider the mollified SHE
\begin{equation} \label{eq:mSHE}
\partial_t u^\epsilon = \frac{1}{2}\Delta u^\epsilon + \beta_\epsilon\, \xi^\epsilon(t,x) u^\epsilon,
\end{equation}
where $\eps>0$ is the spatial scale of regularisation, $j(\cdot)$ is a smooth probability kernel on $\R^2$ and $j^\eps(x)=\eps^{-2}j(x/\eps)$ is
its scaled version, while $\xi^\epsilon := j^\epsilon * \xi$ is the spatial mollification of the white noise $\xi$.

Another way is to discretize space and time. Namely, the white noise $\xi$ is replaced by a family of i.i.d.\ random variables $\omega=(\omega(n,x))_{n\in \N, x\in \Z^2}$ with law $\bbP$ and expectation $\bbE$, and
\begin{equation}\label{eq:lambda}
\begin{gathered}
	\bbE[\omega]=0 \,, \qquad \bbE[\omega^2]=1 \,, \qquad
	\exists\, \beta_0>0: \quad 	\lambda(\beta)
	:= \log \bbE[\rme^{\beta\omega}] < \infty \quad \forall \beta \in [0, \beta_0] \,.
\end{gathered}
\end{equation}
Replacing derivatives in \eqref{eq:SHE} by suitable difference operators, the solution can be expressed in terms of the (point-to-point) partition functions of the {\em directed polymer model}:
\begin{align} \label{eq:paf}
	Z_{M,N}^{\beta_N}(y,x) &=
	\E\bigg[ \rme^{\sum_{n=M+1}^{N-1}
	\{\beta_N \omega(n,S_n) - \lambda(\beta_N)\}} \,\, \ind_{\{S_N=x\}}
	\,\bigg|\, S_M = y \bigg] \,,
\end{align}
where $\E$ is the expectation with respect to the 2D simple symmetric random walk $S=(S_n)_{n\geq 0}$. More precisely, the diffusively rescaled plane-to-point partition functions (with $\Z_{\rm even}:=\{2n: n\in\Z\}$ and $\Z^2_{\rm even}:=\{(x, y)\in\Z^2: x+y\in \Z_{\rm even}\}$)
\begin{equation}\label{eq:uN}
u^{(N)}(t, x):= \sum_{y\in \Z^2_{\rm even}} Z_{0,Nt}^{\beta_N}(y,\sqrt{N} x), \qquad (t, x) \in \frac{1}{N}\Z_{\rm even} \times \frac{1}{\sqrt N} \Z^2_{\rm even},
\end{equation}
is the analogue of $u^\eps(t, x)$ and solves a version of \eqref{eq:SHE}, discretised on spatial scale $1/\sqrt{N}$ and time scale $1/N$, with initial condition $u^{(N)}(0, \cdot)\equiv 1$.

It was first shown in \cite{CSZ17b} that on the intermediate disorder scale
$$
\beta_N= \hat\beta \sqrt{\frac{\pi}{\log N}},
$$
the directed polymer partition functions $u^{(N)}(t, x)$ undergo a phase transition (with critical value $\hat\beta_c=1$) in two different senses:
\begin{itemize}
\item For each $(t, x)\in [0, 1)\times\R^2$, $u^{(N)}(t, x)$ converges to a log-normal limit if $\hat\beta<1$ and converges to $0$ if $\hat\beta\geq 1$;

\item The centered and rescaled field $\beta_N^{-1}(u^{(N)}(t, x)-1)$ converges for $\hat\beta < 1$ to a Gaussian limit that solves the additive SHE (or Edwards-Wilkinson equation)
\begin{align*}
\partial_t v = \frac{1}{2}\Delta v + \sqrt{\frac{\hat\beta^2}{1-\hat\beta^2}} \, \xi,
\end{align*}
where the noise strength diverges as $\hat\beta\uparrow 1$.
\end{itemize}
The same results were also proved in \cite{CSZ17b} for the solution $u^\eps$ of the mollified SHE \eqref{eq:mSHE} on the intermediate disorder scale
$\beta_\eps = \hat\beta \sqrt{\frac{2\pi}{\log 1/\eps}}$. Therefore in the subcritical regime $\hat\beta<1$, the 2D SHE is essentially the Edwards-Wilkinson equation.

It is the critical regime $\hat\beta=1$ that leads to a non-Gaussian limit, called the {\em Critical 2D Stochastic Heat Flow (SHF)}. It turns out
there is a whole critical window around $\hat\beta=1$, determined by the relation
\begin{align}\label{choiceb}
e^{\lambda(2\beta_N)-2\lambda(\beta_N)}-1 = \frac{\pi}{\log N} \Big(1+\frac{\theta+o(1)}{\log N} \Big) \quad \mbox{for some}\ \theta\in \R.
\end{align}
For a more explicit expression of $\beta_N$ in terms of $\theta$, see \cite[(3.12)]{CSZ23a}. For the mollifed SHE \eqref{eq:mSHE}, the corresponding
critical window is given by
\begin{align}\label{beta}
\beta^2_\epsilon=\frac{2\pi}{\log\tfrac{1}{\epsilon}}\Big( 1+\frac{\rho+o(1)}{\log\tfrac{1}{\epsilon}} \Big),
\end{align}
where $\rho=\pi \theta+C$ (see \cite[(1.38)]{CSZ19b} for the precise value of $C$).

The main result of \cite{CSZ23a} is that:
\begin{itemize}
\item If $\beta_N$ is chosen to satisfy \eqref{choiceb} for some $\theta\in \R$ and $u^{(N)}(t, \cdot)$ is regarded as a process of random measures on $\R^2$, then $(u^{(N)}(t, \cdot))_{t\geq 0}$ converges in finite dimensional distribution to a unique (in law) measure-valued process $\mathscr{Z}^\theta_t({\rm d}x)$, which was named the {\em Critical 2D Stochastic Heat Flow} in \cite{CSZ23a}.
\end{itemize}
Prior to \cite{CSZ23a}, the tightness of the sequence of random measures $u^{(N)}(t, \cdot)$ follows trivially from first moment bounds, while moment asymptotics of $u^{(N)}(t, \cdot)$ were studied in \cite{BC98, CSZ19b, GQT21}, which determined all positive integer moments of any subsequential weak limit of $u^{(N)}(t, \cdot)$. However, these moments diverge too fast to uniquely determine the limit (a lower bound of order $\exp(c k^2)$ for the $k$-th moment
was given in \cite{CSZ23b}). The uniqueness was finally achieved in \cite{CSZ23a} by showing that the laws of $u^{(N)}(t, \cdot)$ form a Cauchy sequence, and hence must converge to a unique limit. The proof was based on coarse graining, coupled with a Lindeberg replacement principle.

Recently, Tsai \cite{T24} gave an axiomatic characterization of the critical 2D SHF and showed that there is a version that is almost surely continuous in time. This greatly facilitates the proof of convergence to the SHF. In particular, this axiomatic characterization was used in \cite{T24} to show that the solution $u^\epsilon$ of the mollified SHE \eqref{eq:mSHE} in the critical window \eqref{beta} also converges to the SHF. In Tsai's characterisation, the SHF is the unique (in law) continuous measure-valued process that satisfies: (i) Independent ``increments'' property; (ii) An almost sure Chapman-Kolmogorov property (first defined and verified for the SHF by Clark and Mian \cite{CM24}); (iii) matching first four moments with the SHF. The proof was also based on a Lindeberg replacement principle.

We also recall from \cite[Theorem 1.2]{CSZ23a} some basic properties of the Critical 2D SHF (for simplicity, we only consider constant initial configuration $\mathscr{Z}^\theta_0({\rm d}x)={\rm d}x$):
\begin{itemize}
\item (Scaling Covariance) For all $\sfa>0$, we have
\begin{equation}\label{eq:cov}
	(\mathscr{Z}_{\sfa t}^\theta(\dd (\sqrt{\sfa} x)))_{0 \le t <\infty}
	\stackrel{\rm dist}{=}
	(\sfa\, \mathscr{Z}_{t}^{\theta+ \log \sfa}(\dd x))_{0 \le t <\infty} \,.
\end{equation}
Thus zooming out diffusively ($\sfa\uparrow \infty$) increases the disorder strength $\theta$, while zooming in ($\sfa\downarrow 0$)
decreases the disorder strength (cf.\ the pictures on the front page, where the picture on the right is a result of zooming into the picture on the left).

\item (First and Second Moments) We have
\begin{align*}
\E[\mathscr{Z}^\theta_t({\rm d}x)] & = \dd x\,, \\
\E[\mathscr{Z}^\theta_t({\rm d}x) \mathscr{Z}^\theta_t({\rm d}y)] & = K_t^\theta(x, y) \, {\rm d}x{ \, \rm d}y \,,
\end{align*}
where $K^\theta_t(x,y) \sim C \log \frac{1}{|x-y|}$ as $|x-y|\to 0$. It was first computed in \cite{BC98} before the realisation that
this lies in the critical window of a phase transition \cite{CSZ17b}.
\end{itemize}
For more properties of the Critical 2D SHF, see \cite{CSZ23a} and the lecture notes \cite{CSZ24}.

\subsection{New properties of the Critical 2D SHF} \label{S:newSHF}

In this paper we investigate the spatial regularity of the SHF.
We focus on its one-time marginal $\mathscr{Z}_{t}^\theta(\dd x)$,
which is a \emph{locally finite random measure on $\R^2$} with
$\mathscr{Z}_0^\theta(\dd x)={\rm d}x$ and
$\bbE[\mathscr{Z}_{t}^\theta(\dd x)] = \dd x$ for all $t \ge 0$.
Our main result is that, for each $t>0$, $\mathscr{Z}_{t}^\theta(\dd x)$
is almost surely \emph{singular} with respect to the Lebesgue measure, and hence not a function.

Let $\cU_{B(x,\delta)}(\cdot)$ denote the uniform density on the Euclidean ball in $\R^2$:
\begin{equation}\label{eq:uniform}
	\cU_{B(x,\delta)}(\cdot) := \frac{1}{\pi \delta^2} \, \ind_{B(x,\delta)}(\cdot) \qquad
	\text{where } \
	B(x,\delta) := \big\{y \in \R^2 \colon \ |y-x| < \delta \big\} \,.
\end{equation}
We will mostly focus on the SHF $\mathscr{Z}_{t}^\theta(\dd x)$ \emph{averaged over balls}, that is
\begin{equation}\label{eq:averaged-SHF}
	\mathscr{Z}_{t}^\theta(\cU_{B(x,\delta)})
	:= \frac{\mathscr{Z}_{t}^\theta(B(x,\delta))}{\pi \delta ^2} \,.
\end{equation}
We can now state our first main result.

\begin{theorem}[Singularity of SHF] \label{th:singularity-SHF}
Fix any $t > 0$ and $\theta \in \R$.
Almost surely, the SHF
$\mathscr{Z}_{t}^\theta(\dd x)$ is singular with respect to Lebesgue measure on~$\R^2$.
In fact, the following holds:
\begin{equation}\label{eq:singularity-SHF}
	\text{almost surely,} \qquad
	\lim_{\delta \downarrow 0} \; \mathscr{Z}_{t}^\theta\big(\cU_{B(x,\delta)}\big) = 0
	\quad \text{for Lebesgue a.e.\ $x\in\R^2$} \,.
\end{equation}
\end{theorem}

The singularity of the SHF with respect to Lebesgue can be deduced
from property \eqref{eq:singularity-SHF}
via general arguments (see Proposition~\ref{th:singularity}).
In order to prove \eqref{eq:singularity-SHF},
we show that \emph{in the ``weak disorder limit'' $\theta \to -\infty$,
the SHF averaged on balls
$\mathscr{Z}_{t}^\theta\big(\cU_{B(x,\delta)}\big)$
is asymptotically log-normal for radius $\delta = \delta_\theta^\rho \downarrow 0$ vanishing
as any power of a suitable scale~$\delta_\theta$}.

\begin{theorem}[Log-normality of SHF in the weak disorder limit] \label{th:log-normality-SHF}
Let us define
\begin{equation}\label{eq:deltatheta}
	\delta_\theta :=
	\rme^{\frac{1}{2} \theta} = \rme^{-\frac{1}{2} |\theta|}
	\longrightarrow 0 \qquad \text{as } \theta \to -\infty \,.
\end{equation}
Given any $t > 0$ and $x\in\R^2$, the following convergence in distribution holds:
\begin{equation}\label{eq:log-normal-SHF}
	\forall \rho \in (0,\infty) \colon \qquad
	\mathscr{Z}_{t}^\theta \big(\cU_{B(x,\delta_\theta^\rho)}\big)
	\ \xrightarrow[\ \theta \to -\infty \ ]{d} \
	\rme^{\cN(0,\sigma^2) - \frac{1}{2}\sigma^2}
	\qquad \text{with } \sigma^2 = \log(1+\rho) \,.
\end{equation}
\end{theorem}

\begin{remark}
We stress that the log-normality \eqref{eq:log-normal-SHF} emerges as $\theta \to -\infty$.
For fixed $\theta \in \R$,
the SHF averaged on balls $\mathscr{Z}_{t}^\theta\big(\cU_{B(x,\delta)}\big)$
vanishes as $\delta \to 0$, as shown by \eqref{eq:singularity-SHF}.

On a different note, the SHF $\mathscr{Z}_{t}^{\theta}(\dd x)$
is not the exponential of a (generalised) Gaussian field,
i.e.\ it is not a Gaussian Multiplicative Chaos,
see \cite{CSZ23b}.
\end{remark}

In the proof of Theorem~\ref{th:singularity-SHF} we
deduce \eqref{eq:singularity-SHF}
from \eqref{eq:log-normal-SHF} by
exploiting \emph{the monotonicity of fractional moments of $\mathscr{Z}_{t}^\theta(B)$
with respect to~$\theta$} (see Lemma~\ref{th:monotonicity-SHF}). \blue{The same strategy 
was used in \cite{CSZ17b} to show that the point-to-plane partition function of the 2D directed polymer
in the critical window \eqref{choiceb} converges in distribution to $0$ as $N\to\infty$.}

Using the scaling covariance property \eqref{eq:cov}, we also show that
the SHF \emph{locally vanishes} as the time horizon tends to infinity.
\begin{theorem}[Long-time behavior of SHF\footnote{\blue{It was recently shown in \cite{CT25} that \cite{BCT25} that 
$\mathscr{Z}_{t}^\theta({\rm d}x)$ converges locally to $0$ as $\theta\uparrow \infty$, which implies Theorem \ref{th:long-time-SHF} by scaling relations, see \cite[Corollary 1.4 and Remark 6.3]{CT25} and \cite[Remark 1.4]{BCT25}.}}] \label{th:long-time-SHF}
Fix any $\theta \in \R$. Then,
\begin{equation}\label{eq:long-time-SHF}
	\text{for any bounded \blue{Borel} set $A \subset \R^2$:} \qquad
	\mathscr{Z}_{t}^\theta(A)
	\ \xrightarrow[\ t \to \infty \ ]{\rm in \ prob.} \
	0 \,. 
\end{equation} 
\end{theorem}

We finally investigate the regularity of the SHF
$\mathscr{Z}_{t}^\theta(\dd x)$ as a measure on $\R^2$,
showing that it has negative H\"older regularity $\mathcal{C}^{-\epsilon}$
for arbitrary small $\epsilon > 0$ (the definition of negative H\"older spaces
is recalled in Subsection~\ref{sec:regularity-SHF}).
Since positive H\"older spaces $\cC^\epsilon$ consist of functions,
this shows that, in a sense, the SHF $\mathscr{Z}_{t}^\theta(\dd x)$
\emph{barely fails to be a function}.

\begin{theorem}[Regularity of the SHF]\label{th:regularity-SHF}
Fix any $t > 0$ and $\theta\in\R$.
Almost surely, the SHF
$\mathscr{Z}_{t}^\theta(\dd x)$ belongs to $\mathcal{C}^{0-}
:= \bigcap_{\epsilon > 0} \mathcal{C}^{-\epsilon}$
and, hence, contains no atoms.
\end{theorem}

\blue{
The recent work of Nakashima \cite[Section 7]{N25} indicates that the fine regularity of the SHF could be captured by
suitably defined log-H\"older spaces. This is further supported by the finer moment asymptotics of the mass of the SHF on a small ball \cite{LZ26}. However, the interesting task of determining the  precise
logarithmic regularity of the SHF would require a detailed understanding of the structure of its peaks and it would be the subject of future work.
}
\vskip 2mm

The results above are proved in Section~\ref{sec:proof-SHF}.
The proof of Theorems~\ref{th:log-normality-SHF} and~\ref{th:regularity-SHF}
are based on the approximation of the SHF
via \emph{partition functions of directed polymers}, which was used in the
original construction of the SHF in \cite{CSZ23a} and will be recalled next.

\subsection{Results for directed polymers}\label{sec:DPRM}
To define the directed polymer model, let $S = (S_n)_{n\ge 0}$ be the simple symmetric random walk on $\Z^2$ with
law~$\P$ and expectation $\E$. We denote its transition kernel by
\begin{equation} \label{eq:rw}
	q_n (z) \assign \P (S_n = z \,|\, \nobracket S_0 = 0)
	\qquad \text{for } \ n\in\N_0 = \{0\} \cup \N \,, \ z\in\Z^2 \,.
\end{equation}
We define the \emph{expected replica overlap}
$R_N = \E\big[\sum_{n=1}^N \ind_{\{S_n = S'_n\}}\big]$
where $S'$ is an independent copy of~$S$ with $S'_0=S_0=0$.
By the local central limit theorem \eqref{eq:llt}
\begin{equation}\label{eq:RL}
  R_{N} = \sum_{n = 1}^{N}
     \sum_{x \in \mathbb{Z}^2} q_n (x)^2 = \sum_{n = 1}^{N} q_{2 n} (0) =
     \frac{\log N}{\pi} + O(1) \qquad \text{as } N\to\infty
\end{equation}
(see also \cite[Proposition~3.2]{CSZ19a} for a refined asymptotic behavior).

The \emph{environment (disorder)} is given by a family $(\omega(n,z))_{n\in\N, z\in\Z^2}$ of i.i.d.\ random variables satisfying the assumptions
in \eqref{eq:lambda}. Note that $\lambda(\beta) := \log \bbE [\rme^{\beta \omega}] \sim \frac{1}{2} \beta^2$ as $\beta \to 0$.
We introduce the quantity
\begin{equation}\label{eq:sigmabeta}
	\sigma_{\beta}^2 := \bbvar\big[ \rme^{\beta\omega - \lambda(\beta)} \big]
	= \rme^{\lambda(2\beta) - 2\lambda(\beta)}-1
	\, \underset{\beta\downarrow 0}{\sim} \, \beta^2 \,.
\end{equation}

Given $\varphi, \psi : \Z^2 \to \R$, polymer length $N\in\N$, and inverse temperature (or disorder strength) $\beta \ge 0$,
we define the averaged directed polymer partition function as follows:
\begin{equation}\label{eq:Z}
	Z_N^{\beta}(\varphi, \psi) \assign
	\sum_{z\in\bbZ^2} \varphi (z) \,
	Z_N^{\beta}(z,z')  \,\psi(z')
	 \quad \text{with} \quad
	Z_N^{\beta}(z,z')
	:= \E \pig[\rme^{\mathcal{H}_{(0,  N]}^{\beta}(S)}  \,\ind_{S_N=z'}\, \pig| \nobracket S_0 = z \pig] \,,
\end{equation}
where
\begin{equation}\label{eq:H}
	\mathcal{H}^{\beta}_{I}(S) \assign \sum_{n \in I \cap \Z} \big\{ \beta \,
	\omega (n, S_n) - \lambda(\beta) \big\}
	\qquad \text{for } I \subset \R \,.
\end{equation}
When $\psi\equiv 1$, we will simplify notation and write $Z_N^{\beta}(\varphi):= Z_N^{\beta}(\varphi, 1)$.

\begin{remark}
To comply with the periodicity of the simple random walk, we usually consider
$\varphi$ supported on the even sub-lattice
$$
	\mathbb{Z}^2_{\even} := \{(x, y) \in \Z^2 \colon x + y \text{ is even}\} \,.
$$
\end{remark}

\smallskip

As explained in Section~\ref{S:Briefing}, the Critical 2D SHF $\mathscr{Z}_{t}^\theta(\cdot)$ is the
\emph{scaling limit of the diffusively rescaled partition functions} $Z_{tN}^{\beta_N}(\cdot\, \sqrt{N})$
regarded as a random measure on $\R^2$, if the disorder strength $\beta_N$ is chosen to be in the following critical window:
\begin{equation} \label{eq:betacrit}
	\sigma_{\!\beta_N^{\mathrm{crit}}}^2 = \frac{1}{R_N} \left( 1 +
	\frac{\theta}{\log N} \right) \qquad \text{for some } \theta \in \R \,.
\end{equation}
More precisely, denoting by $\ev{x}$ the point in $\Z^2_\even$ closest to $x \in \R^2$,
the following convergence in distribution was proved in \cite[Theorem~1.1]{CSZ23a}:
for any $t > 0$
\begin{equation} \label{eq:weak-conv}
	Z_{tN}^{\beta_N^{\mathrm{crit}}} \big( \ev{x \sqrt{N}} \big) \, \dd x
	\ \xrightarrow[\ N\to\infty \ ]{d} \
	\mathscr{Z}_{t}^\theta(\dd x),
\end{equation}
which are regarded as random variables taking values in the space of locally finite measures on $\R^2$ equipped with the vague topology, i.e.,
the one generated by the integrals $\int \varphi \, \dd \mu$ for continuous and compactly supported test functions
$\varphi: \R^2 \to \R$.

We strengthen this result to a convergence in distribution of random variables taking values in
negative H\"older spaces $\mathcal{C}^{-\epsilon}$
for any $\epsilon > 0$.

\begin{theorem}[Improved convergence to the SHF]\label{th:regularity-DP}
Fix $\theta\in\R$ and consider $\beta_N^{\mathrm{crit}}$
in the critical regime \eqref{eq:betacrit}. For any $t > 0$, the following
convergence in distribution holds:
\begin{equation} \label{eq:strong-conv}
	\forall \epsilon > 0 \colon \qquad
	Z_{tN}^{\beta_N^{\mathrm{crit}}} \big( \ev{x \sqrt{N}} \big) \, \dd x
	\ \xrightarrow[\ N\to\infty \ ]{d} \
	\mathscr{Z}_{t}^\theta(\dd x) \quad \text{ in } \cC^{-\epsilon} \,.
\end{equation}
\end{theorem}

\noindent
This result directly implies Theorem~\ref{th:regularity-SHF} on the regularity of the SHF.
The proof is given in Section~\ref{sec:proof-SHF} by exploiting
moment bounds from \cite{CSZ23a}
(see Proposition~\ref{th:moments}).

\smallskip

We next look back at the log-normality of the SHF averaged on vanishing balls as $\theta \to -\infty$,
see  Theorem~\ref{th:log-normality-SHF}.
We obtain this result via discrete approximations,
namely we deduce it from an analogue result for directed polymer partition functions,
which we state next.

In order to compare the SHF as $\theta \to -\infty$ with
directed polymers, we need to tune the disorder strength $\beta$
in a \emph{quasi-critical regime}
recently investigated in \cite{CCR23}, where we replace $\theta$ in \eqref{eq:betacrit}
by a sequence \emph{$\theta_N = -|\theta_N| \to -\infty$} at an arbitrarily slow rate:
\begin{equation} \label{eq:quasi-crit}
  \sigma_{\!\beta_N^{\mathrm{quasi\text{-}crit}}}^2
  \assign \frac{1}{R_N} \left( 1 -
  \frac{|\theta_N|}{\log N} \right) \qquad \text{where } \
  1 \ll |\theta_N| \ll \log N \,.
\end{equation}
We call this regime \emph{quasi-critical}
because it interpolates between the \emph{critical regime} \eqref{eq:betacrit},
corresponding to $|\theta_N| = O(1)$,
and the \emph{sub-critical regime}
\cite{CSZ17b,CSZ20, CC22, CD24},
corresponding to $|\theta_N| \approx \log N$, see \eqref{eq:sub-critical} below.

Let us consider directed polymer partition functions $Z_N^{\beta}(\varphi)$ with initial conditions
$\varphi$ that are
uniformly distributed on discrete balls, denoted by $\cU_{B(z,R)}$
(same as their continuum counterparts \eqref{eq:uniform},
with some abuse of notation):
\begin{equation}
  \label{eq:phi0}
  \cU_{B(z,R)} (\cdummy)
  \assign \frac{\ind_{B ( z, R) \cap \mathbb{Z}^2_{\even}}
  (\cdummy)}{\big| B( z, R) \cap
  \mathbb{Z}^2_{\even} \big|} \,.
\end{equation}
We can now state our log-normality result
for directed polymers.

\begin{theorem}[Quasi-critical log-normality]
  \label{th:log-normality}
Consider $\beta_N^{\mathrm{quasi\text{-}crit}}$
in the quasi-critical regime \eqref{eq:quasi-crit}
for a given sequence $1 \ll |\theta_N| \ll \log N$.
Define the scale $\delta_N$ by
\begin{equation}\label{eq:theta}
	\delta_N :=
	\rme^{-\frac{1}{2} \, |\theta_N|} \longrightarrow 0
	\qquad \text{as } N\to\infty \,.
\end{equation}
For any $t > 0$ and $x\in\R^2$, the following convergence in distribution holds:
  \begin{equation}\label{eq:log-normal}
\begin{split}
  \forall \rho \in (0,\infty) \colon \qquad
    Z_{tN}^{\beta_N^{\mathrm{quasi\text{-}crit}}} \pig( \, \cU_{B(x\sqrt{N},\delta_N^\rho\sqrt{N})} \, \pig)
        \,\xrightarrow[\,N\to\infty\,]{\ d \ } \,
    & \ \mathe^{\cN (0, \sigma^2) - \frac{1}{2} \sigma^2} \\
    & \ \text{with }
    \sigma^2 = \log ( 1+\rho )
\end{split}
  \end{equation}
and, furthermore, all positive moments converge.
\end{theorem}

Log-normality was first proved in \cite[Theorem~2.8]{CSZ17b}
for $Z_N^\beta(x) := Z_N^\beta(\ind_{\{x\}})$, i.e.\ the partition function
started at a single point~$x$ (also called \emph{point-to-plane} partition function),
when $\beta = \beta_N^{\text{\rm sub-crit}}$ is chosen in the \emph{sub-critical regime}:
\begin{equation}\label{eq:sub-critical}
	\sigma_{\!\beta_N^{\text{\rm \rm sub-crit}}}^2 \sim \frac{\hat\beta^2}{R_N}
	\qquad \text{for some } \ \hat\beta \in (0,1)
	\qquad (\, \text{i.e.} \ \
	\big(\beta_N^{\text{\rm sub-crit}}\big)^2 \sim \tfrac{\pi \, \hat\beta^2 }{\log N} \,) \,.
\end{equation}
Our proof of Theorem~\ref{th:log-normality} also covers
this regime and allows for averaging over balls of \emph{arbitrary
sub-diffusive polynomial radius} $N^{\,\gamma/2 + o(1)}$ as $N\to \infty$, for any $0 \le\gamma < 1$.
The few changes required are described in Remarks~\ref{rem:subcr2} and~\ref{rem:subcr3} (see also Remark~\ref{rem:subcr1}).

\begin{theorem}[Sub-critical log-normality]
  \label{th:sub-critical-log-normality}
Consider $\beta_N^{\text{\rm sub-crit}}$
in the sub-critical regime \eqref{eq:sub-critical}
for some $\hat\beta \in (0,1)$.
For any $t  > 0$, $x\in\R^2$, one has the convergence in distribution
  \begin{equation}\label{eq:log-normal-sub-critical}
\begin{split}
	\forall \gamma \in [0,1) \colon \qquad
    Z_{tN}^{\beta_N^{\text{\rm sub-crit}}} \pig( \,\cU_{B(x\sqrt{N},
    \sqrt{N}^{\,\gamma+ o(1)})} \,\pig)
    \,\xrightarrow[\,N\to\infty\,]{\ d \ }
    &\  \mathe^{\cN (0, \sigma^2) - \frac{1}{2} \sigma^2} \\
    & \ \text{with } \sigma^2 =
	\log \tfrac{1- \gamma\,\hat\beta^2}{1-\hat\beta^2}
\end{split}
  \end{equation}
and, furthermore, all positive moments converge.
\end{theorem}

Alternative proofs of the log-normality of the point-to-plane partition function in the sub-critical regime
were given in \cite{CC22} and, more recently, in \cite{CD24},
simplifying the original approach in \cite[Theorem~2.8]{CSZ17b}.
A key ingredient in all of these proofs
is the identification of suitable \emph{exponential time scales} which yield an approximate
factorisation of the partition function. 

Remarkably, \emph{a similar structure also emerges in the quasi-critical regime \eqref{eq:quasi-crit}}
when the partition function
is averaged on scales $\delta_N^\rho \sqrt{N}$, for any power~$\rho$,
with $\delta_N$ as in \eqref{eq:theta}.
This key fact is at the core of our proof of Theorem~\ref{th:log-normality} in Section~\ref{sec:log-normality}. 

\blue{\begin{remark}
Our proof of Theorem \ref{th:log-normality} first decomposes the partition function as a telescopic product 
$Z_N = \prod_{i=1}^{M} \frac{Z_{N_i}}{Z_{N_{i-1}}}$ for suitable time scales $N_i$, cf.\ \eqref{eq:decomp}, 
and then shows that the factors are approximately independent, which requires some control on the polymer end-point
distribution associated with each partition function $Z_{N_i}$ and is accomplished by turning off the noise in 
a small time window near $N_i$, see Section~\ref{sec:log-normality} for details. An alternative approach in 
\cite{CD24} directly approximates the partition function $Z_N$ in $L^2$ by a product of independent partition functions
$Z_{N_{i-1}, N_i}$ (see also \cite{GT26} for further extensions to the critical window). It is conceivable that this 
approach will give a simpler proof of Theorem \ref{th:log-normality} since it avoids the need to control 
the polymer end-point distributions. However, our control on the polymer end-point distribution and refined moment
asymptotics are of independent interest.
\end{remark}
}

\begin{remark}[Quasi-critical vs.\ sub-critical regime] \label{rem:analogy-sub-critical0}
Comparing \eqref{eq:log-normal} with \eqref{eq:log-normal-sub-critical}
for $\gamma = 0$, we can draw an analogy between the following two quantities:
\begin{itemize}
\item the quasi-critical partition function
$Z_N^{\mathrm{quasi\text{-}crit}} :=
Z_{N}^{\beta_N^{\mathrm{quasi\text{-}crit}}} \big( \, \cU_{B(0, \delta_N^\rho\sqrt{N})} \, \big)$
of size~$N$, averaged on the ball of radius $\delta_N^\rho \sqrt{N}$
centred at~$0$;

\item the sub-critical point-to-plane partition function $Z_L^{\text{\rm sub-crit}}
:= Z_{L}^{\beta_L^{\text{\rm sub-crit}}} (0)$
of size~$L$, with disorder strength $\hat\beta^2 = \frac{\rho}{1+\rho}$, that is $\frac{1}{1-\hat\beta^2} = 1+\rho$ (to match $\sigma^2$ in
\eqref{eq:log-normal} and \eqref{eq:log-normal-sub-critical}).
\end{itemize}
More precisely, if we divide space into squares of side length $\delta_N^\rho \sqrt{N}$ and
time into intervals of size $(\delta_N^\rho)^2 N$,
we can view the quasi-critical model $Z_N^{\mathrm{quasi\text{-}crit}}$ as effectively
a sub-critical model $Z_L^{\text{\rm sub-crit}}$
with rescaled time horizon $L \approx 1 / (\delta_N^\rho)^2$
and effective disorder strength $\hat\beta^2 = \frac{\rho}{1+\rho}$.

This analogy is made quantitative by our strategy of proof for Theorem~\ref{th:log-normality},
described in Section~\ref{sec:log-normality}. This suggests that, at a conceptual level,
other results that hold in the sub-critical
regime could be transferred to the quasi-critical regime
via this correspondence.

We stress, however, that the quasi-critical regime \eqref{eq:quasi-crit} presents
a fundamental technical challenge: unlike in the sub-critical regime,
the main contribution to the polynomial chaos
expansion of the partition function now comes from \emph{chaos of unbounded order}
(see the proofs of Proposition~\ref{th:2mb} and Theorem~\ref{th:vargen}).
As a consequence, many fundamental tools break down (e.g., hypercontractivity)
and novel arguments are required.
\end{remark}

\subsection{Moment bounds} \label{S:mom}

A key tool in our analysis are
\emph{moment bounds on the partition function} $Z_L^{\beta}(\varphi)$,
see \eqref{eq:Z} (we denote the system size by~$L$ in place of~$N$ for later convenience).
Such bounds, based on a functional operator approach, have been exploited
in several contexts, see \cite{GQT21,CSZ23a,LZ23,CCR23, CZ23, CZ24, CN25}.
We provide here a \emph{universal bound} of independent interest,
which applies to all regimes of $\beta$ mentioned so far (sub-critical, quasi-critical and critical)
and to general initial conditions $\varphi(\cdot)$ supported on sub-diffusive or diffusive scales $O(\sqrt{L})$.

\smallskip

We focus on initial conditions
which are \emph{probability mass functions} on $\Z^2$,\footnote{Since $\varphi \mapsto Z_L^\beta(\varphi)$ is linear,
any $\varphi \ge 0$ with $\sum_{x\in\Z^2} \varphi(x) < \infty$ can be normalised to a probability mass function.}
i.e.\
\begin{equation*}
	\varphi(\cdot) \ge 0 \,, \qquad
	\sum_{x\in\Z^2} \varphi(x) = 1 \,,
\end{equation*}
with finite \emph{mean-squared displacement from its center of mass}:
\begin{equation} \label{eq:bbD}
	\bbD[\varphi] := \sum_{z\in\Z^2} |z-m_\varphi|^2 \, \varphi(z) < \infty
	\qquad \text{with} \qquad
	m_\varphi := \sum_{z\in\Z^2} z \, \varphi(z) \,.
\end{equation}
We require two natural bounds on~$\varphi$.

\begin{itemize}
\item \emph{Exponential localisation on
at most diffusive scale:}
for some $\,\hat{t}\, > 0$, $\sfc_1 < \infty$
\begin{equation} \label{eq:concsqrtL}
	\exists z_0 \in \R^2 \colon \qquad
	\sum_{z\in\Z^2} \varphi(z ) \, \rme^{2\,\hat{t}\, \frac{|z - z_0|}{\sqrt{L}}}
	\le \sfc_1
\end{equation}
(the factor $2$ in the exponent is for later convenience).
This allows $\varphi(\cdot)$ to be localised on a diffusive or sub-diffusive scale, as it implies $\sqrt{\bbD[\varphi]} = O( \sqrt{L})$.

\item \emph{Local uniformity:} for some $\sfc_2 < \infty$
\begin{equation}\label{eq:locunif1}
	\big\| \varphi \big\|_{\ell^2}^2
	= \sum_{z\in\Z^2} \varphi(z)^2 = \E[\varphi(Z)]
	\le \frac{\sfc_2}{\bbD[\varphi]} \, ,
\end{equation}
where $Z$ is a random point in $\Z^2$ with law $\varphi$.
Since $\|\varphi\|_{\ell^2}^2 \le \|\varphi\|_{\ell^\infty} \sum_{z\in\Z^2} \varphi(z)
= \|\varphi\|_{\ell^\infty}$, a sufficient condition is
\begin{equation}\label{eq:locunif1sup}
	\big\| \varphi \big\|_{\ell^\infty}
	\le \frac{\sfc_2}{\bbD[\varphi]} \, ,
\end{equation}
which means that \emph{the peaks of~$\varphi$ are comparable to those of a uniform distribution}
(note that $\varphi$ puts most of its mass in a ball of radius $\sqrt{\bbD[\varphi]}$, by Chebyshev).
\end{itemize}

We do not restrict $\beta$ to any particular regime, but we consider partition functions with \emph{uniformly bounded variance}: for some $\sfc_3 < \infty$
\begin{equation}\label{eq:varub0}
	\bbvar[Z_{L}^\beta(\varphi)]
	\le \sfc_3 \,.
\end{equation}
We will show that, together with \eqref{eq:concsqrtL}, this implies that
$\beta$ lies within or below the critical regime as $L\to\infty$, see Lemma~\ref{th:sigmatheta}.

We are ready to state our general moment bound.

\begin{theorem}[General moment bound]\label{th:genmombou}
Given $h\in\N$ and $\,\hat{t}$, $\sfc_1$, $\sfc_2$,  $\sfc_3 \in (0,\infty)$,
there are constants $L_h, \mathfrak{C}_h  < \infty$
(depending also on $\,\hat{t}, \sfc_1, \sfc_2, \sfc_3$) such that
\begin{equation}\label{eq:genbound}
	\big| \bbE \big[ \big( Z_L^\beta(\varphi) - \bbE[Z_L^\beta(\varphi)] \big)^h \big] \big|
	\,\le\,
	\mathfrak{C}_h
	\bbvar[Z_{L}^\beta(\varphi)]^{\frac{h}{2}}
\end{equation}
uniformly for $\beta \ge 0$, $L\ge L_h$ and
probability mass functions $\varphi$ satisfying \eqref{eq:concsqrtL}, \eqref{eq:locunif1}, \eqref{eq:varub0}.
The bound \eqref{eq:genbound} still holds if, on the LHS, we replace $Z_L^\beta(\varphi)$ with its
restriction to any subset
of random walk paths in its definition \eqref{eq:Z}.
\end{theorem}

We prove Theorem~\ref{th:genmombou} in Section~\ref{sec:moments} in a strengthened form,
see Theorem~\ref{th:genmombou+}, where we
relax the assumption \eqref{eq:locunif1} and we consider partition functions
$Z_L^\beta(\varphi,\psi)$ averaged
at both endpoints (i.e.\ we allow for a ``final condition'' $\psi$, besides the initial condition~$\varphi$).

\begin{remark}[Beyond diffusive scales]
We prove Theorem~\ref{th:genmombou} under fairly general assumptions: the bounded variance condition \eqref{eq:varub0} is necessary, as we explain in Remark~\ref{rem:intermittency}, and the local uniformity assumption \eqref{eq:locunif1} is mild (and we further relax it in Section~\ref{sec:moments}).

Only the exponential localisation condition on at most diffusive scale \eqref{eq:concsqrtL} imposes some real restriction.
For instance, for initial conditions $\varphi(\cdot)$ localised at scale $\sqrt{N}$, one can consider system sizes $L = \epsilon N$ and prove a moment bound like \eqref{eq:genbound} uniformly in $\epsilon > 0$, see e.g.\ \cite[Proposition~2.3]{CCR23} in the quasi-critical regime, but this goes beyond the scope of Theorem~\ref{th:genmombou} because assumption \eqref{eq:concsqrtL} is not satisfied uniformly in $\epsilon > 0$.

We believe that our proof of Theorem~\ref{th:genmombou} could be extended in order to relax the localisation condition \eqref{eq:concsqrtL}, but we refrain from doing so in the present paper.
\end{remark}

\begin{remark}[Hypercontractivity]
The bound \eqref{eq:genbound} shows that,
under the assumptions of Theorem~\ref{th:genmombou},
a form of \emph{hypercontractivity} holds for the diffusively averaged partition function:
moments of order $h > 2$ are controlled by the $\frac{h}{2}$-power
of the second moment.

We point out that hypercontractivity is a general property of Wiener chaos and polynomial
chaos when the main contribution comes from chaos of \emph{bounded order} \cite{J97, MOO10}.
This is the case for the
directed polymer partition function only in the sub-critical regime \eqref{eq:sub-critical},
because in the quasi-critical and critical regimes \eqref{eq:quasi-crit} and \eqref{eq:betacrit}
\emph{the main contribution comes from chaos of unbounded order}.
The fact that the partition function still satisfies a form of hypercontractivity in these
latter regimes,
by Therorem~\ref{th:genmombou}, is highly non-trivial.
\end{remark}

\begin{remark}[Intermittency]\label{rem:intermittency}
The bounded variance assumption \eqref{eq:varub0} is crucial for Theorem~\ref{th:genmombou}.
In fact it is \emph{necessary} for the moment bound \eqref{eq:genbound} to hold as it is a general
fact that, for a sequence of nonnegative random variables $X_n$ with mean $1$ and diverging variance,
we have
\begin{align*}
\forall h >2 \colon \quad \bbE\big[X_N^h \big] \gg \bbE\big( X_N^2\big)^{h/2}.
\end{align*}
This can be seen by defining the size-biased law
$\dd\tilde\bbP_N : =X_N \, \dd\bbP$ and using Jensen's inequality:
\begin{align*}
\bbE\big[ X_N^h\big]
= \tilde\bbE\big[ X_N^{h-1}\big]
\geq \tilde\bbE\big[ X_N\big]^{h-1}
= \bbE\big[ X_N^{2}\big]^{h-1}
\gg \bbE\big[ X_N^{2}\big]^{\tfrac{h}{2}},
\end{align*}
since $h >2$ and the second moment
$\bbE[ X_N^{2}]$ diverges as $N\to\infty$.

In the case $X_N=Z_{L_N}^{\beta_N}(\varphi_N)$ with the support of $\varphi_N$ shrinking to
$0$ fast enough such that the variance diverges as $N\to\infty$, we actually expect a stronger intermittency
of the form
\begin{align*}
	\forall h \ge 3: \qquad
	\bbE \big[ Z_{L_N}^{\beta_N}(\varphi_N)^h \big]
	\geq \bbE[Z_{L_N}^{\beta_N}(\varphi_N)^2]^{h \choose 2} \,.
\end{align*}
Such results have been proved in the continuum setting of the Critical 2D SHF \cite{CSZ23b, LZ26}.
\end{remark}

\subsection{Extensions and related results} \label{S:ext}

Theorem~\ref{th:log-normality-SHF},
which deals with the regime $\theta \to -\infty$,
is proved by approximating the SHF
by the directed polymer partition functions in the
\emph{quasi-critical regime} \eqref{eq:quasi-crit} (see Section~\ref{sec:proof-SHF}).
This strategy is actually very general and leads us to
formulate the following ``meta-theorem''.

\begin{metatheorem} \label{th:meta}
Consider any ``reasonable'' statement about the SHF $\mathscr{Z}_t^\theta(\dd x)$
in the regime $\theta \to -\infty$. Such a statement holds if one can prove
the corresponding statement for the rescaled directed polymer partition functions
$Z_{tN}^\beta(\ev{x\sqrt{N}}) \, \dd x$ with $\beta = \beta_N^{\mathrm{quasi\text{-}crit}}$
in the quasi-critical regime \eqref{eq:quasi-crit}, for any sequence
$|\theta_N| \to \infty$ slowly enough.
\end{metatheorem}

By ``reasonable'' statement we mean some property of locally finite measures on $\R^2$
which is \emph{continuous with respect to the vague topology or the topology of $\cC^{-\epsilon}$}.
Indeed, the basic idea behind Claim~\ref{th:meta} is that the convergence in distribution
\eqref{eq:weak-conv} or \eqref{eq:strong-conv}, which we know to hold \emph{in the critical regime},
can also be applied to $\theta = \theta_N \to -\infty$ slowly enough, allowing us
to effectively transfer the statement from directed polymers to the SHF.
To lighten the exposition, we refrain from formulating a more precise result: we rather refer to
the proof of Theorem~\ref{th:log-normality-SHF} in Section~\ref{sec:proof-SHF}
for a concrete application of this idea.

\bigskip

For example, the quasi-critical regime \eqref{eq:quasi-crit} was recently investigated in \cite{CCR23}
for \emph{diffusive initial conditions}, such as $\cU_{B(0,\delta\sqrt{N})}$ for fixed $\delta > 0$. It was shown that the averaged
partition function concentrates around its mean:
\begin{equation*}
	Z_{tN}^{\beta_N^{\mathrm{quasi\text{-}crit}}}\!\big( \cU_{B(0,\delta\sqrt{N})} \big)
	\ \xrightarrow[\ N\to\infty \ ]{d} \
	1 \,\equiv\,
	\bbE \pig[ Z_{tN}^{\beta_N^{\mathrm{quasi\text{-}crit}}}\!\big( \cU_{B(0,\delta\sqrt{N})} \big) \pig]  \,,
\end{equation*}
its variance vanishes at rate $|\theta_N|^{-1}$ \cite[Proposition~2.1]{CCR23}, and \emph{Gaussian fluctuations} emerge
at the corresponding scale \cite[Theorem~1.1]{CCR23}:
\begin{equation}\label{eq:CCR}
	\sqrt{|\theta_N|} \,
	\pig\{ Z_{tN}^{\beta_N^{\mathrm{quasi\text{-}crit}}}\!\big( \cU_{B(0,\delta\sqrt{N})} \big)
	- 1 \pig\}
	\ \xrightarrow[\ N\to\infty \ ]{d} \
	\cN(0,a_\delta^2) \qquad \text{with} \quad 0 < a_\delta^2 <\infty \,.
\end{equation}
From this one can deduce a corresponding results for the SHF in the weak disorder limit $\theta\to -\infty$, see \cite[Theorem~1.2]{CCR23},
in the spirit of the metatheorem just stated:
\begin{equation}\label{eq:CCR2}
\sqrt{|\theta|} \big(\mathscr{Z}_{t}^\theta(\cU_{B(x,\delta)})-1\big) \ \xrightarrow[\ \theta\to-\infty \ ]{d} \
	\cN(0,a_\delta^2).
\end{equation}
\begin{remark}
We expect $\mathscr{Z}_{t}^\theta(\cU_{B(x,\delta)})$ to have Gaussian fluctuations, as $\theta\to-\infty$,
on all spatial scales $\delta = O(1)$ satisfying $\delta = \rme^{-o(|\theta|)}$, i.e.\ much larger than the scales $\delta_\theta^\rho$ on which log-normality arises in Theorem~\ref{th:log-normality-SHF}. Similarly, we expect the averaged partition function in
\eqref{eq:CCR} to have a Gaussian limit when averaged on spatial scales $\delta \sqrt{N}$ with $\delta_N = O(1)$ satisfying $\delta = \rme^{-o(|\theta_N|)}$, i.e.\ much larger than the scale $\delta_N^\rho$ appearing in Theorem~\ref{th:log-normality}.

We note that in the subcritical regime \eqref{eq:sub-critical}, Gaussian fluctuations for spatial averages on such mesoscopic spatial scales have been
established for the solution of the 2D KPZ equation in \cite{Tao24b} (the subcritical 2D SHE and polymer partition functions are expected to have the same fluctuations on the mesoscopic scale).
\end{remark}

\section{Basic tools}

In this section we collect some basic definitions and
tools that we use in the proof. In particular, we present some key
second moment computations on the partition function
in the quasi-critical regime \eqref{eq:quasi-crit}.
For simplicity, we will abbreviate $\beta_N^{\mathrm{quasi\text{-}crit}}$ by~$\beta_N$.

\subsection{Partition functions}
Given $A < B \in 2\mathbb{Z}$ and
a function $\varphi : \mathbb{Z}^2_{\even}
\rightarrow \mathbb{R}$, we denote by $Z_{(A, B]} (\varphi)$
the directed polymer partition function on the time
interval $(A, B]$ with initial condition $\varphi$ at time~$A$:
\begin{equation} \label{eq:ZAB}
  Z_{(A, B]}^{\beta} (\varphi) \assign \E \Big[ \mathe^{\mathcal{H}_{(A,
  B]}^{\beta}}  \,\Big|\, \nobracket S_A \sim \varphi \Big] \assign \sum_{x \in
  \mathbb{Z}^2_{\even}} \varphi (x) \, \E [\mathe^{\mathcal{H}_{(A,
  B]}^{\beta}}  | \nobracket S_A = x]
\end{equation}
where $\mathcal{H}_I^{\beta}$ is defined in \eqref{eq:H}.
Note that $\bbE[\rme^{\mathcal{H}^{\beta,\omega}_{I}(S)}] = 1$, and hence
\begin{equation*}
	\bbE[ Z_{(A, B]}^{\beta}(\varphi)] = \sum_{x\in\Z^2_\even} \varphi(x) \,.
\end{equation*}
In the special case $(A,B] = (0,L]$ and $\varphi = \ind_{\{x\}}$, we write for short
\begin{equation} \label{eq:ZN}
	Z_L^{\beta} (\varphi) := Z_{(0, L]}^{\beta} (\varphi)
	\qquad \text{and} \qquad
	Z_L^\beta(x) := Z_L^\beta(\ind_{\{x\}}) \,.
\end{equation}

\subsection{Random walk}

Recall the random walk transition kernel $q_n(\cdot)$ from \eqref{eq:rw}.
We give two versions of the local central limit theorem for the simple symmetric
random walk on $\Z^2$, see \cite[Theorems~2.3.5 and~2.3.11]{LL10}:
uniformly for $z\in\Z^2$ and $n\in 2\N$, as $n\to\infty$
\begin{equation} \label{eq:llt}
\begin{split}
	q_n(z) & = \pigl(\, g_{\frac{n}{2}}(z)
	+ O\big(\tfrac{1}{n^2}\big) \, \pigr) \, 2 \cdot \ind_{\Z^2_\even}(z)  \\
	& =  g_{\frac{n}{2}}(z) \;
	\rme^{O(\frac{1}{n}) + O(\frac{|z|^4}{n^3})} \; 2 \cdot \ind_{\Z^2_\even}(z)
	\, \ind_{q_n(z) > 0}
\end{split}
	\qquad \text{where} \quad g_t(x) := \frac{\rme^{-\frac{|x|^2}{2t}}}{2\pi t} \,,
\end{equation}
the factor $2 \cdot \ind_{\Z^2_\even}(z)$ is due to periodicity, while the time argument $\frac{n}{2}$ in the heat kernel
comes from the random walk covariance
$\E[S_n^{(i)} S_n^{(j)}] = \frac{n}{2} \ind_{i=j}$ for $i,j \in \{1,2\}$.
In particular, $q_{2n}(0) \sim \frac{1}{\pi} \cdot \frac{1}{n}$ as $n\to\infty$.
For later use, we fix $0 < \mathfrak{a}_- < \mathfrak{a}_+ < \infty$ such that
\begin{equation} \label{eq:barN}
	\frac{\mathfrak{a}_-}{n} \le q_{2n}(0) \le
	\frac{\mathfrak{a}_+}{n}
	\qquad \forall n \in \N \,.
\end{equation}

We generalize the expected replica overlap
$R_L$ from \eqref{eq:RL} by defining, for $z\in\Z^2_\even$,
\begin{equation}\label{eq:RL2}
	R_L(z) := \sum_{n=1}^L q_{2n}(z) \,,
\end{equation}
which is nothing but the random walk Green's function (on a bounded time interval).
We also introduce the corresponding quadratic form:
\begin{equation} \label{eq:Rfg}
	R_L(\varphi,\varphi) := \sum_{z,w \in \Z^2_\even} \varphi(z) \, R_L(z-w) \, \varphi(w) \,.
\end{equation}

By \eqref{eq:barN}, we can bound
$$
R_{L}(z) - R_{\lfloor \frac{1}{2}L \rfloor}(z) = \sum_{n=\lfloor \frac{1}{2}L \rfloor+1}^L q_{2n}(z)
\le \sum_{n=\lfloor \frac{1}{2}L \rfloor + 1}^L q_{2n}(0) \le \mathfrak{a}_+
$$
uniformly over $z\in\Z^2$. Therefore,
\begin{equation}\label{eq:doubling}
	\text{for any probability mass function $\varphi$:} \qquad
	R_{L}(\varphi,\varphi) - R_{\lfloor \frac{1}{2}L \rfloor}(\varphi,\varphi) \le \mathfrak{a}_+ \,.
\end{equation}

The continuum analogue of $R_L(\cdot)$ is the Green's function
$\cG(x) = \cG(|x|)$ given by
\begin{equation}\label{eq:Green}
	\cG(x) \,:=\, \int_0^1 g_t(x) \, \dd t
	\,=\, \int_0^1 \frac{\rme^{-\frac{|x|^2}{2t}}}{2\pi\, t} \, \dd t
	\,=\, \frac{1}{2\pi} \int_{|x|^2}^\infty \frac{\rme^{-\frac{r}{2}}}{r} \, \dd r \,.
\end{equation}
The following result compares $R_L(\cdot)$ and $\cG(\cdot)$.
The proof
is given in Appendix~\ref{sec:2mb}.

\begin{lemma}[Green's function]\label{th:Green}
Uniformly for $L\in\N$ and $z\in\Z^2$ we can write
\begin{equation} \label{eq:asR}
\begin{split}
	R_L(z)
	&\,=\, 2 \, \cG\Big( \tfrac{|z|+1}{\sqrt{L}} \Big) \, \ind_{z \in \Z^2_\even} \!+ O(1)
	\,=\, \tfrac{1}{\pi}  \log \pigl( 1 + \tfrac{L}{1+|z|^2} \pigr) \, \ind_{z \in \Z^2_\even} \!+ O(1) \,.
\end{split}
\end{equation}
Moreover, for any $t  \in (0,\infty)$, there is $c_t > 0$ such that
\begin{equation}\label{eq:intR}
	\begin{array}{cc}
	\text{uniformly for $L\in\N$, $z\in\Z^2$}\\
	\text{with } |z| \le t \sqrt{L} \wedge  L  \colon
	\end{array} \qquad
	R_L(z) \ge c_t \, \log \pigl( 1 + \tfrac{L}{1+|z|^2} \pigr) \, \ind_{z \in \Z^2_\even} \,.
\end{equation}
(The restriction $|z| \le L$ in \eqref{eq:intR} ensures that
$R_L(z) \ge q_{2L}(z) > 0$ for $z \in \Z^2_\even$.)
\end{lemma}

\subsection{Polynomial chaos expansion}

Let us introduce random variables
\begin{equation*}
	\xi_{n,x}^\beta := \frac{\rme^{\beta \omega(n,x) - \lambda(\beta)}}{\sigma_\beta},
\end{equation*}
which are i.i.d.\ with zero mean and unit variance, thanks to the definition \eqref{eq:sigmabeta}
of $\sigma_\beta$. We can represent the point-to-plane partition function $Z^\beta_L(x)$
as a \emph{polynomial chaos expansion}:
\begin{equation} \label{eq:polynomial}
	Z_L^\beta(x) = 1 + \sum_{k=1}^L \sigma_{\beta}^k
	\sumtwo{0 = n_0 < n_1 < \ldots < n_k \le L}{x_0 = x, \ x_1, \ldots, x_k \in \Z^2}
	\ \prod_{i=1}^k q_{n_i - n_{i-1}}(x_i - x_{i-1}) \, \xi_{n_i, x_i}^{\beta} \,.
\end{equation}
see e.g.\ \cite[eq.~(2.17)]{CSZ20}.
This follows from the definition \eqref{eq:Z}-\eqref{eq:H} by writing
\begin{equation*}
	\rme^{\cH^\beta_{(0,L]}} = \prod_{n=1}^L \prod_{x \in \Z^2}
	\rme^{\{\beta \omega(n,x) - \lambda(\beta)\} \ind_{\{S_n = x\}}}
	= \prod_{n=1}^L \prod_{x \in \Z^2}
	\pig\{ 1 + \sigma_\beta \, \xi_{n,x}^\beta \,
	\ind_{\{S_n = x\}} \pig\}
\end{equation*}
and then expanding the product.

\begin{remark}[Switching off some disorder] \label{rem:switch-off}
We will later consider partition functions where disorder is ``switched off'' in a time interval
$(A,B] \subseteq (0,L]$, meaning that the Hamiltonian $\cH^\beta_{(0,L]}$ is replaced by
$\cH^\beta_{(0,A]} + \cH^\beta_{(B,L]}$, see \eqref{eq:H}.
This amounts to setting $\xi_{n,x} = 0$ for all $n \in (A,B]$, which is equivalent to
restricting the polynomial chaos \eqref{eq:polynomial} to sequences of times
$n_1, \ldots, n_k$ which avoid the interval $(A,B]$.
\end{remark}

\subsection{Second moment of point-to-plane partition function}
Recalling \eqref{eq:sigmabeta}, we define a weighted renewal function
$U_{\beta}(\cdot)$ by setting
$U_{\beta}(0) := 1$ and for $n \ge 1$
\begin{equation} \label{eq:U}
\begin{split}
	U_{\beta}(n)
	& := \sum_{k \ge 1} (\sigma_{\beta}^2)^k
	\sumtwo{0 =: n_0 < n_1 < \cdots < n_k = n}{x_0 := 0, \
	x_1, \ldots, x_k \in \Z^2} \ \prod_{i = 1}^k q_{n_i - n_{i - 1}}(x_i - x_{i-1})^2 \\
	& = \sum_{k \ge 1} (\sigma_{\beta}^2)^k
	\sum_{0 =: n_0 < n_1 < \cdots < n_k = n} \ \prod_{i = 1}^k q_{2 (n_i - n_{i - 1})}(0) \,,
\end{split}
\end{equation}
where the last equality holds by the fact that $\sum_{z \in \Z^2} q_{n}(z)^2 = \sum_{z\in\Z^2}
	q_n(z) \, q_n(-z) = q_{2n}(0)$.

The quantity $U_{\beta}(\cdot)$ arises in the second moment of
the \emph{point-to-plane partition function}:
\begin{equation} \label{eq:barU}
	\bbE\big[ Z_L^{\beta}(0)^2 \big] = \overline{U}_{\beta}(L) :=
	\sum_{n=0}^L U_{\beta}(n) \,,
\end{equation}
which follows by the polynomial chaos
expansion \eqref{eq:polynomial} noting that terms indexed by distinct
space-time sequences $(n_1, x_1), \ldots, (n_k, x_k)$
are orthogonal in~$L^2$.

The second moment $\bbE\big[ Z_L^{\beta}(0)^2 \big]$ is uniformly bounded in~$L\leq N$ and
$N\in\N$ when $\beta=\beta_N$ lies in the sub-critical
regime \eqref{eq:sub-critical} \cite{CSZ19a}. The next result considers the quasi-critical regime \eqref{eq:quasi-crit}
and identifies how the second moment diverges as a function of $\theta_N$ in \eqref{eq:quasi-crit}.
The proof, based on renewal theory, is deferred to Appendix~\ref{sec:2mb}.

\begin{proposition}[Second moment of point-to-plane partition function]\label{th:2mb}
For $\beta = \beta_N$ in the quasi-critical regime \eqref{eq:quasi-crit},
uniformly over $L \in 2\N$ with $L \le N$, we have
\begin{equation}\label{eq:point-to-plane-2ndmom}
	\bbE\big[ Z_L^{\beta_N}(0)^2 \big] = \overline{U}_{\beta_N}(L)
	\sim \frac{1}{1 - \frac{R_L}{R_N}(1-\frac{|\theta_N|}{\log N})}
	\qquad \text{as } N\to\infty \,.
\end{equation}
\end{proposition}

\subsection{Variance of averaged partition functions}

We finally compute the variance of the \emph{averaged partition function} $Z_L^{\beta}(\varphi)
= \sum_{x\in\Z^2_\even} \varphi(x) \, Z_L^\beta(x)$. We write
\begin{equation*}
	\bbvar\big[ Z_L^{\beta}(\varphi) \big]
	= \sum_{x, x' \in \Z^2_\even} \varphi(x) \, \varphi(x') \,
	\bbcov\big[ Z_L^{\beta}(x), Z_L^{\beta}(x') \big] \,.
\end{equation*}
Plugging in the polynomial chaos expansion \eqref{eq:polynomial} and
renaming $n_1 = m$ and $n_k = n$,
by \eqref{eq:U}, we can write
\begin{equation*}
\begin{split}
	\bbcov\big[ Z_L^{\beta}(x), Z_L^{\beta}(x') \big]
	&= \sum_{0 < m \le n \le L} \,
	\sum_{x_1 \in \Z^2} \, q_{m}(x_1 - x) \, q_{m}(x_1 - x') \,
	\sigma_\beta^2 \, U_{\beta}(n-m) \\
	&= \sum_{0 < m \le L} \,
	q_{2m}(x - x') \,\, \sigma_\beta^2 \, \,
	\overline{U}_{\beta}(L-m)  \\
	& = \sum_{0 < m \le L} \,
	q_{2m}(x - x') \,\, \sigma_\beta^2 \, \,
	\bbE\big[ Z_{L-m}^{\beta}(0)^2 \big] \,,
\end{split}
\end{equation*}
where we applied \eqref{eq:barU}. In summary, introducing the shorthand
\begin{equation*}
	q_{2m}(\varphi, \varphi) :=
	\sum_{x,x' \in \Z^2_\even} \varphi(x) \, \varphi(x') \, q_{2m}(x-x') \,,
\end{equation*}
we have thus obtained the key formula
\begin{equation}\label{eq:varcomp}
	\bbvar\big[ Z_L^{\beta}(\varphi) \big]
	\,=\, \sum_{0 < m \le L} \, q_{2m}(\varphi, \varphi) \,\,
	\sigma_\beta^2 \, \,
	\bbE\big[ Z_{L-m}^{\beta}(0)^2 \big] \,.
\end{equation}
Since $L \mapsto \bbE\big[ Z_{L}^{\beta}(0)^2 \big] = \overline{U}_\beta(L)$ is increasing,
recalling the definition \eqref{eq:Rfg} of $R_L(\varphi,\varphi)$, we obtain the bounds
\begin{equation}\label{eq:boundsvar}
	R_{\lfloor \frac{1}{2}L \rfloor}(\varphi,\varphi) \, \,
	\sigma_\beta^2 \, \,
	\bbE\big[ Z_{\lfloor \frac{1}{2}L \rfloor}^{\beta}(0)^2 \big]
	\,\le\, \bbvar\big[ Z_L^{\beta}(\varphi) \big]
	\,\le\, R_{L}(\varphi,\varphi) \, \,
	\sigma_\beta^2 \, \,
	\bbE\big[ Z_{L}^{\beta}(0)^2 \big] \,.
\end{equation}

\begin{remark}
We can also rewrite \eqref{eq:varcomp} more explicitly as
\begin{equation}\label{eq:varexpl}
	\bbvar\big[ Z_L^{\beta}(\varphi) \big]
	\,=\, \sum_{k \ge 1} (\sigma_{\beta}^2)^k
	\sum_{0< n_1 < \cdots < n_k \le L} \ q_{2n_1}(\varphi, \varphi)
	\prod_{i = 2}^k q_{2 (n_i - n_{i - 1})}(0) \,.
\end{equation}
\end{remark}

We now compute the asymptotic behavior of
$\bbvar\big[ Z_L^{\beta}(\varphi) \big]$ for $\beta = \beta_N$ in the quasi-critical
regime \eqref{eq:quasi-crit}, allowing for general system size $L = L_N$
and initial condition $\varphi = \varphi_N$ (this will be essential for the proof of our results).

In Theorem~\ref{th:log-normality}, we consider the partition function $Z_{L_N}^{\beta_N}(\varphi_N)$
of size $L_N=N$ and initial conditions $\varphi_N$ averaged on balls of radius
$\delta_N^w \sqrt{N}$ for $w \in (0,\infty)$
(recall $\delta_N$ from \eqref{eq:theta}). Our next results computes
the variance of $Z_{L_N}^{\beta_N}(\varphi_N)$, showing that
\emph{it is bounded away from zero and infinity}
for general initial conditions $\varphi_N$ that are ``spread out'' on the scale $\delta_N^w \sqrt{N}$
and for general system sizes $L_N = N
\, (\delta_N^2)^{\ell + o(1)}$ with $\ell < w$.

\begin{theorem}[Variance of averaged partition functions]\label{th:vargen}
Let $\beta_N$ be in the quasi-critical regime \eqref{eq:quasi-crit}
for a given sequence $1 \ll |\theta_N| \ll \log N$.
Recall $\delta_N$ from \eqref{eq:theta}.

Let us fix two exponents $0 \le \ell \le w < \infty$.
For $N\in\N$ we consider:
\begin{itemize}
\item two sequences $L_N \in 2\N$ (system size),
$W_N \ge 0$ (scaling factor) such that, as $N\to\infty$,
\begin{gather} \label{eq:L}
	L_N = N \, (\delta_N^2)^{\ell+o(1)} = N \, \rme^{-\ell \, |\theta_N| + o(|\theta_N|)}
	\qquad \text{with } \ell \ge 0  \,, \\
	\label{eq:W}
	W_N = N \, (\delta_N^2)^{w + o(1)}
	= N \, \rme^{-w \, |\theta_N| + o(|\theta_N|)}
	\qquad \text{with } w \ge \ell \,;
\end{gather}
\item probability mass functions
$\varphi_N : \Z^2_\even \to [0,\infty)$
``spread out'' on scale $\sqrt{W_N}$ in the following sense
(recall \eqref{eq:asR}):
\begin{equation}\label{eq:phi-cond0}
	R_{L_N}(\varphi_N, \varphi_N)
	\,=\, \frac{1}{\pi} \, \log \frac{L_N}{W_N} + o(|\theta_N|)
	\,=\, \frac{w-\ell}{\pi} \, |\theta_N|   + o(|\theta_N|) \,.
\end{equation}
\end{itemize}
Then
\begin{equation}\label{eq:2nd-mom}
	\lim_{N\to\infty} \, \bbvar\big[ Z_{L_N}^{\beta_N}(\varphi_N) \big]
	= \frac{w-\ell}{1 + \ell} \,.
\end{equation}
Moreover, the convergence \eqref{eq:2nd-mom} holds uniformly
over system sizes $(L_N)_{N\in\N}$, scaling factors $(W_N)_{N\in\N}$
and initial conditions $(\varphi_N)_{N\in\N}$
for which \eqref{eq:L}, \eqref{eq:W}, \eqref{eq:phi-cond0} hold uniformly.
\end{theorem}

\begin{remark}[Sub-critical regime]\label{rem:subcr1}
Theorem~\ref{th:vargen} can also be applied to the sub-critical regime \eqref{eq:sub-critical}: it suffices
to take $|\theta_N| \sim (1-\hat\beta^2) \log N$ with $\hat\beta^2 \in (0,1)$ (cf.\ \eqref{eq:quasi-crit}
and \eqref{eq:sub-critical}), but the final result \eqref{eq:2nd-mom} must be updated as follows:
\begin{equation}\label{eq:2nd-momsub}
	\lim_{N\to\infty} \, \bbvar\big[ Z_{L_N}^{\beta_N}(\varphi_N) \big]
	= \frac{(w-\ell)\,\hat\beta^2}{1 + \ell\,\hat\beta^2} \,.
\end{equation}
The minor changes required in the proof are described in Remark~\ref{rem:subcrapp}.
\end{remark}

The proof of Theorem~\ref{th:vargen} is deferred to Appendix~\ref{sec:2mb}.
Condition \eqref{eq:phi-cond0} means intuitively that,
sampling two points $x, y$ independently from $\varphi_N$,
their distance $|x-y|$ is roughly order~$\sqrt{W_N}$.
This is made precise by the next results, also proved in Appendix~\ref{sec:2mb}.

\begin{proposition}[Equivalent condition for \eqref{eq:phi-cond0}]\label{th:initial}
Let $L_N$, $W_N$ be as in \eqref{eq:L}, \eqref{eq:W}.
Condition \eqref{eq:phi-cond0} for the
probability mass functions $\varphi_N$ is equivalent to
\begin{equation} \label{eq:phi-cond}
	\sum_{x, y \in \Z^2_\even} \!\!\!\! \varphi_N(x) \, \varphi_N(y) \,
	\log \Bigl( 1 + \tfrac{L_N}{1+|x-y|^2} \Bigr)
	= \log  \frac{L_N}{W_N} + o(|\theta_N|)
	\quad\ \text{as } N\to\infty \,.
\end{equation}
\end{proposition}

\begin{proposition}[Sufficient condition for \eqref{eq:phi-cond0}]\label{th:initial2}
Let $L_N$, $W_N$ be as in \eqref{eq:L}, \eqref{eq:W}.
For probability mass functions $\varphi_N$
to satisfy \eqref{eq:phi-cond0} (or, equivalently, \eqref{eq:phi-cond}) it suffices that they are
``mostly supported on a ball of radius $\sqrt{W_N}$
with atoms of size $O(1/W_N)$'' in the following sense:
there exist $z_N \in \Z^2$ and $0 \le t_N = o(\theta_N)$ such that
\begin{equation} \label{eq:conditions-phi}
	\sum_{|x - z_N| \le \sqrt{W_N \, \rme^{t_N}}}
	\varphi_N(x) = 1 - o(1)
	\qquad \text{and} \qquad
	\sup_{x\in\Z^2} \varphi_N(x) \le \frac{\rme^{t_N}}{W_N} \,,
\end{equation}
\end{proposition}

In particular, by
\eqref{eq:conditions-phi}, condition \eqref{eq:phi-cond0} is satisfied when $\varphi_N$
is the uniform distribution on a ball or when $\varphi_N$ is the
random walk transition kernel, see \eqref{eq:phi0} and \eqref{eq:rw}:
\begin{equation} \label{eq:ball-rw}
	\varphi_N = \cU_{B\big(0,\sqrt{W_N \, \rme^{t_N}}\,\big)} \quad\ \text{and} \quad\
	\varphi_N = q_{W_N \, \rme^{t_N}}
	\quad \text{with $t_N = o(|\theta_N|)$ satisfy (\ref{eq:phi-cond0}) } \,.
\end{equation}

We finally compute the variance of the partition function $Z_{Nt}^{\beta_N}(\varphi_N)$
when $\varphi_N$ is the uniform distribution
in the ball $B(0,\delta_N^\rho \sqrt{N})$, as in
Theorem~\ref{th:log-normality}.

\begin{corollary}\label{th:sanity}
Let $\beta_N$ be in the quasi-critical regime \eqref{eq:quasi-crit}
and recall $\delta_N$ from \eqref{eq:theta}.
For any $t > 0$, $x \in \R^2$ we have
\begin{equation}\label{eq:2nd-mom0}
	\forall \rho \in (0,\infty) \colon \qquad
	\lim_{N\to\infty} \, \bbvar\big[ Z_{Nt}^{\beta_N}(\cU_{B(0,\delta_N^\rho \sqrt{N})}) \big]
	= \rho \,.
\end{equation}
\end{corollary}

\begin{proof}
The initial condition $\varphi_N =  \cU_{B(0,\delta_N^\rho \sqrt{N})}$
fulfills \eqref{eq:conditions-phi} with $W_N := N \delta_N^{2\rho}$ (recall \eqref{eq:phi0}).
Since $W_N$ satisfies \eqref{eq:W} with $w=\rho$,
while $L_N := Nt$ satisfies \eqref{eq:L} with $\ell=0$,
the assumptions of Theorem~\ref{th:vargen} are verified and
we obtain \eqref{eq:2nd-mom0} from \eqref{eq:2nd-mom}.
\end{proof}

\begin{remark}
Since $\bbvar[\mathe^{\cN (0, \sigma^2) - \frac{1}{2} \sigma^2}] = \rme^{\sigma^2}-1$,
relation \eqref{eq:2nd-mom0} is consistent with \eqref{eq:log-normal}.
\end{remark}

\section{Proofs for the SHF}

\label{sec:proof-SHF}

In this section, we prove our main results for the SHF. More precisely, we prove
\begin{itemize}
\item Theorem~\ref{th:singularity-SHF} (singularity) in Subsection~\ref{sec:singularity-SHF};
\item Theorem~\ref{th:log-normality-SHF} (log-normality) in Subsection~\ref{sec:log-normality-SHF};
\item Theorem~\ref{th:long-time-SHF} (long-time behavior) in Subsection~\ref{sec:long-time-SHF};
\item Theorem~\ref{th:regularity-SHF} (regularity) in Subsection~\ref{sec:regularity-SHF},
alongside the corresponding
Theorem~\ref{th:regularity-DP} for directed polymers.
\end{itemize}
We first state a basic monotonicity result. For integrable $\varphi: \R^2 \to \R$, we will write
\begin{equation*}
	\mathscr{Z}_{t}^{\theta}(\varphi) := \int_{\R^2} \varphi(x) \,
	\mathscr{Z}_{t}^{\theta}(\dd x),
\end{equation*}
which is well-defined and has mean $\int_{\R^2} \varphi(x) \, \dd x$.

\begin{lemma}[Convex monotonicity for the SHF]\label{th:monotonicity-SHF}
Fix $t > 0$ and a positive integrable function $\varphi: \R^2 \to \R^+$.
The law of $\mathscr{Z}_{t}^\theta(\varphi)$
is \emph{increasing in $\theta$ w.r.t.\ the convex order}, i.e., for any
convex function $\Psi: \R \to \R$, we have
\begin{equation}\label{eq:convex-order}
	\bbE \big[ \Psi\big( \mathscr{Z}_{t}^{\theta'}\!(\varphi) \big) \big]
	\le \bbE \big[ \Psi\big( \mathscr{Z}_{t}^{\theta}(\varphi) \big) \big]
	\qquad \text{for } \theta' < \theta \,,
\end{equation}
and the reverse inequality holds for concave~$\Psi$.

In particular, fractional moments of the SHF are decreasing in $\theta$:
\begin{equation}\label{eq:fractional-moments}
	\forall \alpha \in (0,1) \colon \qquad
	\bbE \big[ \mathscr{Z}_{t}^{\theta'}\!(\varphi)^\alpha \big]
	\ge \bbE \big[ \mathscr{Z}_{t}^{\theta}(\varphi)^\alpha \big]
	\qquad \text{for } \theta' < \theta \,.
\end{equation}
\end{lemma}

\begin{proof}
It is enough to prove \eqref{eq:convex-order} when $\Psi(x) = O(x)$ grows at most linearly
as $x \to \infty$, since the general case follows by monotone convergence.
We may also assume that $\varphi$ is continuous and compactly supported,
because such functions are dense in~$L^1$.

By the weak convergence \eqref{eq:weak-conv},
it is enough to prove \eqref{eq:convex-order}
when the SHF $\mathscr{Z}_{t}^{\theta}$
is replaced by the rescaled directed polymer partition function
$Z_{tN}^{\beta_N^{\mathrm{crit}}}$,
because the limit $N\to\infty$ is justified by
uniform integrability (via boundedness in $L^2$).
More generally, we claim that
\begin{equation} \label{eq:convex-order-DP}
	\bbE \big[ \Psi\big( Z_{N}^{\beta'}(\varphi) \big) \big]
	\le \bbE \big[ \Psi\big( Z_{N}^{\beta}(\varphi) \big) \big]
	\qquad \text{for } \beta' < \beta \,,
\end{equation}
for any $N \in \N$ and for any positive integrable
$\varphi: \Z^2_\even \to \R^+$.

Relation \eqref{eq:convex-order-DP} is known to hold by the FKG inequality, see e.g.\ the proofs of
\cite[Proposition~3.1]{C17} or
\cite[Theorem~2.8]{CSZ17b}
where the arguments are carried out for fractional moments, but they hold in general.
\end{proof}

\subsection{Singularity of the SHF}
\label{sec:singularity-SHF}

We now prove Theorem~\ref{th:singularity-SHF} on the singularity of the SHF.
Let us first recall some general facts about
measures on the Euclidean space.

By the Lebesgue
Decomposition Theorem \cite[Theorem~3.8]{F99},
any $\sigma$-finite measure $\nu$ on $\R^2$
is the sum $\nu = \nu^{\mathrm{ac}} + \nu^{\mathrm{sing}}$
of an \emph{absolutely continuous part} $\nu^{\mathrm{ac}}(\dd x) = f(x) \, \dd x$
and a \emph{singular part} $\nu^{\mathrm{sing}}(\dd x) \perp \dd x$ which assigns all its mass
to a set of zero Lebesgue measure.

If furthermore $\nu$ is \emph{locally finite}, then
the density $f(x)$ of $\nu^{\mathrm{ac}}$
can be computed as follows:
denoting by $B(x,\delta) := \{ y \in \R^2 \colon |y-x| < \delta\}$ the Euclidean ball in $\R^2$,
we have
\begin{equation}\label{eq:density}
	f(x) = \lim_{\delta \downarrow 0} \frac{\nu(B(x,\delta))}{\pi \, \delta^2}
	\qquad \text{for Lebesgue a.e.\ $x\in\R^2$} \,,
\end{equation}
see \cite[Theorem~3.22]{F99}
(any locally finite measure is regular by \cite[Theorem~7.8]{F99}).

In particular, we summarise the following general result.

\begin{proposition}[Singularity of measures]\label{th:singularity}
Given a locally finite measure $\nu$ on $\R^2$, the limit in \eqref{eq:density}
exists for Lebesgue a.e.\ $x\in\R^2$ and recovers the density $f(x)$ of the absolutely
continuous part of~$\nu$. In particular, $\nu$
is singular with respect to the Lebesgue measure
if and only if the limit in \eqref{eq:density} vanishes
for Lebesgue a.e.\ $x\in\R^2$.
\end{proposition}

\begin{proof}[Proof of Theorem~\ref{th:singularity-SHF}]

Applying Proposition~\ref{th:singularity} to the SHF
$\nu(\dd x) = \mathscr{Z}_{t}^\theta(\dd x)$, we see that
if \eqref{eq:singularity-SHF} holds, then almost surely
the SHF is singular with respect to Lebesgue.

\smallskip

It remains to prove \eqref{eq:singularity-SHF},
which we deduce from \eqref{eq:log-normal-SHF}.
We denote by $(\Omega, \cA, \bbP)$ the probability
space on which the SHF is defined and we indicate explicitly
the dependence on $\omega \in \Omega$ by
$\mathscr{Z}_t^{\theta,\omega}(\dd x)$. Recalling \eqref{eq:averaged-SHF},
we rephrase \eqref{eq:singularity-SHF} as
\begin{equation}\label{eq:singularity-SHF2}
	\text{for a.e.\ $\omega \in \Omega$:} \quad\ \
	L(\omega,x) :=
	\liminf_{n \to \infty} \ \mathscr{Z}_{t}^{\theta,\omega} \big(
	\cU_{B(x,\delta_n)} \big) = 0
	\quad \text{for Lebesgue a.e.\ $x\in\R^2$} \,,
\end{equation}
where we have fixed (arbitrarily) $\delta_n := \frac{1}{n}$
and we have replaced $\lim$ by $\liminf$, in order to obtain a
measurable function $L(\omega,x) \in [0,\infty]$ defined for all $\omega \in \Omega$ and $x \in \R^2$.
We stress that the limit in \eqref{eq:singularity-SHF} \emph{exists
as $\delta \downarrow 0$ for Lebesgue a.e.\ $x\in\R^2$},
see Proposition~\ref{th:singularity}, hence it must coincide
with $L(\omega,x)$ for a.e.\ $\omega \in \Omega$ and for Lebesgue a.e.\ $x\in\R^2$.

To complete the proof,
we need to show that, for a.e.\ $\omega \in \Omega$, we have $L(\omega,x) = 0$
for Lebesgue a.e.\ $x\in\R^2$, or equivalently
$\bbE[ \int_{\R^2} L(\omega,x)^\alpha \, \dd x ] = 0$ for
any fixed $\alpha \in (0,1)$ (recall that $L(\omega,x) \ge 0$).
By Fubini's theorem, it is enough to show that for all $x\in\R^2$
we have $\bbE[L(\omega,x)^\alpha] = 0$. By Fatou's Lemma
\begin{equation} \label{eq:boLalpha}
	\bbE[ L(\omega,x)^\alpha ] \le
	\liminf_{n \to \infty} \;
	\bbE\big[\mathscr{Z}_{t}^{\theta,\omega} \big(
	\cU_{B(x,\delta_n)} \big)^\alpha \big]
\end{equation}
and it remains to estimate the RHS.
To this end, we exploit the \emph{monotonicity of fractional moments}
\eqref{eq:fractional-moments} by replacing $\theta$ with $\theta'_\delta \downarrow -\infty$
and applying the \emph{log-normality} \eqref{eq:log-normal-SHF}.
Let us fix a parameter $\rho \in (0,\infty)$.
\begin{itemize}
\item We rewrite the log-normality \eqref{eq:log-normal-SHF} by
renaming $\delta_\theta^\rho = \rme^{\frac{1}{2} \rho \, \theta}$ as $\delta$,
i.e.\ expressing  $\theta = -\frac{1}{\rho} \log \frac{1}{\delta^{2}}$ as a function of~$\delta$, so that
\eqref{eq:log-normal-SHF} becomes
\begin{equation}\label{eq:log-normal-SHF2}
	\mathscr{Z}_{t}^{-\frac{1}{\rho} \log \frac{1}{\delta^{2}}}\big(
	\cU_{B(x,\delta)} \big)
	\ \xrightarrow[\ \delta \downarrow 0 \ ]{d} \
	\rme^{\cN(0,\sigma^2) - \frac{1}{2}\sigma^2}
	\qquad \text{with} \quad \sigma^2 = \log(1+\rho) \,.
\end{equation}
We note that this weak convergence
also implies convergence of fractional moments,
because the LHS of \eqref{eq:log-normal-SHF2} is bounded in $L^1$
(recall that $\bbE[\mathscr{Z}_{t}^\theta(\dd x)] = \dd x$).

\item If we set $\theta'_\delta := -\frac{1}{\rho} \log \frac{1}{\delta^{2}} \to -\infty$
as $\delta \downarrow 0$,
then we can apply the monotonicity
of fractional moments \eqref{eq:fractional-moments} with $\varphi = \cU_{B(x,\delta)}$
to estimate
the RHS of \eqref{eq:boLalpha}.
\end{itemize}
Overall, we obtain for any fixed $\alpha \in (0,1)$
\begin{equation*}
	\forall \rho \in (0,\infty): \qquad
	\bbE[ L(\omega,x)^\alpha ] \le
	\bbE \Big[ \pig( \rme^{\cN(0,\sigma^2) - \frac{1}{2}\sigma^2} \pig)^\alpha \Big]
	= \rme^{\frac{1}{2} (\alpha^2 - \alpha) \sigma^2 }
	= \frac{1}{(1+\rho)^{\frac{\alpha(1-\alpha)}{2}}} \,,
\end{equation*}
where in the last equality we plugged in the value of $\sigma = \log (1+\rho)$
from \eqref{eq:log-normal-SHF2}. Since $\alpha(1-\alpha) > 0$
for $\alpha \in (0,1)$, letting $\rho \to \infty$ we finally obtain
$\bbE[ L(\omega,x)^\alpha ] = 0$.
\end{proof}

\subsection{Log-normality of the SHF}
\label{sec:log-normality-SHF}

We prove Theorem~\ref{th:log-normality-SHF} on the log-normality of the SHF.
More precisely
we deduce it from the corresponding result for directed polymers, see Theorem~\ref{th:log-normality}
(which we prove in Section~\ref{sec:log-normality}).

To this end, recalling the uniform distribution on continuum and discrete balls, see
\eqref{eq:uniform} and \eqref{eq:phi0},
from the convergence in distribution \eqref{eq:weak-conv}, we obtain
\begin{equation}\label{eq:scaling}
	\forall x \in \R^2, \ \delta > 0 \colon \qquad
	Z_{tN}^{\beta_N^{\mathrm{crit}}}\!\pig(  \, \cU_{B(x\sqrt{N},\delta\sqrt{N})} \, \pig)
	\ \xrightarrow[\ N\to\infty \ ]{d} \
	\mathscr{Z}_{t}^\theta \big(\cU_{B(x,\delta)}\big) \,.
\end{equation}
Strictly speaking we cannot plug $\cU_{B(x,\delta)}(\cdot)$ as a
test function into \eqref{eq:weak-conv}, because it is not continuous. However, for any $\epsilon > 0$,
we can approximate $\varphi_\epsilon(\cdot) \le \cU_{B(x,\delta)}(\cdot) \le \psi_\epsilon(\cdot)$ with
 continuos functions  $\varphi_\epsilon, \psi_\epsilon$ such that
$0 \le \psi_\epsilon(\cdot) - \varphi_\epsilon(\cdot) \le \ind_{A_\epsilon}(\cdot)$, where
we define the annulus
$A_\epsilon := B(x,\delta+\epsilon) \setminus B(x,\delta-\epsilon)$.
Replacing $\cU_{B(x,\delta)}(\cdot)$ by $\varphi_\epsilon$ or $\psi_\epsilon$ in \eqref{eq:weak-conv},
we commit an error in $L^1$ which is $O(\epsilon)$, i.e.\
the Lebesgue measure of $A_\epsilon$. This justifies \eqref{eq:scaling}.

\begin{proof}[Proof of Theorem~\ref{th:log-normality-SHF}]
We fix $t > 0$, $x\in\R^2$,
$\rho \in (0,\infty)$ and define $\delta_\theta := \rme^{\frac{1}{2} \theta}$
as in \eqref{eq:deltatheta}. It suffices to prove the convergence in distribution \eqref{eq:log-normal-SHF}
when $\theta$ ranges in an arbitrary negative sequence
$\theta_k = -|\theta_k| \to -\infty$, which we fix henceforth.
Introducing the shorthands
\begin{equation*}
	Y_k := \mathscr{Z}_{t}^{\theta_k}\big( \cU_{B(x,\delta^\rho_{\theta_k})}
	\big) \,, \qquad
	Y := \rme^{\cN(0,\sigma^2) - \frac{1}{2}\sigma^2}
	\quad \text{with} \quad \sigma^2 = \log(1+\rho) \,,
\end{equation*}
we need to show that $Y_k \to Y$ in distribution.
It suffices to fix any bounded and continuous function $\Phi: \R \to \R$ and to prove that
\begin{equation} \label{eq:to-prove-log-norm}
	\lim_{k\to\infty} \bbE[\Phi(Y_k)] = \bbE[\Phi(Y)] \,.
\end{equation}

The idea is to approximate the SHF with the directed polymer partition function.
Recalling the convergence in distribution \eqref{eq:scaling}
in the critical regime \eqref{eq:betacrit}, we abbreviate
\begin{equation} \label{eq:Wdef}
	W_{N,k} := Z_{tN}^{\beta_N^{\mathrm{crit}}(k)}\!\pig(  \, \cU_{B(x\sqrt{N},
	\delta^\rho_{\theta_k}\sqrt{N})} \, \pig)
	\qquad \text{where} \quad
	\sigma_{\!\beta_N^{\mathrm{crit}}(k)}^2 = \frac{1}{R_N} \left( 1 +
	\frac{\theta_k}{\log N} \right) \,.
\end{equation}
For fixed $k\in\N$, we have $W_{N,k} \to Y_k$ in distribution as $N\to\infty$. Therefore
we can choose $N = N_k$ large enough so that $\bbE[\Phi(W_{N_k,k})]$ is close to
$\bbE[\Phi(Y_k)]$. Let us define
\begin{equation*}
\begin{split}
	N_1 &:= \min\big\{N \in \N \colon \ |\bbE[\Phi(W_{N,1})] - \bbE[\Phi(Y_1)]| \le 1 \big\} \,, \\
	N_k &:= \min\big\{N > N_{k-1} \colon \
	N \ge \rme^{k |\theta_k|} \ \text{ and } \ |\bbE[\Phi(W_{N,k})] - \bbE[\Phi(Y_k)]|
	\le \tfrac{1}{k} \big\} \,,
\end{split}
\end{equation*}
so that  we have by construction $N_1 < N_2 < \ldots$ and, for every $k\in\N$,
\begin{equation} \label{eq:Nk}
	N_k \ge \rme^{k |\theta_k|} \,, \qquad
	|\bbE[\Phi(W_{N_k,k})] - \bbE[\Phi(Y_k)]| \le \tfrac{1}{k} \,.
\end{equation}
By the triangle inequality, our goal \eqref{eq:to-prove-log-norm} holds if we show that
\begin{equation} \label{eq:to-prove-log-norm2}
	\lim_{k\to\infty} \bbE[\Phi(W_{N_k,k})] = \bbE[\Phi(Y)] \,.
\end{equation}

We claim that this holds by \eqref{eq:log-normal}, because \emph{the directed polymer
partition functions $W_{N_k,k}$ from \eqref{eq:Wdef}
satisfy the assumptions of Theorem~\ref{th:log-normality}}.
Note indeed that:
\begin{itemize}
\item the sequence $\beta_N^{\mathrm{crit}}(k)$
for $N = N_k$
is in the \emph{quasi-critical regime} \eqref{eq:quasi-crit}, since by \eqref{eq:Wdef}
\begin{equation*}
	\sigma_{\!\beta_{N_k}^{\mathrm{crit}}(k)}^2 = \frac{1}{R_{N_k}} \left( 1 +
	\frac{\theta_k}{\log N_k} \right) = \frac{1}{R_{N_k}} \left( 1 -
	\frac{|\theta_k|}{\log N_k} \right) \qquad \text{with} \quad
	1 \ll |\theta_k| \ll \log N_k \,,
\end{equation*}
where $|\theta_k| \ll \log N_k$ holds by the first inequality in \eqref{eq:Nk};

\item the initial condition $\cU_{B(x\sqrt{N}, \delta^\rho_{\theta_k}\sqrt{N})}$
for $N = N_k$ satisfies the requirement that $\delta_{\theta_k} \downarrow 0$
at rate \eqref{eq:theta}, because by definition
$\delta^\rho_{\theta_k} := \rme^{\frac{1}{2} \rho \, \theta_k} = \rme^{-\frac{1}{2} \rho \, |\theta_k|}$
as in \eqref{eq:deltatheta}.
\end{itemize}
We can thus apply Theorem~\ref{th:log-normality} to $(W_{N_k,k})_{k\in\N}$:
relation \eqref{eq:log-normal}
along the subsequence $N=N_k$ yields directly \eqref{eq:to-prove-log-norm2}
and completes the proof.
\end{proof}

\begin{remark}
The strategy in the proof of Theorem~\ref{th:log-normality-SHF} is very general and it
shows that the convergence in distribution \eqref{eq:weak-conv} or \eqref{eq:strong-conv},
which are proved for each $\theta$ in the critical regime \eqref{eq:betacrit},
can effectively be transferred to the quasi critical regime \eqref{eq:quasi-crit},
provided we take $\theta_N \to -\infty$ slow enough.
This naturally leads to Metatheorem~\ref{th:meta}.
\end{remark}

\subsection{Long-time behaviour of the SHF}
\label{sec:long-time-SHF}

We prove Theorem~\ref{th:long-time-SHF} on the long-time behaviour of the SHF.
This is a corollary of Theorem~\ref{th:singularity-SHF} on
the singularity of the SHF, together with the following scale-covariance property,
proved in \cite[Theorem~1.2]{CSZ23a}:
\begin{equation}\label{eq:scale-covariance}
	\forall \sfa > 0 \colon \qquad
	\mathscr{Z}_{\sfa t}^{\theta}\big( \cU_{B(0, \sqrt{\sfa} \, R)} \big)
	\overset{d}{=} \mathscr{Z}_{t}^{\theta + \log \sfa}\big( \cU_{B(0, R)} \big) \,,
\end{equation}
which holds for any $t,R \in (0, \infty)$ and $\theta \in \R$.

\begin{proof}[Proof of Theorem~\ref{th:long-time-SHF}]
To prove \eqref{eq:long-time-SHF}, it suffices to show that
\begin{equation}\label{eq:long-time-SHF2}
	\forall R < \infty \colon \qquad
	\mathscr{Z}_{t}^\theta\big( \cU_{B(0,R)} \big)
	\ \xrightarrow[\ t \to \infty \ ]{d} \
	0 \,,
\end{equation}
which follows if we show that for some fixed $\alpha \in (0,1)$ the fractional moment vanishes:
\begin{equation}\label{eq:long-time-SHF3}
	\forall R < \infty \colon \qquad
	\lim_{t\to\infty} \ \bbE \big[ \mathscr{Z}_{t}^\theta\big(  \cU_{B(0,R)} \big)^\alpha \big]
	\,=\, 0 \,.
\end{equation}

Exploiting first the scaling relation \eqref{eq:scale-covariance} with $\theta$ replaced
by $\theta - \log \sfa$ and $\sfa = t^{-1}$,
and then the monotonicity of fractional moments \eqref{eq:fractional-moments}, we obtain
for all $t \ge 1$
\begin{equation*}
	\bbE \big[ \mathscr{Z}_{t}^\theta\big(  \cU_{B(0,R)} \big)^\alpha \big]
	= \bbE \pigl[
	\mathscr{Z}_{1}^{\theta + \log t}\pigl( \cU_{B(0,R/\sqrt{t})} \pigr)^\alpha \pigr]
	\le \bbE \pigl[ \mathscr{Z}_{1}^{\theta}\pigl( \cU_{B(0,R/\sqrt{t})} \pigr)^\alpha \pigr] \,.
\end{equation*}
Applying \eqref{eq:singularity-SHF} with $\delta = \frac{R}{\sqrt{t}}$,
we see that $A_t := \mathscr{Z}_{1}^{\theta}\pigl( \cU_{B(0,R/\sqrt{t})} \pigr) \to 0$
in distribution as $t \to \infty$.
The random variables $(A_t)_{t \ge 1}$ are bounded in $L^1$,
because $\bbE[A_t] = 1$, hence by uniform integrability
we obtain $\bbE[A_t^\alpha] \to 0$ for any $\alpha \in (0,1)$, which completes the proof.
\end{proof}

\subsection{Improved convergence and regularity of the SHF}
\label{sec:regularity-SHF}

 Theorem~\ref{th:regularity-DP} entails that, almost surely,
$\mathscr{Z}_{t}^\theta \in \mathcal{C}^{-\epsilon}$ for every $\epsilon > 0$,
hence Theorem~\ref{th:regularity-SHF} follows
(delta measures $\delta_x$ in $\R^d$ are in $\cC^{\gamma}$
only for $\gamma \le -d$, hence $\mathscr{Z}_{t}^\theta$ is non-atomic).

\smallskip

The rest of this subsection is devoted to the proof of Theorem~\ref{th:regularity-DP}.
We first recall the definition of negative H\"older spaces
(see \cite[Section~2]{FM17} or \cite[Section~12]{CZ20} for more details).
Let us introduce some notation in any dimension $d\in\N$.
\begin{itemize}
\item Let $C^\infty_c$ denote the family of
smooth and compactly supported functions $\varphi: \R^d \to \R$.

\item For $r \in \N_0 = \{0,1,2,\ldots\}$,
let $C^r_c$ denote the family of
compactly supported functions
of class $C^r$, for which we define
\begin{equation*}
	\| \varphi \|_{C^r} := \max_{r_1, \ldots, r_d \in \N_0 \colon r_1 + \ldots + r_d \le r}
	\| \partial_1^{r_1} \cdots \partial_d^{r_d} \varphi \|_\infty \,.
\end{equation*}

\item Let $\mathscr{B}^r$ denote the family of functions
$\varphi \in C^\infty_c$ supported on $\overline{B(0,1)}$
with $\| \varphi \|_{C^r} \le 1$.

\item Given a function $\varphi: \R^d \to \R$, we denote by $\varphi_x^\lambda$
its $\lambda$-scaled version centred at~$x$:
\begin{equation*}
	\varphi_x^\lambda(\cdot) := \lambda^{-d} \, \varphi\big(\lambda^{-1}(\,\cdot - x)\big)
	\qquad \text{for } x \in \R^d, \ \lambda > 0 \,.
\end{equation*}
\end{itemize}

\begin{definition}[Negative H\"older spaces]\label{def:Holder}
Given $\gamma < 0$, the negative H\"older space $\mathcal{C}^{\gamma}$
is the family of  linear functionals $T: C^\infty_c \to \R$ with the following
property: for any $K \in \N$, there is a constant $c_K < \infty$ such that
\begin{equation} \label{eq:neg-property}
	|T(\varphi_x^\lambda)| \le c_K \, \lambda^\gamma \qquad
	\forall x \in B(0,K), \ \lambda \in (0,1], \ \varphi \in \mathscr{B}^r ,
\end{equation}
where $r = r(\gamma) := \lfloor -\gamma + 1\rfloor$
(any integer $r > -\gamma$ would yield an equivalent definition).
\end{definition}

Any distribution $T \in \cC^\gamma$
can be canonically extended (by continuity) from $C^\infty_c$ to $C^r_c$,
for any integer $r > -\gamma$,
hence we can consider $T(\varphi)$ for $\varphi \in C^r_c$.
In order to prove that $T \in \cC^\gamma$,
it turns out that it is enough to check property \eqref{eq:neg-property}
for a \emph{finite family of $2^d$ well-chosen
test functions $\varphi \in C^r_c$}:
a so-called wavelet basis \cite[Section~2]{FM17}, which we denote by
\begin{equation} \label{eq:wavelets}
	\phi \quad \text{and} \quad \big\{\psi^{(i)}\big\}_{1 \le i < 2^d} \quad
	\text{(which satisfy $\textstyle\int_{\R^d} \psi^{(i)}(x) \, \dd x = 0$)} \,.
\end{equation}
(The details of such functions are immaterial for our goals.)

This yields a convenient criterion for a sequence $(\mathscr{T}_N^\omega)_{N\in\N}$
of random distributions
to be tight in
the H\"older space $\mathcal{C}^{\gamma}$ with $\gamma < 0$
(in the spirit of the classical Kolmogorov moment criterion for $\gamma > 0$).
The following is a special case of \cite[Theorem~2.30]{FM17}.

\begin{theorem}[Tightness criterion for negative H\"older spaces]\label{th:Kolmogorov}
Fix $\gamma < 0$ and  an integer $r > -\gamma$.
Let $\{\phi, \ \psi^{(i)} \colon 1\le i < 2^d\}$ be
a $C^r_c$ wavelet basis in~$\R^d$, see \eqref{eq:wavelets}.

Let $(\mathscr{T}_N^\omega)_{N\in\N}$ be a sequence of random linear forms
on $C^r_c$, that is, for every $\omega$ in a probability space $(\Omega,\cA, \bbP)$
and every $N\in\N$,
we have a linear functional $\mathscr{T}_N^\omega : C^r_c \to \R$,
such that $\omega \mapsto \mathscr{T}_N^\omega(\varphi)$ is a random variable
for every $\varphi \in C^r_c$.

Assume that for some $p \in [1,\infty)$ and $C < \infty$ the following bounds hold:
\begin{equation}\label{eq:bounds-Holder}
	\forall N \in \N, \  x \in \R^d \colon \qquad
	\begin{cases}
	\rule{0pt}{1.2em}\displaystyle \bbE \pigl[
	\big|\mathscr{T}_N^\omega(\phi(\,\cdot - x)) \big|^p \pigr]^{1/p} \le C \,, & \\
	\rule[-.3em]{0pt}{1.8em}\displaystyle\bbE \pigl[ \pigl|\mathscr{T}_N^\omega
	\big( (\psi^{(i)})_x^\lambda \big) \pigr|^p \pigr]^{1/p} \le C
	\, \lambda^\gamma
	&  \ \ \forall \lambda \in (0,1], \ 1 \le i < 2^d \,.
	\end{cases}
\end{equation}
Then $(\mathscr{T}_N^\omega)_{N\in\N}$
is a tight sequence of random variables taking values in the space $\cC^{\gamma'}$
for any $\gamma' < \gamma - \frac{d}{p}$.
\end{theorem}


\begin{remark}[Topology of H\"older spaces]
Given $\gamma < 0$ and any distribution $T \in \cC^\gamma$, let us denote
$\ev{T}_{K,\gamma} = c_K$ the best constant
in the inequality \eqref{eq:neg-property}. Defining the distance
\begin{equation*}
	d_{\cC^\gamma}(T,T') := \sum_{K \in \N} 2^{-K} \,
	\frac{\ev{T-T'}_{K,\gamma}}{1+\ev{T-T'}_{K,\gamma}} \,,
\end{equation*}
we have that $\cC^\gamma$ is a complete metric space, but it is \emph{not separable},
see \cite{FM17}.

To ensure separability, one can define $\cC_0^\gamma$ as the
closure of smooth compactly supported functions
$C^\infty_c$ under the distance $d_{\cC^\gamma}$.
One has the strict inclusion $\cC_0^\gamma \subset \cC^\gamma$,
however for any $\tilde\gamma > \gamma$ one can sandwich
$\cC^{\tilde\gamma} \subset \cC_0^\gamma \subset \cC^\gamma$
(so the difference is ``small'' in a sense).

The results in \cite{FM17} are formulated for the separable spaces $\cC_0^\gamma$
(called $\cC^\gamma$ in that paper). However, the tightness criterion in Theorem~\ref{th:Kolmogorov}
applies also to the usual spaces $\cC^\gamma$,
because $\cC_0^\gamma$ is a closed subset of $\cC^\gamma$, hence
compact sets in $\cC_0^\gamma$ are also compact in $\cC^\gamma$.
\end{remark}

\smallskip

We now turn to the proof of Theorem~\ref{th:regularity-DP}, which is based on
the following moment bounds.
Recall the initial conditions $\cU_{B(z,R)}(\cdot)$ from \eqref{eq:phi0}
and the convergence in distribution \eqref{eq:scaling}.

\begin{proposition}[Moment bounds]\label{th:moments}
Fix $t >0$ and $\theta\in\R$. Let $\beta_N^{\mathrm{crit}}$ be
in the critical regime \eqref{eq:betacrit}.
For any $h \in 2\N$, $\epsilon > 0$, $\delta_0 < \infty$,
there is a constant $\mathtt{C} = \mathtt{C}_{p,\epsilon,\delta_0}^{t,\theta} < \infty$ such that
\begin{equation}\label{eq:moments-DP}
	\sup_{N\in\N} \;
	\bbE \pigl[ Z_{tN}^{\beta_N^{\mathrm{crit}}}\!\pig(  \, \cU_{B(x\sqrt{N},\delta\sqrt{N})} \,
	\pigr)^h \pigr]^{1/h}
	\le \mathtt{C} \, \delta^{-\epsilon}
	\qquad \forall x\in\R^2, \ \delta \in (0,\delta_0) \,,
\end{equation}
and hence
\begin{equation}\label{eq:moments-SHF}
	\bbE \pigl[ \mathscr{Z}_{t}^\theta\big( \cU_{B(x,\delta)}\big)^h \pigr]^{1/h}
	\le \mathtt{C} \, \delta^{-\epsilon}
	\qquad \forall x\in\R^2, \ \delta \in (0,\delta_0) \,.
\end{equation}
\end{proposition}

\begin{proof}
In view of \eqref{eq:scaling},
the bound \eqref{eq:moments-SHF} follows by \eqref{eq:moments-DP}
and Fatou's Lemma.

To prove the bound \eqref{eq:moments-DP}, we
apply equation~ (6.1) in \cite[Theorem~6.1]{CSZ23a} for $\psi \equiv 1$
(we just exchange $N, \tilde N$):
given  $h\in\N$ and $1 < p < \infty$,
there is a constant $\mathsf{C}' = \mathsf{C}'(\theta,h,p) < \infty$ such that,
uniformly in large $N, \tilde N\in\N$ and integrable $\varphi: \R^2 \to \R$, we have
\begin{equation*}
	\bbE \pigl[ \pigl( Z_{\tilde N}^{\beta_N^{\mathrm{crit}}}\!\big(  \varphi \big)
	- 1 \pigr)^h \pigr]^{\frac{1}{h}} \le \bigg(\frac{\mathsf{C}'}{\log (1+
	N/\tilde N)}\bigg)^{\frac{1}{h}} \,
	\bigg\| \frac{\varphi}{w}\bigg\|_p \, \big\| w \big\|_q \,,
\end{equation*}
where $q$ is the dual of~$p$ (i.e. $\frac{1}{p}+\frac{1}{q} = 1$)
and $w(\cdot) := \rme^{-|\cdot|}$ is a weight function.
In particular, for $\tilde{N} = tN$ and
$\varphi(\cdot) := \frac{1}{\pi \delta^2} \ind_{B(x,\delta)}(\cdot)$, we obtain
\begin{equation*}
	\bbE \pigl[ \pigl( Z_{tN}^{\beta_N^{\mathrm{crit}}}\!\pig(  \, \cU_{B(x\sqrt{N},\delta\sqrt{N})} \, \pig)
	- 1 \pigr)^h \pigr]^{\frac{1}{h}}
	\le \bigg(\frac{\mathsf{C}'}{\log (1+\frac{1}{t})}\bigg)^{\frac{1}{h}} \,
	\bigg(  \int\limits_{B(x,\delta)} \frac{\rme^{p|y|}}{(\pi \delta^2)^p} \, \dd y \bigg)^{\frac{1}{p}} \,
	\bigg(  \int\limits_{\R^2} \rme^{-q|y|} \, \dd y \bigg)^{\frac{1}{q}} .
\end{equation*}
The first integral is bounded by $\rme^{\delta_0} \, (\pi \delta^2)^{\frac{1}{p}-1}$
while the second integral equals $(2\pi \, q^{-2})^{1/q}$ which is uniformly bounded
for $1 < q < \infty$. For a suitable constant $\mathsf{C}''$, we then obtain
\begin{equation*}
	\bbE \pigl[ \pigl( Z_{tN}^{\beta_N^{\mathrm{crit}}}\!\pig(  \, \cU_{B(x\sqrt{N},\delta\sqrt{N})} \, \pig)
	- 1 \pigr)^h \pigr]^{\frac{1}{h}}
	\le \mathsf{C}'' \, \delta^{-2(1-\frac{1}{p})} \,,
\end{equation*}
and taking $p>1$ sufficiently close to~$1$ we have $2(1-\frac{1}{p}) \le \epsilon$.
The bound \eqref{eq:moments-DP} then follows because
$\|Z\|_h \le 1 + \|Z-1\|_h$ by the triangle inequality.
\end{proof}

\smallskip

\begin{proof}[Proof of Theorem~\ref{th:regularity-DP}]

Fix $t > 0$ and $\theta \in \R$. Our goal is to prove \eqref{eq:strong-conv}.
We also fix $\epsilon > 0$ and some integer $r > \epsilon$.
For $N\in\N$, we define the random linear form
\begin{equation*}
	\mathscr{T}_N^\omega(\varphi)  := \int_{\R^2} \varphi(y) \,
	Z_{tN}^{\beta_N^{\mathrm{crit}},\omega} \big( \ev{\sqrt{N} y} \big) \, \dd y
	= \int_{\R^2} \tfrac{1}{N}\, \varphi\big(\tfrac{z}{\sqrt{N}} \big) \,
	Z_{tN}^{\beta_N^{\mathrm{crit}},\omega} \big( \ev{z} \big) \, \dd z \,,
\end{equation*}
where  $\ev{x}$ is the point in $\Z^2_\even$ closest to $x \in \R^2$.
If $\varphi$ is supported on the ball $B(0,R)$, then
$\frac{1}{N} \,
\varphi_x^\lambda (\frac{\cdot}{\sqrt{N}}) = \frac{1}{\lambda^2 N}
\, \varphi(\frac{\cdot \, - x \sqrt{N}}{\lambda \sqrt{N}})$
is supported in $B(x\sqrt{N},  R\lambda\sqrt{N})$.
Recalling \eqref{eq:phi0}, we can then bound
\begin{equation*}
	|\mathscr{T}_N^\omega(\varphi_x^\lambda)|
	\le \|\varphi\|_\infty \, \frac{\mathscr{T}_N^\omega(B(x\sqrt{N},
	R\lambda\sqrt{N}))}{\lambda^2 N}
	\le c \, \|\varphi\|_\infty \, R^2 \,
	Z_{tN}^{\beta_N^{\mathrm{crit}},\omega} \big( \, \cU_{B(x\sqrt{N},
	R\lambda\sqrt{N})}  \big) \,,
\end{equation*}
where the constant $c$ accounts for the discrepancy between the cardinality
$\big| B( z, R) \cap  \mathbb{Z}^2_{\even} \big|$ and the area $\pi R^2$.
Applying the bound \eqref{eq:moments-DP}
for $\delta = R \lambda$ and $h = p \in 2\N$, we obtain
\begin{equation*}
	\forall x \in \R^2, \ \lambda \in (0,1] \colon \qquad
	\sup_{N\in\N} \, \bbE \pig[ \big|\mathscr{T}_N^\omega(\varphi_x^\lambda) \big|^p \pig]^{1/p}
	\le C \, \lambda^{-\epsilon} \qquad \text{with} \ \
	C := \mathtt{C} \, c \, \|\varphi\|_\infty \, R^{2-\epsilon} \,.
\end{equation*}
In particular, choosing $\varphi = \phi$ or $\varphi = \psi^{(i)}$, $1 \le i < 2^d$,
we see that both bounds in \eqref{eq:bounds-Holder} are satisfied
for $\gamma = -\epsilon$, hence $(\mathscr{T}_N^\omega)_{N\in\N}$
is tight in $\cC^{\gamma'}$ for all $\gamma' < -\epsilon - \frac{2}{p}$.
Since $\epsilon > 0$ and $p \in 2\N$ are arbitrary, we conclude
that $(\mathscr{T}_N^\omega)_{N\in\N}$ is tight in $\cC^{-\epsilon}$
for any $\epsilon > 0$.

By the direct half of Prohorov's theorem \cite[Theorem~5.1]{B99}
(which holds for metric spaces), tightness implies relative compactness.
It remains to show that, for any weakly converging subsequence $\mathscr{T}_{N_k}^\omega \to
\mathscr{T}^\omega$, the limit $\mathscr{T}^\omega$ has the same law as the
SHF $\mathscr{Z}_{t}^{\theta,\omega}$.

The law of any random element $\mathscr{T}^\omega$ of $\cC^\gamma$
is determined by
the laws of the random vectors $(\mathscr{T}^\omega(\varphi_1), \ldots, \mathscr{T}^\omega(\varphi_k))$
for $k\in\N$ and $\varphi_1, \ldots, \varphi_k \in C^\infty_c$.
By the linearity of $\varphi \mapsto \mathscr{T}^\omega(\varphi)$ and
the Cramer-Wold device, it is enough to focus on the law
of $\mathscr{T}^\omega(\varphi)$ for a given $\varphi \in C^\infty_c$.
It only remains to show that $\mathscr{T}^\omega(\varphi)$ has the same
distribution as $\mathscr{Z}_{t}^{\theta,\omega}(\varphi)$:
but this follows from the convergence in distribution
\eqref{eq:weak-conv} in the vague topology, which yields
$\mathscr{T}_{N_k}^\omega(\varphi) \to \mathscr{Z}_{t}^{\theta,\omega}(\varphi)$ in distribution
for any $\varphi \in C^\infty_c \subseteq C^0_c$.
The proof is completed.
\end{proof}

\section{Proof of Theorem~\ref{th:log-normality}}
\label{sec:log-normality}

In this section we prove \eqref{eq:log-normal}. By translation invariance, we only consider the case $x=0$.
We also set for simplicity $t=1$ (the proof extends to any $t>0$ with almost no change).
Throughout this section, we fix $\rho \in (0,\infty)$ and,
to lighten notation, we abbreviate
\begin{equation} \label{eq:abbrev}
	Z_N^\av := Z_{N}^{\beta_N} \pig( \, \cU_{\delta_N^\rho\sqrt{N}} \, \pig)
	\qquad \text{where} \quad
	\cU_{\delta_N^\rho\sqrt{N}} := \cU_{B(0,\delta_N^\rho\sqrt{N})} \,.
\end{equation}
where we recall that $\beta_N = \beta_N^{\mathrm{quasi\text{-}crit}}$
is in the quasi-critical regime \eqref{eq:quasi-crit}
for a given sequence $1 \ll |\theta_N| \ll \log N$, and
$\delta_N$ is defined in \eqref{eq:theta}.
We then rephrase our goal \eqref{eq:log-normal} as follows:
\begin{equation}\label{eq:log-normal2}
    Z_N^\av \,\xrightarrow[\,N\to\infty\,]{\ d \ } \,
    \ \mathe^{\cN (0, \sigma^2) - \frac{1}{2} \sigma^2} \qquad \text{with }
    \sigma^2 = \log ( 1+\rho ) \,.
\end{equation}

Once \eqref{eq:log-normal2} is proved, we have the convergence of positive moments: for all $p\ge 0$
\begin{equation*}
	\lim_{N\to\infty} \bbE\big[\big( Z_N^\av \big)^p] = \bbE\big[ \big(\mathe^{\cN (0, \sigma^2) - \frac{1}{2} \sigma^2}\big)^p \big]
	= \rme^{\frac{p(p-1)}{2} \, \sigma^2}
	= (1+\rho)^{\frac{p(p-1)}{2}} \,.
\end{equation*}
This follows by the weak convergence \eqref{eq:log-normal2} once we show that positive moments are uniformly bounded, say $\sup_{N\in\N} \bbE[(Z_N^\av)^{2h}] < \infty$ for all $h\in\N$ (so that $(Z_N^\av)^p$ is uniformly integrable). We already know from Corollary~\ref{th:sanity} that $\bbvar[Z_N^\av]$ is bounded, hence any moment $\bbE[(Z_N^\av)^{2h}]$ is bounded too by Theorem~\ref{th:genmombou} (the assumptions \eqref{eq:concsqrtL}, \eqref{eq:locunif1sup} and \eqref{eq:varub0} of Theorem~\ref{th:genmombou} are satisfied by $\varphi = \cU_{(0,\delta_N^\rho\sqrt{N})}$ with
$\bbD[\varphi] \sim (\delta_N^\rho)^2 N$).

\begin{remark}
The log-normality \eqref{eq:log-normal2} is reminiscent of the corresponding result for the point-to-plane partition function $Z_L^{\beta}(0)$
in the sub-critical regime $\beta = \beta_L^{\text{\rm sub-crit}}$, see \eqref{eq:sub-critical}
with $\hat\beta \in (0,1)$,
that we proved in \cite[Theorem~2.8]{CSZ17b}.
As we described in Remark~\ref{rem:analogy-sub-critical0}, the heuristic behind \eqref{eq:log-normal2}
is that, after coarse-graining space on the scale $\delta_N^\rho \sqrt{N}$ and time on
the scale $\delta_N^{2\rho} N$, the averaged partition function $Z_N^\av$ in the quasi-critical regime
becomes comparable to a sub-critical point-to-plane partition function
with disorder parameter $\hat\beta^2 = \frac{\rho}{1+\rho}$ and effective
time horizon $L = 1 / \delta_N^{2\rho}$.

The key to the analysis of $Z_L^{\beta_L^{\text{\rm sub-crit}}}(0)$
in the sub-critical regime in \cite{CSZ17b}
was a \emph{multiscale structure} with time scales
$L^\alpha$ for $\alpha \in (0, 1)$.  We will see that
the same multiscale structure emerges in our analysis of $Z_N^\av$ with corresponding time scales
$(\delta_N^{2\rho} N) \, L^\alpha$ with $L = 1 / \delta_N^{2\rho}$ and $\alpha \in (0,1)$, see \eqref{eq:scales}.
This justifies the heuristic comparison just described.
\end{remark}

\smallskip

We divide the proof of \eqref{eq:log-normal2} into several steps.

\smallskip
\subsection*{Overall strategy}

Let us fix a (large) integer $M\in\N$, which will be the number of time scales.
For technical reasons, we will first approximate
the original partition function $Z_N^\av$
by switching off the disorder in suitable time strips,
which defines $Z_N^\off$ in \eqref{eq:tildeZ} and provides some smoothing between
consecutive time scales; we will then introduce almost diffusive restrictions on the polymer paths,
which defines $Z_N^\diff$ in \eqref{eq:hatZ}.
We will show that $Z_N^\off$ and $Z_N^\diff$ (which also depend on~$M$)
are good approximations in the following sense:
\begin{equation}\label{eq:appro1}
	\forall M \in \N: \qquad
	Z_N^\av -Z_N^\off
	\, \xrightarrow[\, N\to\infty \, ]{L^2} \, 0 \,, \qquad
	Z_N^\off - Z_N^\diff \, \xrightarrow[\, N\to\infty \, ]{L^1} \, 0 \,.
\end{equation}
Therefore to prove our goal \eqref{eq:log-normal2}, we can just replace $Z_N^\av$ with $Z_N^\diff$.

\begin{remark}
A similar idea of switching-off the noise in suitably chosen strips to obtain a
smoothing approximation was also used by Dunlap-Gu \cite{DG22} in their treatment of
nonlinear SHE in a subcritical regime.
\end{remark}

We introduce explicit time scales
\begin{equation} \label{eq:separation}
	0 =: N_0 < \tilde{N}_1 \ll N_1 \ll \tilde{N}_2 \ll N_2 \ll \ldots \ll
	\tilde{N}_{M} \ll N_{M} = N \,.
\end{equation}
We define $Z_N^\off$ by switching off the noise in the time strips $(\tilde{N}_i, N_i]$,
see \eqref{eq:tildeZ}, then we define $Z_N^\diff$ by restricting polymer paths
at times $\tilde{N}_i$ and $N_i$ to an almost diffusive ball, see \eqref{eq:AL} and \eqref{eq:hatZ}.
The scales $N_i$ are defined as follows:
\begin{equation}\label{eq:scales}
	N_0 := 0\,, \qquad
	N_{i} := \pig\llbracket (N \delta_N^{2\rho}) \, (\tfrac{1}{\delta_N^{2\rho}})^{\frac{i}{M}}
	\pig\rrbracket = \pig\llbracket N \, (\delta_N^{2\rho})^{1-\frac{i}{M}}
	\pig\rrbracket \quad \text{for } \ i=1,\ldots, M \,,
\end{equation}
where $\llbracket a \rrbracket := 2 \, \lfloor a/2 \rfloor$
is an even proxy for $a \in \R$.
Note that $N_{i+1} \approx N_{i} \, (\tfrac{1}{\delta_N^{2\rho}})^{\frac{1}{M}}$
for $i \ge 1$.
The intermediate scales $\tilde N_i$ are then defined by
  \begin{equation} \label{eq:tildeNi}
   \tilde N_i := \bigg\llbracket \frac{N_i}{( \log \frac{1}{\delta_N^{2\rho}} )^3} \bigg\rrbracket
   \qquad \text{for } i=1, 2, \ldots, M \,.
  \end{equation}

By restricting random walk paths in the definition
of $Z_{N}^\diff$ to the time interval $[0,N_i]$, we obtain a sequence
of quantities
$Z_{N,i}^\diff$ for $i=1, \ldots, M$ with $Z_{N,M}^\diff = Z_{N}^\diff$,
see \eqref{eq:hatZ}.
Then, by a telescopic product, we can write
\begin{equation} \label{eq:decomp}
	\qquad\qquad\qquad Z_N^\diff
	= \prod_{i=1}^{M} \frac{Z_{N,i}^\diff}{Z_{N,i-1}^\diff}
	\qquad \qquad \text{(with $Z_{N,0}^\diff := 1$)} \,.
\end{equation}
The choice of the scales $N_i$ is made so that the \emph{the ratios in the RHS have variance of the same order $\frac{1}{M}$, albeit with a varying prefactor} (see Theorem~\ref{th:2mom}).
The fact that these scales are well separated, see \eqref{eq:separation},
ensures that \emph{the ratios in the RHS are approximately independent}.

\begin{remark}[Exponential time scales]
The choice of the scales \eqref{eq:scales}, which leads to the decomposition \eqref{eq:decomp}
into approximately independent factors,
resembles what is observed in the sub-critical regime \eqref{eq:sub-critical} in \cite{CSZ17b} and
in the more recent papers \cite{CC22,CD24}.
\end{remark}

\smallskip

Denoting by $\cF_L := \sigma\{\omega(n,z) \colon n \le L, z \in \Z^2\}$ the $\sigma$-algebra
generated by disorder variables up to time~$L$,
we introduce the conditional expectation
\begin{equation}\label{eq:mui}
	m_{N,i} := \bbE \bigg[ \frac{Z_{N,i}^\diff}{Z_{N,i-1}^\diff}
	\,\bigg|\, \cF_{N_{i-1}} \bigg]
	\qquad \text{for } i=1, \ldots, M \,,
\end{equation}
which turns out to be close to~$1$ with high probability (see \eqref{eq:bound-cond-exp}).
We define $\Delta_{N,i}$ as the
normalised and centred version of the ratios in the RHS of \eqref{eq:decomp}:
\begin{equation}\label{eq:Delta}
	\Delta_{N,i} := \frac{1}{m_{N,i}}
	\, \frac{Z_{N,i}^\diff}{Z_{N,i-1}^\diff} - 1  \,.
\end{equation}
This leads to the identity
\begin{equation} \label{eq:logZ}
\begin{split}
	\log Z_N^\diff
	&= \sum_{i=1}^{M} \log \frac{Z_{N,i}^\diff}{Z_{N,i-1}^\diff}
	= \sum_{i=1}^{M} \big\{ \log\big(1 + \Delta_{N,i} \big)
	+ \log m_{N,i}  \big\}  \\
	& = \sum_{i=1}^{M}
	\big\{ \Delta_{N,i} - \tfrac{1}{2} \Delta_{N,i}^2
	+ r(\Delta_{N,i})
	+ \log m_{N,i} \big\} \,,
\end{split}
\end{equation}
where $r(\cdot)$ is the remainder in the second order Taylor expansion of the logarithm:
\begin{equation}\label{eq:remainder0}
	r(x) := \log (1+x) - \big( x - \tfrac{x^2}{2} \big) \,.
\end{equation}

To complete the proof  of our goal \eqref{eq:log-normal2},
we are going to show that, for $M = M_N \to \infty$
slowly enough, the following three convergences in distribution hold:
\begin{align}
	\label{eq:CLT}
	\sum_{i=1}^{M_N} \Delta_{N,i} & \, \xrightarrow[\, N\to\infty \, ]{d} \, \cN(0,\sigma^2) \,, \\
	\label{eq:LLN}
	\sum_{i=1}^{M_N} \Delta_{N,i}^2 & \, \xrightarrow[\, N\to\infty \, ]{d} \, \sigma^2 \,, \\
	\label{eq:remainder}
	\sum_{i=1}^{M_N} \Big(r(\Delta_{N,i}) + \log m_{N,i}\Big) & \, \xrightarrow[\, N\to\infty \, ]{d} \, 0 \,,
\end{align}
with $\sigma^2 $ as in \eqref{eq:log-normal2}. Intuitively,
these relations hold because \emph{the random variables $\Delta_{N,i}$
are approximately independent},
due to the separation of time scales \eqref{eq:scales}.
The proof of \eqref{eq:log-normal2} then follows
by first choosing $M = M_N\to\infty$ slowly enough such that \eqref{eq:appro1} still holds,
and then applying the identity \eqref{eq:logZ}.

\smallskip

The rest of the proof is divided into the following steps:
\begin{itemize}
\item in Steps~1 and~2 we define $Z_N^\off$ and $Z_N^\diff$ and
prove the two limits in \eqref{eq:appro1};
\item in Step~3 we give a convenient representation
for the ratio $Z_{N,i}^\diff/Z_{N,i-1}^\diff$ as a partition function
on the time interval $(N_{i-1},N_i]$
with initial condition given by a \emph{polymer distribution at time $N_{i-1}$},
and we show that the latter is close to the free random walk thanks to the fact
that noise has been turned off in the time interval $(\tilde N_i, N_i]$.

\item in Steps~4 and~5 we compute the variance of $\Delta_{N,i}$ and bound
its higher moments;

\item in Step~6 we prove \eqref{eq:CLT} through the \emph{martingale CLT};

\item in Step~7 we prove
relations \eqref{eq:LLN} and \eqref{eq:remainder} by variance bounds.
\end{itemize}

\smallskip
\subsection{Step 1: switching off the noise}

The first approximation $Z_N^\off$
of the partition function
$Z_N^\av$, recall \eqref{eq:abbrev},  is obtained
by ``switching off the noise'' in the time strips $(\tilde N_i, N_i]$ for $1 \le i \le M$,
see \eqref{eq:scales} and \eqref{eq:tildeNi}:
this will ensure that the endpoint distribution of the polymer at time $N_i$ is comparable to the random walk
transition kernel. Recalling \eqref{eq:ZAB} and \eqref{eq:H}, we thus define
\begin{equation}\label{eq:tildeZ}
	Z_N^\off
	\assign \E \Big[
	\rme^{\cH^{\beta_N}_{(0,N] \setminus \bigcup_{j=1}^M (\tilde N_j, N_j]}}
	\,\Big|\, \nobracket S_0 \sim \cU_{\delta_N^\rho\sqrt{N}} \Big]  \,.
\end{equation}
In this step, we prove the first relation in \eqref{eq:appro1}:
  \begin{equation} \label{eq:ZtildeZ}
	\forall M \in \N: \qquad
    \lim_{N \rightarrow \infty} \mathbb{E} \big[ \big( Z_N^\av
    - Z_N^\off   \big)^2 \big] = 0 \, .
  \end{equation}

Let $\cG$ be the $\sigma$-algebra generated by the disorder variables that
have \emph{not} been switched off, namely
 $\cG := \sigma\big( \omega(n,z) \colon
n \in \bigcup_{i=1}^{M} (N_{i-1}, \tilde N_i] , \, z \in \Z^2 \big)$, then we can write
\begin{equation*}
	Z_N^\off = \bbE \pig[ Z_N^\av  \,\pig|\, \cG \pig] \,,
\end{equation*}
that is $Z_N^\off$ is the orthogonal projection of $Z_N^\av$
onto the linear subspace of $L^2$ generated by $\cG$-measurable random variables.
It follows that
\begin{equation*}
	\bbE \big[ \big( Z_N^\av - Z_N^\off \big)^2 \big]
	= \bbE \big[ (Z_N^\av) ^2 \big]
	- \bbE \big[( Z_N^\off)^2 \big] \,,
\end{equation*}
hence to prove \eqref{eq:ZtildeZ} we need to show that
\begin{equation}\label{eq:appro1a}
	\forall M \in \N: \qquad
	\lim_{N\to\infty} \Big\{ \bbE \big[ (Z_N^\av)^2 \big]
	- \bbE \big[ ( Z_N^\off)^2 \big] \Big\} = 0 \,.
\end{equation}

\smallskip
\begin{proof}[Proof of \eqref{eq:appro1a}]

Let us define a variant $\tilde Z_{N,j}^\off$ of \eqref{eq:tildeZ} where we only switch off
disorder in a given interval $(\tilde N_j, N_j]$, namely
\begin{equation}\label{eq:tildeZj}
	\text{for } \ j=1,\ldots, M \colon \qquad
	\tilde Z_{N,j}^\off
	\assign \E \Big[
	\rme^{\cH^{\beta_N}_{(0,N] \setminus (\tilde N_j, N_j]}}
	\,\Big|\, \nobracket S_0 \sim \cU_{\delta_N^\rho \sqrt{N}} \Big]  \,.
\end{equation}
We claim that we can bound the difference in \eqref{eq:appro1a} by
\begin{equation}\label{eq:appro1a2}
	\bbE \big[ (Z_N^\av)^2 \big]
	- \bbE \big[ ( Z_N^\off)^2 \big]
	= \bbvar \big[ Z_N^\av \big]
	- \bbvar \big[ Z_N^\off \big]
	\le \sum_{j=1}^M
	\pig\{ \bbvar \big[ Z_N^\av \big]
	- \bbvar \big[ \tilde Z_{N,j}^\off \big]  \pig\} \,.
\end{equation}
It then suffices to estimate separately each term in this sum.

\smallskip

In order to prove \eqref{eq:appro1a2},
recall the polynomial chaos expansion \eqref{eq:polynomial} for the point-to-plane partition
function $Z_N^\beta(x)$, which yields a corresponding polynomial chaos expansion for the averaged
partition function $Z_N^\beta(\varphi) = \sum_{x\in\Z^2_\even} \varphi(x) \, Z_N^\beta(x)$.
The polynomial chaos expansion for $Z_N^\off$ from
\eqref{eq:tildeZ} is a \emph{subset} of the  polynomial chaos expansion
for $Z_N^\av$: it is obtained by restricting the sum
to times $n_1, \ldots, n_k$ which \emph{avoid all intervals $(\tilde N_j, N_j]$} for $j=1,\ldots, M$
(see Remark~\ref{rem:switch-off}),
hence its variance admits a formula like \eqref{eq:varexpl}
with the same restriction on the sum.
Then the difference $\bbvar \big[ Z_N^\av \big]
	- \bbvar \big[ Z_N^\off \big]$
is given again by formula \eqref{eq:varexpl}
where the sequence of times $n_1, \ldots, n_k$ is now required
to \emph{intersect at least one of the intervals $(\tilde N_j, N_j]$ for $j=1,\ldots, M$}.
By a union bound, we obtain precisely \eqref{eq:appro1a2}.

\smallskip

Let us finally focus on a given term in the RHS of \eqref{eq:appro1a2}.
Since $\bbvar \big[ Z_N^\av \big] \to \rho$
as $N\to\infty$, see \eqref{eq:2nd-mom0}, it is enough to show that
\begin{equation}\label{eq:ourgoal}
	\forall j \in \{1,\ldots, M\} \colon \qquad
	\lim_{N\to\infty} \bbvar \big[ \tilde Z_{N,j}^\off \big] = \rho \,.
\end{equation}
We recall that $\tilde Z_{N,j}^\off$ corresponds to
switching off disorder between
$A_N = \tilde N_j$ and $B_N = N_j$.
It is instructive (and more transparent) to fix $0 < a \le b < 1$ and
consider general times
\begin{equation*}
\begin{split}
	A_N &= N\, (\delta_N^{2\rho})^{1-a + o(1)}
	= N \, \rme^{-(1-a) \, \rho \, |\theta_N| + o(|\theta_N|)} \,, \\
	B_N &= N\, (\delta_N^{2\rho})^{1-b + o(1)}
	= N \, \rme^{-(1-b) \, \rho \, |\theta_N| + o(|\theta_N|)} \,.
\end{split}
\end{equation*}
Later we will specialize to $b=a= \frac{j-1}{M}$ due to the choice of $N_j$ and $\tilde N_j$
in \eqref{eq:scales} and \eqref{eq:tildeNi}.

We compute $\bbvar \big[ \tilde Z_{N,j}^\off \big]$ by formula
\eqref{eq:varexpl} where the sum is restricted to times $n_1, \ldots, n_k$
that \emph{do not intersect the interval $(A_N, B_N]$}.
We split $\bbvar \big[ \tilde Z_{N,j}^\off \big] = I_1 + I_2 + I_3$ as follows:
\begin{itemize}
\item part $I_1$: all times $n_i$ take place before $A_N$;
\item part $I_2$: all times $n_i$ take place after $B_N$;
\item part $I_3$: the first time $n_1$ takes place before $A_N$,
the last time $n_k$ takes place after $B_N$.
\end{itemize}

The first contribution $I_1$ is nothing but the variance of
the averaged partition function of polymer length~$A_N$, that is (recall \eqref{eq:abbrev})
\begin{equation}\label{eq:I1}
	I_1 = \bbvar \big[ Z_{A_N}^{\beta_N}(\cU_{\delta_N^\rho\sqrt{N}}) \big]
	\sim \frac{a \rho}{1+(1-a)\rho} \,,
\end{equation}
where we applied \eqref{eq:2nd-mom} with
$L_N = A_N$ for which $\ell=(1-a)\rho$, see \eqref{eq:L},
and $W_N = N \delta_N^{2\rho}$ for which $w=\rho$,
see \eqref{eq:W} (and recall \eqref{eq:ball-rw}).

The second contribution $I_2$, when all times $n_i$ are after $B_N$,
corresponds to switching off the noise in the whole interval $[0, B_N]$, hence
we have a partition function on the interval $(B_N, N]$ (whose length is $L_N = N - B_N \sim N$
since $b < 1$, that is \eqref{eq:L} holds with $\ell = 0$)
with initial condition at time~$B_N$ given by $\cU_{\delta_N^\rho\sqrt{N}} * q_{B_N}$,
i.e., the distribution of the random walk at time $B_N$ with initial condition $\cU_{\delta_N^\rho\sqrt{N}}$
(which satisfies assumption \eqref{eq:conditions-phi} with $W_N = B_N$,
hence \eqref{eq:W} holds with $w=(1-b)\rho$).
We thus obtain by \eqref{eq:2nd-mom}
\begin{equation*}
	I_2 \sim (1-b)\rho \,.
\end{equation*}

We finally consider the third contribution $I_3$, when there are times $n_i$
both before $A_N$ and after $B_N$:
recalling \eqref{eq:varexpl}, we can write
\begin{equation*}
	I_3 = \sumtwo{0 < m \le g \le A_N}{B_N < d\le n \le N}
	q_{2m}\pig(\cU_{\delta_N^\rho\sqrt{N}}, \, \cU_{\delta_N^\rho\sqrt{N}}\pig)
	\, \sigma_{\beta_N}^2 \, U_{\beta_N}(g-m)
	\, \sigma_{\beta_N}^2 \, q_{2(d-g)}(0) \,
	U_{\beta_N}(n-d) \,,
\end{equation*}
where we recall that $U_{\beta_N}(\cdot)$ was defined in \eqref{eq:U}.
Summing over $n$ we obtain $\overline{U}_{\beta_N}(N - d)$,
see \eqref{eq:barU}, and
restricting the sum to $d \le \frac{1}{2} N$ we can bound
$\overline{U}_{\beta_N}(N-d) \ge \overline{U}_{\beta_N}(\tfrac{1}{2}N)$.
Summing $q_{2(d-g)}(0)$ over $B_N < d\le \frac{1}{2} N$ then gives,
recalling \eqref{eq:RL},
\begin{equation*}
	\sum_{B_N < d \le \frac{1}{2}N} q_{2(d-g)}(0)
	\,=\, R_{\frac{1}{2}N-g} - R_{B_N - g}
	\,\sim\, \frac{1}{\pi} \log \frac{\frac{1}{2}N-g}{B_N-g}
	\,\ge\, \frac{1}{\pi} \log \frac{\frac{1}{2}N}{B_N}
\end{equation*}
because the minimum is attained at $g = 0$.
Recalling that $\sigma_{\beta_N}^2
\sim \frac{\pi}{\log N}$, we then obtain
\begin{equation*}
\begin{split}
	I_3
	& \ge \Bigg\{
	\sum_{0 < m \le g \le A_N} q_{2m}\pig(\cU_{\delta_N^\rho\sqrt{N}}, \, \cU_{\delta_N^\rho\sqrt{N}}\pig)
	\, \sigma_{\beta_N}^2 \, U_{\beta_N}(g-m) \Bigg\}
	\, \frac{\log \frac{\frac{1}{2}N}{B_N}}{\log N} \,
	\overline{U}_{\beta_N}(\tfrac{1}{2}N) \\
	& = \bbE \big[ Z_{A_N}^{\beta_N}(\cU_{\delta_N^\rho\sqrt{N}}) ^2 \big] \,
	\frac{\log \frac{\frac{1}{2}N}{B_N}}{\log N} \,
	\overline{U}_{\beta_N}(\tfrac{1}{2}N) \,,
\end{split}
\end{equation*}
where the  equality follows by \eqref{eq:barU} and \eqref{eq:varcomp}.
Applying \eqref{eq:point-to-plane-2ndmom} we see that
$\overline{U}_{\beta_N}(\tfrac{1}{2}N) \sim \frac{\log N}{|\theta_N|}$,
while $\log \frac{\frac{1}{2}N}{B_N} \sim (1-b) \log \frac{1}{\delta_N^{2\rho}} \sim
(1-b)\rho\, |\theta_N|$, see \eqref{eq:theta}.
Recalling \eqref{eq:I1} we then obtain
\begin{equation*}
	I_3 \ge (1+o(1)) \, I_1 \cdot \frac{(1-b) \rho\, |\theta_N|}{\log N} \cdot \frac{\log N}{|\theta_N|}
	\sim I_1 \, (1-b) \rho \,.
\end{equation*}

Overall, summing the three parts $I_1$, $I_2$ and $I_3$,
we have shown that
\begin{equation*}
\begin{split}
	\bbvar \big[ \tilde Z_{N,j}^\off \big]
	&= I_1 + I_2 + I_3
	\ge I_1 \, (1+ (1-b)\rho) + I_2 + o(1)
	\sim \frac{1+(1-b)\rho}{1+(1-a)\rho} \, a\rho + (1-b)\rho \\
	& = \rho \,-\, (b-a)\frac{\rho\,(1+\rho)}{1+(1-a)\rho}
	= \rho \,-\, O(b-a) \,.
\end{split}
\end{equation*}
This last expression vanishes for $a=b$, which completes the proof of
\eqref{eq:ourgoal}.
\end{proof}

\smallskip
\subsection{Step 2: almost diffusive approximation}
In this step, we prove the second relation in \eqref{eq:appro1}:
  \begin{equation} \label{eq:ZhatZ}
	\forall M \in \N: \qquad
    \lim_{N \rightarrow \infty} \mathbb{E} \big[ | Z_{N}^\off
    - Z_{N}^\diff   |\big] = 0 \, ,
  \end{equation}
Let  $\mathcal{D}_m$ be the event that the random walk is ``almost diffusive'' at
  time $m$, in the following sense:
  \begin{equation}  \label{eq:AL}
    \mathcal{D}_m \assign \left\{ | S_m | \le \sqrt{m \log
    \tfrac{1}{\delta_N^2}} \right\} \,.
  \end{equation}
We define $Z_{N}^\diff $ by restricting  $Z_{N}^\off $ in \eqref{eq:tildeZ}
 to the event $\bigcap_{j=0}^M \mathcal{D}_{\tilde N_j} \cap \mathcal{D}_{N_j}$.
It is actually useful to define $Z_{N,i}^\diff$
for each scale $N_i$, see \eqref{eq:scales}
(note that $N_i = N$ for $i=M$):
  \begin{equation}\label{eq:hatZ}
\begin{split}
  \text{for } i=1,\ldots, M: \qquad
   & Z_{N,i}^\diff \assign
    \E \bigg[ \prod_{j=1}^i \mathe^{\mathcal{H}_{(N_{j-1}, \tilde N_j]}^{\beta_N}}
    \, \ind_{\mathcal{D}_{\tilde N_j} \cap \mathcal{D}_{N_j}}
	\,\bigg|\, \nobracket S_0 \sim \cU_{\delta_N^\rho\sqrt{N}} \bigg] \\
	& \text{and we set} \quad Z_{N}^\diff := Z_{N,M}^\diff \,.
\end{split}
  \end{equation}

Let us prove \eqref{eq:ZhatZ}.
Since $|Z_{N}^\off
    - Z_{N}^\diff | = Z_{N}^\off
    - Z_{N}^\diff$ and $\mathbb{E} [ \mathe^{\mathcal{H}_{(a,b]}} ] = 1$, we have
  \begin{equation*}
\begin{split}
    \mathbb{E} \big[ | Z_{N}^\off
    - Z_{N}^\diff |\big]
    & =1 - \P
    \bigg( \bigcap_{j=1}^{M} \mathcal{D}_{\tilde N_j} \cap \mathcal{D}_{N_j}
   \, \bigg| \,  S_0 \sim \cU_{\delta_N\sqrt{N}} \bigg) \\
    & \le \sum_{m \in \{\tilde N_1, \, N_1, \, \ldots, \, \tilde N_{M}, \, N_{M}\}}
    \P \Big( | S_m | > \sqrt{m \log \tfrac{1}{\delta_N^2}}
    \,\Big|\, \nobracket S_0 \sim \cU_{\delta_N^\rho \sqrt{N}} \Big) \, .
\end{split}
  \end{equation*}
  We recall that under $\P (\, \cdummy \,|\, \nobracket S_0 \sim \cU_{\delta_N^\rho \sqrt{N}})$
  we have $| S_0 | \le \delta_N^\rho  \sqrt{N}$. Since
  $\tilde{N}_1\gg N \delta_N^{2\rho}$,
  see \eqref{eq:scales} and \eqref{eq:tildeNi}, for $m \ge \tilde{N}_1$ we can
  bound $\delta_N^{\rho} \sqrt{N} \le \sqrt{m}$, hence
  \[ \sqrt{m \log \tfrac{1}{\delta_N^2}} - \delta_N^\rho  \sqrt{N}
     \ge \sqrt{m} \Big( \sqrt{\log \tfrac{1}{\delta_N^2}} - 1 \Big)
     \ge \frac{1}{2}\, \sqrt{m}\, \sqrt{\log
     \tfrac{1}{\delta_N^2}} \quad \text{for large } N \,. \]
Then \eqref{eq:ZhatZ} holds because
for any $m \in \{\tilde N_1, \, N_1, \, \ldots, \, \tilde N_{M}, \, N_{M}\}$, we can
use Gaussian tail estimates for the simple symmetric random walk to bound
\begin{equation}\label{eq:Gaussian-tail}
\begin{split}
	\P \Big( | S_m | > \sqrt{m \log \tfrac{1}{\delta_N^2}}  \,\Big|\,
	\nobracket S_0 \sim \cU_{\delta_N^\rho \sqrt{N}} \Big)
	\le \P \Big( | S_m - S_0 | > \tfrac{1}{2} \sqrt{m}
	\, \sqrt{ \log \tfrac{1}{\delta_N^2}}  \Big) \  & \\
        \le C\, \exp \big( - \frac{1}{C} \, \big| \sqrt{\log \tfrac{1}{\delta_N^2}} \big|^2 \big)
     = C \, (\delta_N^2)^{\frac{1}{C}}  \xrightarrow{N \rightarrow \infty}
     0 \, .  &
\end{split}
\end{equation}
This concludes the proof of  \eqref{eq:ZhatZ}.

\smallskip
\subsection{Step 3: polymer distribution}

In this step we give a convenient representation
for the ratio $Z_{N,i}^\diff/Z_{N, i-1}^\diff$
in terms of a directed polymer. This
will be exploited to estimate the moments of $\Delta_{N,i}$, see \eqref{eq:Delta}.

Let us introduce the polymer endpoint distribution $\mu_{N,i}(\cdot)$ at time $N_i$,
corresponding to the partition function $Z_{N,i}^\diff$ in \eqref{eq:hatZ}:
\begin{equation} \label{eq:phiNi}
	\text{for } i=1,\ldots, M \text{ and } x \in \Z^2_\even: \qquad
	\mu_{N,i}(x) := \frac{Z_{N,i}^\diff\big(\cU_{\delta_N^\rho\sqrt{N}},x\big)}{Z_{N,i}^\diff} \,.
\end{equation}
where we define $Z_{N,i}^\diff \big(\cU_{\delta_N^\rho\sqrt{N}}, x\big)$ by restricting
paths in the definition of $Z_{N,i}^\diff$ to $S_{N_i} = x$, that is
\begin{equation}\label{eq:hatZx}
    Z_{N,i}^\diff \big(\cU_{\delta_N^\rho\sqrt{N}}, x\big) \assign
    \E \bigg[ \bigg( \prod_{j=1}^i \mathe^{\mathcal{H}_{(N_{j-1}, \tilde N_j]}^{\beta_N}}
    \, \ind_{\mathcal{D}_{\tilde N_j} \cap \mathcal{D}_{N_j}}  \bigg) \, \ind_{\{S_{N_i} = x\}}
	\,\bigg|\, \nobracket S_0 \sim \cU_{\delta_N^\rho\sqrt{N}} \bigg] \, .
\end{equation}
By the Markov property, the following representation holds
for $i=2,\ldots, M$:
  \begin{equation} \label{eq:ratiorepr}
\begin{split}
    \frac{Z_{N,i}^\diff}{Z_{N,i-1}^\diff}
    & = \E \Big[
    \mathe^{\mathcal{H}^{\beta_N}_{(N_{i-1}, \tilde N_i]}} \, \ind_{\mathcal{D}_{\tilde N_i}
    \cap \mathcal{D}_{N_i}}  \Big|\,
    S_{N_{i-1}} \sim \mu_{N,i-1} \Big] \\
    & = \sum_{x \in
    \mathbb{Z}^2_{\even}} \mu_{N,i-1} (x)\, \E \Big[
    \mathe^{\mathcal{H}^{\beta_N}_{(N_{i-1}, \tilde N_i]}} \, \ind_{\mathcal{D}_{\tilde N_i}
    \cap \mathcal{D}_{N_i}}  \Big|\,
    S_{N_{i-1}} = x \Big] \,.
\end{split}
  \end{equation}
The same formula holds also for $i=1$ provided
we define
\begin{equation*}
	\mu_{N,0}(x) := \cU_{\delta_N^\rho\sqrt{N}}(x) \,.
\end{equation*}

\begin{remark}
The switching off of the noise ensures that the ``initial distribution'' $\mu_{N,i-1}$
in \eqref{eq:ratiorepr} is sufficiently smooth, as we show below.
This will be needed in the next steps
to compute the variance and to estimate the moments of $\Delta_{N,i}$
from \eqref{eq:Delta}.
\end{remark}

Representation \eqref{eq:ratiorepr} is very useful. For instance, recalling \eqref{eq:mui},
we can compute
\begin{equation} \label{eq:mucomp}
	m_{N, i-1} = \P \big(\mathcal{D}_{\tilde N_i} \cap \mathcal{D}_{N_i}  \,\big|\,
	S_{N_{i-1}} \sim \mu_{N,i-1} \big) \,.
\end{equation}
Note that for $S_{N_{i-1}} \sim \mu_{N,i-1}$, we have
$|S_{N_{i-1}}| \le \sqrt{N_{i-1} \log \frac{1}{\delta_N^2}}$
due to the restriction to the event $\cD_{N_{i-1}}$,
see \eqref{eq:hatZx} and \eqref{eq:AL}. Therefore for any $i=1,\ldots, M$, we can bound
\begin{equation*}
\begin{split}
	m_{N,i-1}
	\ge 1 &- \P\Big( |S_{\tilde N_i} - S_{N_{i-1}}| > (\sqrt{\tilde N_i} - \sqrt{N_{i-1}})
	\sqrt{\log \tfrac{1}{\delta_N^2}} \,\Big) \\
	& - \P\Big( |S_{N_i} - S_{N_{i-1}}| > (\sqrt{N_i} - \sqrt{N_{i-1}})
	\sqrt{\log \tfrac{1}{\delta_N^2}} \,\Big) \,.
\end{split}
\end{equation*}
Since $N_{i-1} \ll \tilde N_i \ll N_i$, see \eqref{eq:scales} and \eqref{eq:tildeNi},
arguing as in \eqref{eq:Gaussian-tail}, we have: for some $C=\blue{C(M)} < \infty$,
\begin{equation}\label{eq:bound-cond-exp}
	\text{for } i=1,\ldots, M: \qquad
	1 - C \, (\delta_N^2)^{\frac{1}{C}} \le m_{N, i-1} \le 1 \,.
\end{equation}

We conclude this step by showing that the polymer distribution
\emph{$\mu_{N,i}(x)$ is close to the
random walk transition kernel $q_{N_i - \tilde N_i}(x)$}, see \eqref{eq:rw}.
Intuitively, this holds because:
\begin{itemize}
\item we switched off disorder between times
$\tilde N_i$ and $N_i$, see \eqref{eq:hatZx}, therefore the polymer
evolves between these times as simple random walk;
\item we know that $|S_{\tilde N_i}| \le \sqrt{\tilde N_i \log \frac{1}{\delta_N^2}}$
(due to the event $\cD_{\tilde N_i}$), therefore $|S_{\tilde N_i}| \ll \sqrt{N_i}$
by the choice of $\tilde N_i$ in \eqref{eq:tildeNi}. To compute $S_{N_i}$,
we can therefore pretend that $S_{\tilde N_i} \simeq 0$,
namely we can approximate $S_{N_i} - S_{\tilde N_i} \simeq S_{N_i-\tilde{N}_i}$,
which is distributed as $q_{N_i - \tilde N_i}(\cdot)$.
\end{itemize}

Let us now be precise: we set for short
  \begin{equation}\label{eq:epsN}
	\epsilon_N := \big(\log \tfrac{1}{\delta_N^2}\big)^{-1/3} \xrightarrow[\, N\to\infty\, ]{} 0 \,,
\end{equation}
and we prove that for any $M\in\N$ and $i= 1, \ldots, M$, we have
\begin{equation}\label{eq:cM}
\begin{split}
    \mu_{N,i} \in \mathcal{M}_{N,i} \assign \bigg\{
    &   \varphi (\cdot) \ge 0
    \ \text{ supported in }   \Big\{ | \cdot | \le \sqrt{N_i \log
    \tfrac{1}{\delta_N^2}} \, \Big\}     \\
    & \ \text{ with } \ \sum_x \varphi (x) = 1 \ \text{ and } \ \varphi (\cdummy) \le (1+\epsilon_N)^2\,
    q_{N_i - \tilde N_{i}} (\cdummy) \, \bigg\} \,.
\end{split}
  \end{equation}
(This relation requires $i \ge 1$, but we will not need it for $i=0$.)

\smallskip

\begin{proof}[Proof of \eqref{eq:cM}]
We only need to prove that $\mu_{N,i} (\cdot) \le (1+\epsilon_N)^2\,
q_{N_i - \tilde N_{i}} (\cdot)$.
We express $\mu_{N,i}(x)$ by summing over the polymer position at
time $\tilde N_i < N_i$: denoting by $\phi_{\tilde N_i}(z)$ the corresponding
distribution (defined as in \eqref{eq:hatZx}-\eqref{eq:phiNi} with $\{S_{N_i}=x\}$
replaced by $\{S_{\tilde N_i}=z\}$),
and recalling the random walk transition kernel from \eqref{eq:rw},
for $x\in\Z^2_\even$ with $|x| \le \sqrt{N_i \log \tfrac{1}{\delta_N^2}}$, we get
  \begin{equation} \label{eq:phiratio}
  	\mu_{N,i}(x) =
	\frac{\displaystyle\sum_{	| z | \le \sqrt{\tilde N_i \log \frac{1}{\delta_N^2}}} \phi_{\tilde N_i} (z)
	\, q_{N_i - \tilde N_i} (x-z) }{\displaystyle \sum_{| z | \le
	\sqrt{\tilde N_i \log \frac{1}{\delta_N^2}}}  \; \phi_{\tilde N_i} (z) \;
	\bigg\{\sum_{| x' | \le \sqrt{N_i \log \frac{1}{\delta_N^2}}}
	\, q_{N_i - \tilde N_i} (x'-z) \bigg\}} \,.
  \end{equation}
To obtain $\mu_{N,i} (\cdot) \le (1+\epsilon_N)^2\,
q_{N_i - \tilde N_{i}} (\cdot)$, it suffices to prove
the following bounds on the numerator and denominator in \eqref{eq:phiratio}:
for any $M\in\N$ and $i=1,\ldots, M$, we have, for large~$N$,
\begin{align}
	\label{eq:claimratio1}
	q_{N_i - \tilde N_i} (x-z) & \le (1+\epsilon_N) \, q_{N_i - \tilde N_i} (x) \,, \\
	\label{eq:claimratio2}
	\sum_{| x' | \le \sqrt{N_i \log \frac{1}{\delta_N^2}}}
	\, q_{N_i - \tilde N_i} (x'-z) & \ge (1+\epsilon_N)^{-1} \,,
\end{align}
uniformly over $z,x \in \Z^2_\even$ that satisfy
(recall \eqref{eq:scales} and \eqref{eq:tildeNi})
\begin{equation}\label{eq:boundsxz}
	| z |  \le \sqrt{\tilde N_i \log \tfrac{1}{\delta_N^2}}
	= \sqrt{N_i  \big( \log \tfrac{1}{\delta_N^2} \big)^{- 2}} \,,
	\qquad | x | \le \sqrt{N_i \log \tfrac{1}{\delta_N^2}} \,.
\end{equation}

We first prove \eqref{eq:claimratio1}.
We fix $M\in\N$ and $i \in \{1,\ldots, M\}$.
By the local limit theorem \eqref{eq:llt}, uniformly over
for $x,z \in \Z^2_\even$, we can write
  \begin{equation*}
    \frac{q_{N_i - \tilde N_i} (x - z)}{q_{N_i - \tilde N_i} (x)} = \mathe^{\frac{| x |^2
    - | x - z |^2}{N_i - \tilde N_i} + O \left( \frac{| x |^4 + | x - z |^4}{(N_i - \tilde N_i)^3} \right)
    + o(1)}  \,.
  \end{equation*}
For large $N$, we have $N_i - \tilde N_i \ge \frac{1}{2} N_i$
and $N_i \ge (\delta_N^2)^{-\frac{1}{M}} \ge (\log \frac{1}{\delta_N^2})^3$,
see \eqref{eq:scales} and \eqref{eq:tildeNi}. Therefore by \eqref{eq:boundsxz}, we can bound
\begin{gather*}
	\frac{| x |^4 + | x - z |^4}{(N_i - \tilde N_i)^3}
	\le C \, \frac{(\log \frac{1}{\delta_N^2})^2}{N_i}
	\le \frac{C}{\log \frac{1}{\delta_N^2}} \,, \\
	\bigg| \frac{| x |^2 - | x - z |^2}{N_i - \tilde N_i} \bigg| =
	\bigg|\frac{2 \langle z, x \rangle - | z |^2}{N_i - \tilde N_i} \bigg|
	\le 2 \, \frac{2 | z | (| x | + | z |)}{N_i}
	\le \frac{C}{\sqrt{\log \frac{1}{\delta_N^2}}} \,.
\end{gather*}
Both estimates are $o(\epsilon_N)$ as $N\to\infty$, see \eqref{eq:epsN}, hence
\eqref{eq:claimratio1} follows. Finally, to prove \eqref{eq:claimratio2},
we note that the LHS equals
\begin{equation*}
\begin{split}
	\P\Big(|z + S_{N_i - \tilde N_i}|  \le \sqrt{N_i \log \tfrac{1}{\delta_N^2}} \,\Big)
	& \ge \P\Big(|S_{N_i}| \le \tfrac{1}{2} \sqrt{N_i \log \tfrac{1}{\delta_N^2}} \,\Big) \\
	& \ge 1 - C \, \exp \big( - \tfrac{1}{C} \big| \sqrt{ \log \tfrac{1}{\delta_N^2}} \big|^2 \big) \\
	& = 1 - C \, (\delta_N^{2})^C \ge (1+\epsilon_N)^{-1} \,,
\end{split}
\end{equation*}
where the last inequality holds for $N$ large enough, see \eqref{eq:epsN}.
\end{proof}

\smallskip
\subsection{Step 4: variance computation}
In this step, we compute the asymptotic variance
of $\Delta_{N_i}$ from \eqref{eq:Delta}, which is needed to prove \eqref{eq:CLT}
and \eqref{eq:LLN}.
We work under the conditional probability
$\bbP(\,\cdot\,|\cF_{N_{i-1}})$ and note from \eqref{eq:mui} and \eqref{eq:Delta} that
\begin{equation} \label{eq:meanzero}
	\bbE[\Delta_{N_i}|\cF_{N_{i-1}}] = 0 \,.
\end{equation}
Therefore we focus on the second moment.
In the next result, we exploit the control on the polymer distribution $\mu_{N,i-1}$
that we obtained in \eqref{eq:cM} in the previous step.

\begin{theorem}[Second moment asymptotics] \label{th:2mom}
For any $M\in\N$ and $i=1,\ldots, M$, we have
the a.s.\ convergence (uniformly over \blue{the disorder variables $(\omega(n, \cdot))_{n\leq N_{i-1}}$ that generate $\cF_{N_{i-1}}$})
\begin{equation}\label{eq:variance}
	\lim_{N \to \infty} \bbE[(\Delta_{N,i})^2\,|\,\cF_{N_{i-1}}]
	= \frac{1}{M} \, \frac{\rho}{1 + (1-\tfrac{i}{M}) \, \rho} \,.
\end{equation}
\end{theorem}

\begin{proof}
Recalling the definition \eqref{eq:Delta} of $\Delta_{N,i}$
and the representation \eqref{eq:ratiorepr}, we can write
\begin{equation} \label{eq:Deltarepr}
	\Delta_{N,i}
	=    \frac{1}{m_{N,i}} \,   \sum_{x \in
    \mathbb{Z}^2_{\even}} \mu_{N,i-1} (x)\,
\E \Big[
   \Big( \mathe^{\mathcal{H}^{\beta_N}_{(N_{i-1}, \tilde N_i]}} \,-\, 1 \Big)
    \, \ind_{\mathcal{D}_{\tilde N_i}
    \cap \mathcal{D}_{N_i}}  \, \Big|\,
    S_{N_{i-1}} =x \Big] \,.
\end{equation}

Removing the constraint $\ind_{\mathcal{D}_{\tilde N_i}\cap \mathcal{D}_{N_i}}$,
the RHS of \eqref{eq:Deltarepr} would simply become
\begin{equation} \label{eq:Deltatilderepr}
	\tilde{\Delta}_{N,i} :=
	\frac{1}{m_{N,i}} \, \pigl( Z_{\blue{N_{i-1}, \tilde{N}_i}}^{\beta_N}(\varphi_N) - 1 \pigr) \qquad
	\text{with} \qquad \varphi_N = \mu_{N,i-1} \,.
\end{equation}
We can now apply Theorem~\ref{th:vargen}, because
$L_N := \tilde{N}_i - N_{i-1}$ satisfies \eqref{eq:L} with $\ell = (1-\frac{i}{M}) \, \rho$,
(cf.~\eqref{eq:scales} and \eqref{eq:tildeNi}), while $\varphi_N =  \mu_{N,i-1}$
satisfies \eqref{eq:phi-cond0} with $W_N = N_{i-1}- \tilde{N}_{i-1}$
thanks to \eqref{eq:cM}, \eqref{eq:conditions-phi}, and \eqref{eq:ball-rw}, and \eqref{eq:W}
holds with $w = (1-\frac{i-1}{M}) \, \rho$.
Recalling from \eqref{eq:bound-cond-exp} that $m_{N,i} = 1 - o(1)$,
relation \eqref{eq:2nd-mom} then yields
\begin{equation} \label{eq:varDeltatilde}
	\lim_{N\to\infty} \bbE\big[ (\tilde{\Delta}_{N,i})^2 \big| \cF_{N_{i-1}} \big]
	= \lim_{N\to\infty} \bbvar\big[ Z_{\blue{N_{i-1}, \tilde{N}_i}}^{\beta_N}(\varphi_N) \big]
	= \frac{w-\ell}{1 + \ell} = \frac{\rho/M}{1 + (1-\frac{i}{M}) \, \rho} \,,
\end{equation}
which matches our goal \eqref{eq:variance}.
 It remains to show that removing
the constraint $\ind_{\mathcal{D}_{\tilde N_i}\cap \mathcal{D}_{N_i}}$ from \eqref{eq:Deltarepr}
is immaterial, that is, almost surely we have
$\bbE\big[ (\Delta_{N,i} - \tilde{\Delta}_{N,i})^2 \, \big| \, \cF_{N_{i-1}}\big] \to 0$.
To this end, we note that
\begin{align*}
\tilde{\Delta}_{N,i}-\Delta_{N,i}
&=  \frac{1}{m_{N,i}} \,   \sum_{x \in
    \mathbb{Z}^2_{\even}} \mu_{N,i-1} (x)\,
\E \Big[
   \Big( \mathe^{\mathcal{H}^{\beta_N}_{(N_{i-1}, \tilde N_i]}} \,-\, 1 \Big)
    \, \ind_{\mathcal{D}^c_{\tilde N_i}
    \cup \mathcal{D}^c_{N_i}}  \, \Big|\,  S_{N_{i-1}} =x \Big].
\end{align*}

Since in the expectation above, disorder is restricted to the time interval $(N_{i-1},\tilde N_i]$ with $N_{i-1}\ll \tilde N_i \ll N_i$, we can show that the contribution from the event $\mathcal{D}_{\tilde N_i}\cap \mathcal{D}^c_{N_i}$, which implies $|S_{N_i}| > \sqrt{ N_i \log \frac{1}{\delta_N^2}}$,
is negligible via an analysis similar to that performed in the proof of \eqref{eq:cM}, which uses the fact that the simple symmetric random walk has a
negligible probability of having super-diffusive displacement on the time interval $[\tilde N_i, N_i]$.

So we will focus on showing that, conditional on $\cF_{N,i}$, the $L^2$ norm of
\begin{align}\label{toestim}
\frac{1}{m_{N,i}} \,   \sum_{x \in
    \mathbb{Z}^2_{\even}} \mu_{N, i-1}(x)\,
\E \Big[
   \Big( \mathe^{\mathcal{H}^{\beta_N}_{(N_{i-1}, \tilde N_i]}} \,-\, 1 \Big)
    \, \ind_{\mathcal{D}^c_{\tilde N_i}}  \, \Big|\,  S_{N_{i-1}} =x \Big]
\end{align}
is negligible.
First, recalling  \eqref{eq:bound-cond-exp} we have that $m_{N, i-1} \ge 1 - C \, (\delta_N^2)^{\frac{1}{C}}$
and so we can neglect this term. Secondly, by \eqref{eq:cM}, we can bound $\mu_{N, i-1}(x)\leq C q_{N_{i-1}-\tilde N_{i-1}} (x)$.
Using the chaos expansion \eqref{eq:polynomial}, the $L^2$ norm of the sum in \eqref{toestim} can be bounded by a multiple of
\begin{align}\label{toestim2}
&\sumtwo{x,x',z,z' \in \mathbb{Z}^2_{\even}}{|y|,|y'| > ( \tilde N_i \log \delta_N^{-2} )^{1/2} }
\sum_{N_{i-1}\leq a<b \leq \tilde N_i}
 q_{N_{i-1}-\tilde N_{i-1}} (x) q_{N_{i-1}-\tilde N_{i-1}} (x') \,
    q_{a-N_{i-1}} (z-x) q_{a-N_{i-i}} (z-x') \notag\\
&   \hskip 4cm \times \sigma_{\beta_N}^2 \,
U_{\beta_N}(b-a,z'-z) \,\, q_{\tilde N_{i}-b} (y-z') q_{\tilde N_{i}-b} (y'-z')  \notag\\
=&\!\!\! \!\!\!\!\!\! \sumtwo{z,z' \in \mathbb{Z}^2_{\even}}{|y|,|y'| > ( \tilde N_i \log \delta_N^{-2} )^{1/2} }
\sum_{N_{i-1}\leq a<b \leq \tilde N_i}
 q_{a-\tilde N_{i-1}} (z)^2 \,\sigma_{\beta_N}^2U_{\beta_N}(b-a,z'-z) \,\, q_{\tilde N_{i}-b} (y-z') q_{\tilde N_{i}-b} (y'-z'),
\end{align}
where we used the Chapman-Kolmogorov equation to go from the first line to the second and we used the notation
\begin{equation} \label{eq:Ub}
\begin{split}
	U_{\beta}(n, z)
	& := \sum_{k \ge 1} (\sigma_{\beta}^2)^k
	\sumtwo{0 =: n_0 < n_1 < \cdots < n_k = n}{x_0 := 0, \
	x_1, \ldots, x_{k-1}\in \Z^2, x_k=z } \ \prod_{i = 1}^k q_{n_i - n_{i - 1}}(x_i - x_{i-1})^2 ,
\end{split}
\end{equation}
that is, the renewal function from \eqref{eq:U} with the end point pinned at $z$.
To bound \eqref{toestim2}, we distinguish between two cases: either
$ |z'| \leq \frac{1}{2} ( \tilde N_i \log \delta_N^{-2} )^{1/2}$, or $z'$ satisfies
the opposite inequality. In the first case, the decay from the random walk kernels will make the contribution
to \eqref{toestim2} negligible since $|y-z|, |y'-z|>\frac{1}{2}( \tilde N_i \log \delta_N^{-2} )^{1/2}$.
In the second case, we can drop the constraints on $y,y'$ in the sum to obtain that, the corresponding contribution to
\eqref{toestim2} is bounded (up to constants) by
\begin{align}\label{toestim3}
&\!\!\! \!\!\!\!\!\! \sumtwo{ z\in \mathbb{Z}^2_{\even}}{|z'| >\frac{1}{2} ( \tilde N_i \log \delta_N^{-2} )^{1/2} }
\sum_{N_{i-1}\leq a<b \leq \tilde N_i}
 q_{a-\tilde N_{i-1}} (z)^2 \,\sigma_{\beta_N}^2 U_{\beta_N}(b-a,z'-z).
\end{align}
In the sum above, we can identify $(a, z)$ with $(n_1, x_1)$ in \eqref{eq:Ub} and use \eqref{eq:Ub} to rewrite
\eqref{toestim3} and bound it by
\begin{align}\label{toestim4}
\sumtwo{0<b' \leq \tilde N_i-{\tilde N_{i-1}}}{ |z'| >\frac{1}{2} ( \tilde N_i \log \delta_N^{-2} )^{1/2} } U_{\beta_N}(b', z') \,\,\,\,\,
\leq \!\!\!\!
\sumtwo{0<b' \leq \tilde N_i}{|z'| >\frac{1}{2} ( \tilde N_i \log \delta_N^{-2} )^{1/2} } U_{\beta_N}(b',z'),
\end{align}
where in the inequality, we enlarged the range of summation for $b'$.
We can then use the next bound, which can be proved by following the same steps
as in \cite[Lemma 3.5]{CSZ23a}:
\begin{align*}
\sum_{x\in \bbZ^2} U_{\beta_N}(n,x) e^{\lambda |x|} \leq c e^{c\lambda^2 n} U_{\beta_N}(n).
\end{align*}
From this bound, the negligibility of \eqref{toestim4} as $N\to\infty$ follows easily by a Markov type inequality with
appropriate choices $\lambda$.
\end{proof}

\begin{remark}[Sub-critical regime]\label{rem:subcr2}
In the sub-critical regime \eqref{eq:sub-critical}, we need to modify \eqref{eq:variance} to:
\begin{equation}\label{eq:variance-subcritical}
	\lim_{N \to \infty} \bbE[(\Delta_{N,i})^2\,|\,\cF_{N_{i-1}}]
	= \frac{1}{M} \, \frac{\rho\, \hat\beta^2}{1 + (1-\tfrac{i}{M}) \, \rho\,\hat\beta^2} \,.
\end{equation}
The proof is the same, except that in \eqref{eq:varDeltatilde} we need to apply \eqref{eq:2nd-momsub} in place of \eqref{eq:2nd-mom}, see Remark~\ref{rem:subcr1}.
\end{remark}

\smallskip
\subsection*{Step 5. (Higher moment bounds)}

In this step, we control higher moments
of $\Delta_{N, i}$ defined in \eqref{eq:Delta},
proving the following bound
(recall the second moment computation \eqref{eq:variance}).

\begin{proposition}[High moment bounds] \label{prop:higher-moments}
For any $h\in\N$, there is a constant $\mathfrak{C}_h < \infty$ such that,
for any $M\in\N$ and $i=1,\ldots, M$,
we have the a.s.\ bound (uniformly over \blue{the disorder variables $(\omega(n, \cdot))_{n\leq N_{i-1}}$ that generate $\cF_{N_{i-1}}$}):
\begin{equation}\label{eq:moments}
	\limsup_{N \to \infty} \, \big| \,\bbE[(\Delta_{N, i})^h\,|\,\cF_{N_{i-1}}] \,\big|
	\le \mathfrak{C}_h \,
	\bigg(\frac{1}{M} \, \frac{\rho}{1+
	(1-\frac{i}{M})\rho}\bigg)^{\frac{h}{2}}  \,.
\end{equation}
\end{proposition}

Note that the case $h=1$ is trivial by \eqref{eq:meanzero},
while the case $h=2$ holds by \eqref{eq:variance}, and hence we focus on $h \ge 3$,
in which case the bound is a direct consequence of Theorem~\ref{th:genmombou}.

\begin{proof}
As in the proof Theorem~\ref{th:2mom}, see \eqref{eq:Deltarepr}, conditioned on $\cF_{N_{i-1}}$,
$\Delta_{N, i}$ can be written as a modified partition function where the random walk is restricted to
$\mathcal{D}_{\tilde N_i}\cap \mathcal{D}_{N_i}$. In \eqref{eq:Deltarepr}, we can ignore the mean
$m_{N,i} = 1 - o(1)$ (see \eqref{eq:bound-cond-exp}). Thanks to Theorem~\ref{th:genmombou} (whose
assumptions we check in a moment), we can obtain a moment upper bound by removing the random walk
restriction and applying \eqref{eq:genbound} to get
\begin{equation*}
	\big| \,\bbE[(\Delta_{N, i})^h\,|\,\cF_{N_{i-1}}] \,\big|
	\le \mathfrak{C}_h  \bbvar\big[Z_{L_N}^{\beta_N}(\varphi_N) \,\big|\, \cF_{N_{i-1}}
	\big]^{\frac{h}{2}} \,.
\end{equation*}
Our goal \eqref{eq:moments} then follows from \eqref{eq:variance},
since, as we showed in the proof of Theorem~\ref{th:2mom},
\begin{equation} \label{eq:varZLphi}
	\lim_{N \to \infty}
	\bbvar\big[Z_{L_N}^{\beta_N}(\varphi_N) \,\big|\, \cF_{N_{i-1}} \big]
	= \lim_{N \to \infty} \bbE[(\Delta_{N,i})^2\,|\,\cF_{N_{i-1}}]
	= \frac{1}{M} \, \frac{\rho}{1+
	(1-\frac{i}{M})\rho}\,.
\end{equation}

It only remains to check the assumptions of Theorem~\ref{th:genmombou}, namely
that $\beta_N$ from \eqref{eq:quasi-crit} and
$L_N = \tilde{N}_i - N_{i-1}$,
$\varphi_N = \mu_{N,i-1}$ from \eqref{eq:Deltatilderepr} fullfill conditions
\eqref{eq:concsqrtL}, \eqref{eq:locunif1} and \eqref{eq:varub0}.
\begin{itemize}
\item The bounded variance condition \eqref{eq:varub0} clearly holds by \eqref{eq:varZLphi}.

\item By \eqref{eq:cM}, we have $\varphi_N =  \mu_{N,i-1}  \le (1+\epsilon_N)^2
\, q_{N_{i-1} - \tilde{N}_{i-1}} \le 2\, q_{N_{i-1} - \tilde{N}_{i-1}}$, so
$\varphi_N$ is exponentially
concentrated on the scale $\sqrt{N_{i-1} - \tilde{N}_{i-1}} \sim \sqrt{N_{i-1}} $ much smaller than
$\sqrt{\tilde{N}_i} \sim \sqrt{L_N}$,
thus condition \eqref{eq:concsqrtL} holds by the local limit theorem \eqref{eq:llt}.

\item Again, $\varphi_N \le 2\, q_{N_{i-1} - \tilde{N}_{i-1}}$ yields
$\|\varphi_N\|_{\ell^\infty} \sim \frac{C}{N_{i-1} - \tilde{N}_{i-1}} \sim
\frac{C'}{\bbD[\varphi_N]}$, see \eqref{eq:bbD} and \eqref{eq:llt},
hence condition \eqref{eq:locunif1} is fullfilled.
\end{itemize}
The proof is complete.
\end{proof}

\begin{remark}[Sub-critical regime]\label{rem:subcr3}
	In the sub-critical regime \eqref{eq:sub-critical}, we need to modify \eqref{eq:moments} to:
	\begin{equation}\label{eq:moments-subcritical}
		\limsup_{N \to \infty} \, \big| \,\bbE[(\Delta_{N, i})^h\,|\,\cF_{N_{i-1}}] \,\big|
		\le \mathfrak{C}_h \,
		\bigg(\frac{1}{M} \, \frac{\rho \, \hat\beta^2}{1+
		(1-\frac{i}{M})\rho\,\hat\beta^2}\bigg)^{\frac{h}{2}} \, ,
	\end{equation}
	in agreement with the variance \eqref{eq:variance-subcritical}.
\end{remark}

\medskip
\subsection*{Step 6. Proof of  \eqref{eq:LLN} and \eqref{eq:CLT}}
We will apply the Central Limit Theorem for arrays of martingale differences.
In particular, we will make use of the following special version of \cite[Theorem 3.5]{HH80}.
\begin{theorem}
For each $n\geq 1$, let $(S_{n, i})_{1\leq i\leq M_n}$ be a mean zero square integrable martingale adapted to the filtration
$(\cF_{n, i})_{1\leq i\leq M_n}$. Let $\Delta_{n,i}:=S_{n,i}-S_{n, i-1}$ be the associated martingale differences.
Denote
\begin{align*}
V_{n,i}^2:=\sum_{j=1}^i \bbE\big[\Delta_{n,j}^2 \,|\, \cF_{n,j-1}\big]
\quad \text{and} \quad
U_{n,i}^2:=\sum_{j=1}^i \Delta_{n,j}^2 .
\end{align*}
Assume that
\begin{align}
&\qquad  \bbE\big[ \big| V_{n,M_n}^2 -\sigma^2 \big|\big]\xrightarrow[n\to\infty]{} 0, \quad \text{and} \quad
\max_{i\leq M_n}  \bbE\big[ \Delta_{n,i}^2 \,\big| \, \cF_{n,i-1} \big] \xrightarrow[n\to\infty]{\bbP} 0. \label{MG-CLT1}
\end{align}
Then the following three statements are equivalent:
\begin{align}
\text{(i)}&\quad \blue{\forall \eps>0,} \ \sum_{i\leq M_n} \bbE\big[ \Delta_{n,i}^2 \, \ind_{|\Delta_{n,i}|>\epsilon}\big] \xrightarrow[n\to\infty]{} 0,
\label{HH1}\\
\text{(ii)}&  \quad \qquad\bbE\big[ \big| U_{n,M_n}^2-\sigma^2\big| \big]\xrightarrow[n\to\infty]{} 0,
\label{HH2}\\
\text{(iii)} & \quad S_n:=\sum_{i=1}^{M_n} \Delta_{n,i} \xrightarrow[n\to\infty]{d} \cN(0,\sigma^2).
\label{HH3}
\end{align}
\end{theorem}
In our setting,
the first condition in \eqref{MG-CLT1} follows from Theorem~\ref{th:2mom} (note the uniformity 
in the convergence therein over \blue{the disorder variables $(\omega(n, \cdot))_{n\leq N_{i-1}}$ that 
generate $\cF_{N_{i-1}}$}) by choosing $M_N\to\infty$ slowly enough and then applying
a Riemann sum approximation that shows $\sigma^2 = \int_0^\rho \frac{\dd t}{1+ t}= \log (1+\rho)$.

The second condition in \eqref{MG-CLT1} follows from the higher moment estimate
\eqref{eq:moments} and a union bound as follows:
\begin{align*}
\bbP\Big( \max_{i\leq M_N}  \bbE\big[ \Delta_{N,i}^2 \,\big| \, \cF_{N,i-1} \big] >\epsilon \Big)
&\leq
\sum_{i=1}^{M_N} \bbP\Big(  \bbE\big[ \Delta_{N,i}^2 \,\big| \, \cF_{N,i-1} \big] >\epsilon \Big) \\
& \leq \frac{1}{\epsilon^2} \sum_{i=1}^{M_N} \bbE\big[ \Delta_{N,i}^4 \big]
= O\Big( \frac{1}{M_N}\Big),
\end{align*}
where the last equality is justified by choosing $M_N \to \infty$ slowly enough.
Indeed, by \eqref{eq:moments}, for any fixed $M < \infty$, there is $N = \cN_M$ large enough such that
\begin{equation*}
	\forall N \ge \cN_M,  \ \forall 1\leq i\leq M \colon \qquad
	\bbE[(\Delta_{N,i})^4\,|\,\cF_{N_{i-1}}] \le \mathfrak{C}_4 \, \bigg(\frac{2\rho}{M}\bigg)^2 \,.
\end{equation*}
We can assume that $\lim_{M\to\infty} \cN_M = \infty$. Hence
we can choose a sequence $(M_N)_{N\to\infty}$ with $M_N \to \infty$ slowly enough
such that $N \ge \cN_{M_N}$, and hence
\begin{equation*}
	\sum_{i=1}^{M_N} \bbE[(\Delta_{N,i})^4\,|\,\cF_{N_{i-1}}] \le \mathfrak{C}_4 \, M_N\, \bigg(\frac{2\rho}{M_N}\bigg)^2 = O\Big( \frac{1}{M_N}\Big)\,.
\end{equation*}

Condition \eqref{HH1} follows in the same way via Chebyshev's inequality and the higher moment
estimate \eqref{eq:moments}. In turn, this implies \eqref{HH2} and \eqref{HH3}, which are respectively
our desired relation \eqref{eq:LLN} and \eqref{eq:CLT}.

\medskip

\subsection*{Step 7. Proof of  \eqref{eq:remainder}}
To show the negligibility of the second term in \eqref{eq:remainder}, i.e.,
$\sum_{i=1}^{M_N}\log m_{N,i}$, we apply \eqref{eq:bound-cond-exp} to obtain
\begin{align*}
0\leq - \sum_{i=1}^{M_N}\log m_{N,i} \leq - \sum_{i=1}^{M_N} \log (1-C (\delta_N^2)^{\frac{1}{C}})
\le C' M_N (\delta_N)^{\frac{1}{C}},
\end{align*}
which goes to $0$ if $M_N\to\infty$ slowly enough such that $M_N \ll (\delta_N)^{-\frac{1}{C}}$.

For the first term in \eqref{eq:remainder},
 recall that $r(x)=\log(1+x)-\big(x-\frac{x^2}{2} \big)$. \blue{We may restrict to the event $G_N:= \cap_{i=1}^{M_N} \{ |\Delta_{N, i}|\leq 1/2\}$, 
 since
\begin{equation}\label{eq:GNc}
 \bbP(G_N^c) \leq \sum_{i=1}^{M_n} \bbP(|\Delta_{N, i}|> 1/2) \leq 4 \sum_{i=1}^{M_n} \bbE[\Delta_{N, i}^2 1_{\{ |\Delta_{N, i}|>1/2\}}] \asto{N} 0
\end{equation}
by \eqref{HH1} shown above. On the event $G_N$,} we use the elementary estimate
$|r(x)|\leq C |x|^3 \wedge x^2$ for $|x|\leq 1/2$ to obtain
\begin{align*}
\Big| \sum_{i=1}^{M_N} \bbE \big[r(\Delta_{N,i})\big] \Big|
&\leq C \sum_{i=1}^{M_N} \Big(\bbE \big[\, |\Delta_{N,i}|^2 \, \ind_{|\Delta_{N,i}| > \epsilon } \, \big]+\bbE \big[ \,|\Delta_{N,i}|^3 \, \ind_{|\Delta_{N,i}| \leq \epsilon } \,\big]\Big) \\
& \leq C \sum_{i=1}^{M_N} \Big( \bbE \big[\, \Delta_{N,i}^2 \, \ind_{|\Delta_{N,i}| > \epsilon } \, \big]+\epsilon \bbE \big[ \,\Delta_{N,i}^2 \,\big]\Big),
\end{align*}
which converges to $0$ by \eqref{HH1} and \eqref{HH2} and by letting $\epsilon$ be arbitrarily small.

The proof of \eqref{eq:log-normal2} is now complete.\qed

\smallskip
\section{Higher moment bounds}
\label{sec:moments}

We prove a strengthened version of the general moment bound in Theorem~\ref{th:genmombou}.
We consider the averaged \emph{point-to-point} partition function
$Z_L(\varphi, \psi)$ defined in \eqref{eq:Z}: for $\varphi, \psi: \Z^2 \to \R$ and $L \in \N$,
\begin{equation}\label{eq:Z2}
	Z_L^\beta(\varphi,\psi) := \!\!
	\sum_{z,w \in \Z^2} \! \varphi(z) \, Z_L^\beta(z,w) \, \psi(w)
	\qquad \text{with} \quad
		Z_L^{\beta}(z,w)
	:= \E \pig[\rme^{\mathcal{H}_{(0,  L)}^{\beta,\omega}(S)}
	\, \ind_{\{S_L=w\}}  \pig| \nobracket S_0 = z \pig] \,.
\end{equation}

Let $\varphi$ be a probability mass function on $\Z^2$ satisfying the localization
condition \eqref{eq:concsqrtL}, which we recall here for convenience: for some $\,\hat{t}\, > 0$, $\sfc_1 < \infty$
\begin{equation} \label{eq:concsqrtL+}
	\exists z_0 \in \R^2 \colon \qquad
	\sum_{z\in\Z^2} \varphi(z) \, \rme^{2\,\hat{t}\, \frac{|z - z_0|}{\sqrt{L}}}
	\le \sfc_1  \,.
\end{equation}
Instead of \eqref{eq:locunif1}, we impose the weaker condition that,
for some $\sfc_2' < \infty$,
\begin{equation}\label{eq:locunif}
	\frac{\log\bigl( L\,
	\| \varphi\|_{\ell^2}^2
	\bigr)}{R_L(\varphi,\varphi)} \le \sfc_2' \,,
\end{equation}
where $R_L(\varphi,\varphi)$ is defined in \eqref{eq:Rfg}.
The fact that \eqref{eq:locunif} is indeed implied by \eqref{eq:locunif1} (when \eqref{eq:concsqrtL+} holds)
is shown in the next result, proved in Appendix~\ref{sec:locunif}.

\begin{lemma}\label{th:locunif}
If a probability mass function $\varphi$ fullfills \eqref{eq:concsqrtL+}, there exists
$\mathfrak{c} = \mathfrak{c}(\,\hat{t}\,, \sfc_1) > 0$ such that
\begin{gather}
	\label{eq:rphibound0}
	R_L(\varphi,\varphi) \ge
	R_{L/2}(\varphi,\varphi) \ge \mathfrak{c} \, \log\pig(1 + \tfrac{L/2}{1+4 \, \bbD[\varphi]}\pig) \,,
\end{gather}
and hence  condition \eqref{eq:locunif1} implies condition \eqref{eq:locunif}.
\end{lemma}

We also require condition \eqref{eq:varub0}, that is boundedness of the variance
of the \emph{point-to-plane partition function}
$Z_L^\beta(\varphi) = Z_L^\beta(\varphi,1)$, which we recall here for convenience:
\begin{equation}\label{eq:varub0+}
	\bbvar[Z_{L}^\beta(\varphi)]
	\le \sfc_3 \,.
\end{equation}
We stress that we impose no assumption on $\psi$ in $Z_L^\beta(\varphi,\psi)$.

We can now state our strengthened moment bound, which
generalises Theorem~\ref{th:genmombou}.

\begin{theorem}[Strengthened general moment bound]\label{th:genmombou+}
Given $h\in\N$ and $\,\hat{t}\,$, $\sfc_1$, $\sfc_2'$, $\sfc_3 \in (0,\infty)$,
there exist constants $L_h, \mathfrak{C}_h  < \infty$
(depending also on $\,\hat{t}\,, \sfc_1, \sfc_2', \sfc_3$) such that
\begin{equation}\label{eq:genbound2}
	\big| \bbE \big[ \big( Z_L^\beta(\varphi,\psi) - \bbE[Z_L^\beta(\varphi,\psi)] \big)^h \big] \big|
	\,\le\, \mathfrak{C}_h
	\bbvar[Z_{L}^\beta(\varphi)]^{\frac{h}{2}}
	\, \bigg\| \psi(\cdot)\, \rme^{-\frac{\,\hat{t}\, }{2}\frac{|\,\cdot\, - z_0| }{\sqrt{L}}}
	\bigg\|_{\ell^\infty}^h
\end{equation}
uniformly for $\beta \in [0,\beta_0]$, $L\ge L_h$, for probability mass functions $\varphi$ and $z_0\in \Z^2$ satisfying
\eqref{eq:concsqrtL+}, \eqref{eq:locunif} and \eqref{eq:varub0+}, and for arbitrary function~$\psi$. Furthermore:
\begin{itemize}
\item $z_0$ in \eqref{eq:genbound2} (from \eqref{eq:concsqrtL+}) can be replaced by the mean
$m_\varphi$ of $\varphi$ (see \eqref{eq:bbD});

\item the bound \eqref{eq:genbound2} still holds if, on the LHS, we replace $Z_L^\beta(\varphi, \psi)$ by its
restriction to any subset of random walk paths in its definition \eqref{eq:Z2}.
\end{itemize}
\end{theorem}

The rest of this section is devoted to proving Theorem~\ref{th:genmombou+}.

\subsection{Preliminary lemmas}
We collect here some technical lemmas that will be useful in the proof.
We first show that, for any probability mass function $\varphi$ satisfying
condition \eqref{eq:concsqrtL+}, we have a uniform
lower bound on $R_{L/2}(\varphi, \varphi)$.
The proof is given in Appendix~\ref{sec:boundsphi}.

\begin{lemma}\label{th:boundsphi}
Given $\hat{t}, \sfc_1 \in (0,\infty)$, there exists $\eta > 0$
(depending on $\hat{t}, \sfc_1$) such that, for any probability mass function $\varphi$
satisfying \eqref{eq:concsqrtL+} with constants $\hat{t}$ and $\sfc_1$, we have
\begin{equation}\label{eq:Rub}
	R_{L/2}(\varphi, \varphi) \ge \eta \,.
\end{equation}
\end{lemma}

We next show that the assumptions of Theorem~\ref{th:genmombou+}
force $\beta$ to be \emph{at most critical}. To this purpose, recalling \eqref{eq:RL} and \eqref{eq:betacrit}, we have the equivalence for any $\theta\in\R$, as $L\to\infty$,
\begin{equation} \label{eq:equivalence}
	\sigma_\beta^2 = \frac{1}{R_L - \theta + o(1)}
	\qquad \iff \qquad
	\sigma_\beta^2 = \frac{1}{R_L} \bigg(1+\frac{\pi\, \theta + o(1)}{\log L}\bigg)\,.
\end{equation}
The next result is proved in Appendix~\ref{sec:sigmatheta}.

\begin{lemma}\label{th:sigmatheta}
Given $\hat{t}, \sfc_1, \sfc_3 \in (0,\infty)$, there exists $\bar\theta  \in [0,\infty)$
(depending on $\hat{t}, \sfc_1, \sfc_3$) such that, for any $\beta \ge 0$,
$L \in \N$ and any probability mass function $\varphi$ satisfying \eqref{eq:concsqrtL+} and \eqref{eq:varub0+}
with constants $\hat{t}, \sfc_1$ and $\sfc_3$, we have (recall $\sigma_\beta^2$ from \eqref{eq:sigmabeta})
\begin{equation} \label{eq:sigmatheta}
	\sigma_\beta^2 \le \frac{1}{R_L - \bar\theta} \,.
\end{equation}
\end{lemma}

We finally define an exponentially dampened version of $R_L$ from \eqref{eq:RL2}:
\begin{equation} \label{eq:Rlambda}
	R_L^{(\hat\lambda)} := \sum_{n=1}^L \rme^{-\hat\lambda \frac{n}{L}} \, q_{2n}(0)
	\qquad \text{for } \hat\lambda \ge 0 \,,
\end{equation}
We will use this quantity to give a proxy for the second moment $\bbE\big[Z_{L/2}^{\beta}(0)^2\big]$, as shown in the next result, proved in Appendix~\ref{sec:2lb}.

\begin{lemma} \label{th:2lb}
Recall the constant  $\mathfrak{a}_+$ from \eqref{eq:barN}.
If $\beta \ge 0$, $L\in\N$ and $\hat\lambda \ge 0$ satisfy
\begin{equation} \label{eq:cond-for-LB}
	\sigma_\beta^2 \le \frac{1}{R_{L}^{(\hat\lambda)} + 4 \, \mathfrak{a}_+} \,,
\end{equation}
then
\begin{equation} \label{eq:LBvar}
	\bbE\big[Z_{L/2}^{\beta}(0)^2\big] \ge
	\frac{1}{2} \, \frac{1}{1-\sigma_\beta^2 \, R_{L}^{(\hat\lambda)}} \,.
\end{equation}
\end{lemma}

In order to achieve condition \eqref{eq:cond-for-LB}, starting from \eqref{eq:sigmatheta}, it is enough to take $\hat\lambda \ge 0$ large enough, as we show in the next elementary result, proved in Appendix~\ref{sec:choice}.

\begin{lemma} \label{th:choice}
Recall the constant $\mathfrak{a}_-$ from \eqref{eq:barN}.
For any $0\leq \hat\lambda \leq L$, we have
\begin{equation} \label{eq:choice}
	R_L^{(\hat\lambda)} \le
	R_L - \mathfrak{a}_- \, \log \tfrac{\hat\lambda}{2} \,.
\end{equation}
\end{lemma}

We are now ready to describe the strategy of the proof of Theorem~\ref{th:genmombou+}.

\subsection{A general estimate}
We bound the moments of the partition function exploiting the \emph{functional operator approach} developed in \cite{CSZ23a,LZ23,CCR23}.
The following general estimate is extracted from \cite[Section~4]{CCR23} (see Appendix~\ref{sec:addpr} for the details). A comparison with the original bound from \cite{CSZ23a} is discussed in Remark~\ref{rem:comparisonCSZ}.

\begin{theorem}[\cite{CCR23}] \label{th:moment-bound-CCR}
Fix any exponent $h\in\N$, $h \ge 3$, system size $L\in\N$ and
coupling constant $\beta > 0$ small enough,
say $\beta \le \beta_0$ for a suitable $\beta_0 = \beta_0(h) > 0$.
Given $\hat t > 0$ and $\hat\lambda \ge 0$,
there are constants $\mathtt{K}_h^{(\hat t)}, \,
\cC_h^{(\hat t,\hat \lambda)} < \infty$ such that,
assuming $\sigma_\beta^2 \, R_L^{(\hat\lambda)} < 1$ and
defining
\begin{equation} \label{eq:Gammapar}
	\Gamma = \Gamma_{h, \beta ,L}^{(\hat t,\hat\lambda)} :=
	\mathtt{K}_{h}^{(\hat t)} \,
	\frac{\sigma_\beta^2}{1-\sigma_{\beta}^2 \, R_{L}^{(\hat\lambda)}} \,,
\end{equation}
the following bound holds for any $1 < p,q < \infty$ with $\frac{1}{p}+\frac{1}{q} = 1$ and any functions $\varphi, \psi$ on $\Z^2$:
\begin{gather} \label{eq:genbo}
	\begin{split}
		\big| \bbE \big[ \big( Z_{L}^{\beta}(\varphi,\psi)- \bbE[Z_{L}^{\beta}(\varphi,\psi)] \big)^h \big] \big|
		\le & \ \cC_h^{(\hat t,\hat \lambda)}
		 \sum_{r=1}^\infty
		\bigl( p\, q \, \Gamma
		\bigr)^{\max\{r, \frac{h}{2}\}}
		\\
		&  \times \pigl\| \varphi(\cdot) \, \rme^{\frac{\hat t}{\sqrt{L}}|\cdot|} \pigr\|_{\ell^1}
		\, \pigl\| \varphi(\cdot) \, \rme^{\frac{\hat t}{\sqrt{L}}|\cdot|} \pigr\|_{\ell^p}^{h-1} \\
		&  \times  \pigl\|  \psi(\cdot) \, \rme^{-\frac{\hat t}{\sqrt{L}}|\cdot|} \pigr\|_{\ell^\infty}
		\, \pigl\|  \psi(\cdot) \, \rme^{-\frac{\hat t}{\sqrt{L}}|\cdot|} \pigr\|_{\ell^q}^{h-1} \,.
	\end{split}
	\end{gather}
The series converges iff $p\, q  \,\Gamma < 1$, in which case we can bound it by
\begin{equation} \label{eq:geosu}
		\sum_{r=1}^\infty \bigl( p\, q \, \,
		\Gamma \bigr)^{\max\{r, \frac{h}{2}\}}
		\le \rme^h \, \frac{\bigl(p\,q\,\Gamma\bigr)^{\frac{h}{2}}}{1-\,p\,q\,\Gamma}
		\qquad \text{if}
		\quad  0 < p\, q  \,\Gamma < 1 \,.
\end{equation}
	\end{theorem}

We will deduce our goal \eqref{eq:genbound2} from \eqref{eq:genbo}. We first need to have a suitable control on the second and third lines of \eqref{eq:genbo} in order to fit our assumption on $\varphi$, see in particular \eqref{eq:concsqrtL+} and \eqref{eq:locunif}. Since we will be interested in taking $q$ large, we may assume
\begin{equation} \label{eq:pq}
	q \in [4,\infty) \,.
\end{equation}
We will need the following basic interpolation result.

\begin{lemma}
Fix $p \in (1,2)$ and $q \in (2,\infty)$
with $\frac{1}{p} + \frac{1}{q} = 1$.
For any $f,g$ on $\Z^2$,
\begin{equation} \label{eq:interpola}
	\| f \, g \|_{\ell^p} \le \pigl\|f \, g^{\frac{q}{q-2}} \pigr\|_{\ell^1}^{1-\frac{2}{q}}
	\, \| f \|_{\ell^2}^{\frac{2}{q}} \,.
\end{equation}
\end{lemma}

\begin{proof}
We write $f^p = f^{1-\alpha} \, f^{2 \alpha}$, with
$\alpha = p-1 = \frac{p}{q} \in (0, \frac{1}{2})$, and apply Hölder to get
\begin{equation*}
	\bigl\| f \, g \bigr\|_{\ell^p}^p
	= \sum_{z\in\Z^2} |f(z)|^p \, |g(z)|^p
	\le \bigg( \sum_{z\in\Z^2} |f(z)| \, |g(z)|^{\frac{p}{1-\alpha}} \bigg)^{1-\alpha}
	\bigg( \sum_{z\in\Z^2} f(z)^2 \bigg)^\alpha
\end{equation*}
which coincides with \eqref{eq:interpo} since $\frac{1-\alpha}{p} = 1 - \frac{2}{q}$
and $\frac{\alpha}{p} = \frac{1}{q}$.
\end{proof}

Let us look back at the second and third lines of \eqref{eq:genbo}.
We apply \eqref{eq:interpola} with $f=\varphi$
and $g(\cdot) = \rme^{\frac{\hat t}{\sqrt{L}}|\cdot|}$: for $q \ge 4$ we have $\frac{q}{q-2} \le 2$, hence
\begin{equation} \label{eq:interpo}
	\pigl\| \varphi(\cdot) \, \rme^{\frac{\hat t}{\sqrt{L}}|\cdot|} \pigr\|_{\ell^p}
	\le \pigl\| \varphi(\cdot) \, \rme^{\frac{2 \hat t}{\sqrt{L}}|\cdot|} \pigr\|_{\ell^1}^{1 - \frac{2}{q}}
	\, \big\| \varphi \big\|_{\ell^2}^{\frac{2}{q}}
	\le \pigl\| \varphi(\cdot) \, \rme^{\frac{2 \hat t}{\sqrt{L}}|\cdot|} \pigr\|_{\ell^1}
	\, \big\| \varphi \big\|_{\ell^2}^{\frac{2}{q}} \,.
\end{equation}
We next make the simple estimate that, for some $c < \infty$,
\begin{equation}\label{eq:interpsi}
	\pigl\|  \psi(\cdot) \, \rme^{-\frac{\hat t}{\sqrt{L}}|\cdot|} \pigr\|_{\ell^q}
	\le \pigl\|  \psi(\cdot) \, \rme^{-\frac{\hat t}{2\sqrt{L}}|\cdot|} \pigr\|_{\ell^\infty}
	\Biggl(\sum_{z\in\Z^2} \rme^{-\frac{\hat t}{2\sqrt{L}} q|z|} \Biggr)^{\frac{1}{q}}
	\le \pigl\|  \psi(\cdot) \, \rme^{-\frac{\hat t}{2\sqrt{L}}|\cdot|} \pigr\|_{\ell^\infty}
	\, \bigg( \frac{c L}{\hat t^2} \bigg)^{\frac{1}{q}}
\end{equation}
because
$\sum_{z\in\Z^2} \rme^{-\frac{s}{2} q|z|} \le \sum_{z\in\Z^2} \rme^{- \frac{s}{2} |z|}	\le \frac{c}{s^2}$.
Plugging these estimates into \eqref{eq:genbo}, as well as the bound \eqref{eq:geosu} with $p \le 2$, we obtain the following corollary of Theorem~\ref{th:moment-bound-CCR}.

\begin{proposition} \label{th:mb}
	For any $h\in\N$, $h \ge 3$, there is $\beta_0 = \beta_0(h) > 0$ such that for any $\beta \in [0,\beta_0]$ and $L\in\N$ the following holds.
	Given $\hat t > 0$ and $\hat\lambda \ge 0$,
	there are constants $\mathtt{K}_h^{(\hat t)} ,\,
	\cC_h^{(\hat t,\hat \lambda)} < \infty$ such that,
	recalling $\Gamma$ from \eqref{eq:Gammapar}, for any
	\begin{equation} \label{eq:2qGamma}
		q \ge 4 \quad \text{such that} \quad 0 < 2q \, \Gamma < 1 \,,
	\end{equation}
we can bound, for any functions $\varphi, \psi$ on $\Z^2$,	
	\begin{gather}
		\label{eq:moment-bound-impr}
	\begin{split}
		\big| \bbE \big[ \big( Z_{L}^{\beta}(\varphi,\psi)- \bbE[Z_{L}^{\beta}(\varphi,\psi)] \big)^h \big] \big|
		\le \, & \cC_h^{(\hat t,\hat\lambda)} \,
		\frac{(2q\, \Gamma)^{\frac{h}{2}}}{1- 2q\, \Gamma}
		\, \big( L \, \| \varphi \|_{\ell^2}^2 \big)^{\frac{h-1}{q}} \\
		& \times  \pigl\| \varphi(\cdot) \, \rme^{\frac{2\hat t}{\sqrt{L}}|\,\cdot\, |} \pigr\|_{\ell^1}^h
		\,  \pigl\|  \psi(\cdot) \, \rme^{-\frac{\hat t}{2\sqrt{L}}|\,\cdot\, |} \pigr\|_{\ell^\infty}^h  \,.
	\end{split}
	\end{gather}
	\end{proposition}

In the next subsection we will show that our goal \eqref{eq:genbound2} follows by \eqref{eq:moment-bound-impr}.

\begin{remark}[Comparison with \cite{CSZ23a}]\label{rem:comparisonCSZ}
From the proof of \cite[Theorem~6.1]{CSZ23a} one can extract a bound similar to \eqref{eq:genbo}, but with a different dependence on the boundary conditions $\varphi, \psi$, namely the second and third lines of \eqref{eq:genbo} are replaced by\footnote{The factors $L^{\frac{1}{q}}$, $L^{\frac{1}{p}}$ in \eqref{eq:comparisonCSZ} arise from operator norms, see \cite[Proposition~6.6]{CSZ23a}, while $\frac{1}{L}$ is due to an averaging on the system size from $L$ to~$2L$ which is performed in the proof.}
\begin{equation} \label{eq:comparisonCSZ}
\begin{split}
	L^{\frac{1}{q}}
	\, \pigl\| \varphi(\cdot) \, \rme^{\frac{\hat t}{\sqrt{L}}|\cdot|} \pigr\|_{\ell^p}^{h}
	\times  \frac{L^{\frac{1}{p}}}{L}
	\, \|\psi\|_\infty^h \, \pigl\| \rme^{-\frac{\hat t}{\sqrt{L}}|\cdot|} \pigr\|_{\ell^q}^{h} =
	\, \pigl\| \varphi(\cdot) \, \rme^{\frac{\hat t}{\sqrt{L}}|\cdot|} \pigr\|_{\ell^p}^{h}
	\, \|\psi\|_\infty^h \, \pigl\| \rme^{-\frac{\hat t}{\sqrt{L}}|\cdot|} \pigr\|_{\ell^q}^{h} \,.
\end{split}
\end{equation}
The bound \eqref{eq:genbo} is better for two reasons:
\begin{itemize}
	\item the quantity $\pigl\|  \psi(\cdot) \, \rme^{-\frac{\hat t}{\sqrt{L}}|\cdot|} \pigr\|_{\ell^q}$ is smaller than $\|\psi\|_\infty \, \pigl\| \rme^{-\frac{\hat t}{\sqrt{L}}|\cdot|} \pigr\|_{\ell^q}$ and it allows for unbounded functions~$\psi$;
	\item the power $h-1$ in \eqref{eq:genbo} is better than $h$ in \eqref{eq:comparisonCSZ} in case $D = \bbD[\varphi] \ll L$, see \eqref{eq:bbD}: for instance, if $\varphi$ is a probability mass function supported on a ball of radius $\sqrt{D}$ with $\|\varphi\|_\infty = O(\frac{1}{D})$ and $\psi \equiv 1$ is constant (for simplicity), we have for some $c>0$
	\begin{equation*}
		\pigl\| \varphi(\cdot) \, \rme^{\frac{\hat t}{\sqrt{L}}|\cdot|} \pigr\|_{\ell^p} \ge c \,
		\frac{D^{\frac{1}{p}}}{D}
		= c \, D^{-\frac{1}{q}} \,, \qquad
		\pigl\|  \psi(\cdot) \, \rme^{-\frac{\hat t}{\sqrt{L}}|\cdot|} \pigr\|_{\ell^q} \ge c \,
		L^{\frac{1}{q}} \,,
	\end{equation*}
	hence \eqref{eq:comparisonCSZ} is larger than \eqref{eq:genbo} by a factor $\ge c^2 \, (\frac{L}{D})^{\frac{1}{q}} \gg 1$.
\end{itemize}
\end{remark}

\subsection{Proof of Theorem~\ref{th:genmombou+}}

Given $h\in\N$ and constants $\hat{t}$, $\sfc_1$, $\sfc_2'$, $\sfc_3 \in (0,\infty)$, we need to prove the bound \eqref{eq:genbound2} for all $\beta\geq 0$ and $L$ large enough, uniformly over probability mass functions $\varphi$ satisfying \eqref{eq:concsqrtL+}, \eqref{eq:locunif}, \eqref{eq:varub0+} and over arbitrary functions $\psi$ on $\Z^2$.
For simplicity, we assume $z_0 = 0$ (it suffices to replace $\varphi$ by $\varphi(\,\cdot + z_0)$ and likewise for $\psi$).

We will deduce \eqref{eq:genbound2} from \eqref{eq:moment-bound-impr}.
To apply Proposition~\ref{th:mb}, we note that $\beta \le \beta_0(h)$ is guaranteed by \eqref{eq:sigmatheta} if we take $L \ge L_h$ for $L_h$ large enough (depending on $h$ and $\bar\theta$ in \eqref{eq:sigmatheta}, hence on $\hat{t}, \sfc_1, \sfc_3$).
We show below that condition \eqref{eq:2qGamma} can be satisfied as follows:
\begin{aenumerate}
	\item\label{it:a} we first fix $\hat\lambda \ge 0$ such that $\Gamma \in (0, \frac{1}{16}]$;
	\item\label{it:b} then we pick $q \ge 4$ with
	$2q \, \Gamma \le \frac{1}{2}$ (to discard the denominator $1 - 2q \, \Gamma \ge \frac{1}{2}$ in \eqref{eq:moment-bound-impr}).
\end{aenumerate}
We can now apply the bound \eqref{eq:moment-bound-impr}. We show below that
\begin{aenumerate}\setcounter{enumi}{2}
	\item\label{it:c} for suitable constants $C, C' < \infty$
	(depending on $h$, $\hat{t}$, $\sfc_1$, $\sfc_2'$, $\sfc_3$),
	\begin{equation}\label{eq:2bounds}
		2q \, \Gamma \le C \, \bbvar[Z_{L}^\beta(\varphi)] \,, \qquad
		\big( L \, \| \varphi \|_{\ell^2}^2 \big)^{\frac{1}{q}} \le C' \,.
	\end{equation}
	\end{aenumerate}
Plugging these bounds into \eqref{eq:moment-bound-impr}, in view of \eqref{eq:concsqrtL+}, we see that the RHS of \eqref{eq:genbo} gives precisely \eqref{eq:genbound2} with the constant
\begin{equation} \label{eq:costante}
	\mathfrak{C}_h = 2 \, \cC_h^{(\hat t,\hat\lambda)} \, C^{\frac{h}{2}} \, (C')^{h-1} \, \sfc_1^h
\end{equation}

We complete the proof of Theorem~\ref{th:genmombou+} proving the steps \eqref{it:a}, \eqref{it:b} and \eqref{it:c}.

\medskip

\noindent
\emph{Step \eqref{it:a}.}
To find the desired $\hat\lambda \ge 0$, we note that by \eqref{eq:Gammapar} we have the equivalence
\begin{equation} \label{eq:condone}
	0 < \Gamma \le \tfrac{1}{16} \qquad \iff \qquad
	\sigma_\beta^2 \le \frac{1}{R_L^{(\hat\lambda)} + 16 \,\mathtt{K}_{h}^{(\hat t)}} \,.
\end{equation}
We will later need the similar condition \eqref{eq:cond-for-LB}.
Both conditions hold if we take $\hat\lambda$ large enough,
thanks to Lemmas~\ref{th:sigmatheta} and~\ref{th:choice}: more explicitly, by \eqref{eq:sigmatheta} and \eqref{eq:choice},
we can fix
\begin{equation}\label{eq:hatlambdapick}
	\hat\lambda := 2 \, \rme^{2\pi \left(\bar\theta + \max\left\{ 16 \, \mathtt{K}_{h}^{(\hat t)}, 4 \, \mathfrak{a}_+\right\}\right)} \,.
\end{equation}

\medskip

\noindent
\emph{Step \eqref{it:b}.}
We will see below that it would be convenient to take $q \approx R_{L/2}(\varphi, \varphi)$.
To ensure that $q \ge 4$ and also $2q \, \Gamma \le \frac{1}{2}$
i.e.\ $q \le \frac{1}{4\,\Gamma}$, since $R_{L/2}(\varphi,\varphi) \ge \eta$ by \eqref{eq:Rub}, we define
\begin{equation} \label{eq:qfix}
	q := \min\bigg\{ 4\, \frac{R_{L/2}(\varphi,\varphi)}{\eta} \,, \,
	\frac{1}{4 \, \Gamma} \bigg\} \,.
\end{equation}
In particular, we have by definition
\begin{equation} \label{eq:qGamma}
	q \, \Gamma \le \frac{4}{\eta}
	\, R_{L/2}(\varphi,\varphi) \,
	\Gamma \,.
\end{equation}

\medskip

\noindent
\emph{Step \eqref{it:c}}. Let us prove the
first relation in \eqref{eq:2bounds}.
By definition of $\Gamma$, see \eqref{eq:Gammapar}, and by
Lemma~\ref{th:2lb} (note that condition \eqref{eq:cond-for-LB}
is ensured by our choice of $\hat\lambda$, \blue{which is possible thanks to \eqref{eq:sigmatheta} and \eqref{eq:choice}}), we can bound
\begin{equation*}
	R_{L/2}(\varphi,\varphi) \, \Gamma
	\le \mathtt{K}_{h}^{(\hat t)}
	\, R_{L/2}(\varphi,\varphi) \,
	\frac{\sigma_\beta^2}{1-\sigma_{\beta}^2 \, R_{L}^{(\hat\lambda)}}
	\le 2\, \mathtt{K}_{h}^{(\hat t)}
	\, R_{L/2}(\varphi,\varphi) \,
	\sigma_\beta^2 \,\,
	\bbE\big[Z_{L/2}^{\beta}(0)^2\big] \,.
\end{equation*}
Applying the first bound in \eqref{eq:boundsvar} we then obtain
\begin{equation} \label{eq:RGamma}
	R_{L/2}(\varphi,\varphi) \, \Gamma \le
	2\, \mathtt{K}_{h}^{(\hat t)} \, \bbvar[Z_{L}^\beta(\varphi)] \,,
\end{equation}
hence, recalling \eqref{eq:qGamma},
the first relation in \eqref{eq:2bounds} holds
with $C := \frac{16 \, \mathtt{K}_{h}^{(\hat t)}}{\eta}$.

\smallskip

Let us finally prove the second relation in \eqref{eq:2bounds}.
By assumption \eqref{eq:locunif} we can bound
\begin{equation*}
	\big( L \, \| \varphi \|_{\ell^2}^2 \big)^{\frac{1}{q}}
	= \rme^{\frac{1}{q} \log (L \, \| \varphi \|_{\ell^2}^2)}
	\le \rme^{\frac{1}{q} \, \sfc_2' \, R_L(\varphi,\varphi)} \,.
\end{equation*}
It remains to show that, for some constant $c > 0$
(depending on $h$, $\hat{t}$, $\sfc_1$, $\sfc_2'$, $\sfc_3$),
\begin{equation}\label{eq:last-to-show}
	q \ge c \, R_{L}(\varphi,\varphi) \,,
\end{equation}
so that the second relation in \eqref{eq:2bounds} holds
with $C' = \rme^{\frac{1}{c}\,\sfc_2'}$.

Let us finally prove \eqref{eq:last-to-show}.
It follows by \eqref{eq:RGamma} and assumption \eqref{eq:varub0+} that
\begin{equation*}
	\frac{1}{\Gamma} \ge \frac{1}{2 \, \mathtt{K}_{h}^{(\hat t)} \, \sfc_3} \, R_{L/2}(\varphi,\varphi) \,,
\end{equation*}
hence, recalling \eqref{eq:qfix}, we can bound
\begin{equation}\label{eq:lbq}
	q \ge a \, R_{L/2}(\varphi,\varphi) \qquad
	\text{with} \quad a := \tfrac{4}{\eta} \wedge  \tfrac{1}{8 \, \mathtt{K}_{h}^{(\hat t)} \, \sfc_3} \,.
\end{equation}
At last, by \eqref{eq:doubling}
and \eqref{eq:Rub}, for any probability mass function $\varphi$ we can write
\begin{equation} \label{eq:Rratio}
	\frac{R_{L}(\varphi,\varphi)}{R_{L/2}(\varphi,\varphi)}
	\le 1 + \frac{R_{L}(\varphi,\varphi)-R_{L/2}(\varphi,\varphi)}{R_{L/2}(\varphi,\varphi)}
	\le 1 + \frac{\mathfrak{a}_+}{\eta} \, .
\end{equation}
Combining \eqref{eq:Rratio} with \eqref{eq:lbq} then gives
\eqref{eq:last-to-show} with $c := a/(1+\mathfrak{a}_+/\eta)$.\qed

\appendix

\section{Second moment computations}
\label{sec:2mb}

We first prove Lemma~\ref{th:Green} about Green's functions.
Then we prove Proposition~\ref{th:2mb} on the second moment of the point-to-plane partition function.
Finally, we prove Theorem~\ref{th:vargen} on the variance of the averaged partition function,
together with Propositions~\ref{th:initial} and~\ref{th:initial2}.

\smallskip

\begin{proof}[Proof of Lemma~\ref{th:Green}]
From \eqref{eq:Green}, we see that $\cG(a) = \frac{1}{2\pi} \log (1+a^{-2}) + O(1)$ uniformly over $a>0$.
This already proves the second equality in \eqref{eq:asR}.
To prove the first equality, we plug the first line of \eqref{eq:llt} into \eqref{eq:RL2} to get
\begin{equation}\label{eq:RL2bis}
	R_L(z) = \bigg\{ \sum_{n=1}^L g_{n}(z) \bigg\} \, 2 \cdot \ind_{\Z^2_\even}(z)
	\,+\, O(1)  \,.
\end{equation}
We can replace $g_{n}(z) = g_{n}(|z|)$
by $g_{n}(|z|+1)$,
because their difference is $O(n^{-3/2})$ and can be absorbed in the $O(1)$ term.
We next write
\begin{equation*}
	\sum_{n=1}^L g_{n}(|z|+1)
	= \frac{1}{L} \sum_{n=1}^L g_{\frac{n}{L}}\Big(\tfrac{|z|+1}{\sqrt{L}} \Big)
	= \cG\Big( \tfrac{|z|+1}{\sqrt{L}} \Big) + \Delta_L\Big( \tfrac{|z|+1}{\sqrt{L}} \Big)
\end{equation*}
where we set
\begin{equation*}
	\Delta_L(x) := \sum_{n=1}^L \int_{\frac{n-1}{L}}^{\frac{n}{L}}
	\big\{ g_{\frac{n}{L}}(x) - g_t(x) \big\} \, \dd t \,.
\end{equation*}
It remains to show that $\Delta_L(x) = O(1)$ is uniformly bounded
for $|x| = \frac{|z|+1}{\sqrt{L}} \ge \frac{1}{\sqrt{L}}$.

By direct computation we see
that $|\partial_u g_u(x)| \le \frac{c}{u^2}\, \rme^{-\frac{|x|^2}{c u}}$ for some $c< \infty$.
Then, uniformly for $|x| \ge \frac{1}{\sqrt{L}}$, we can bound the term $n=1$ in $\Delta_L(x)$ by
an absolute constant $C < \infty$:
\begin{equation*}
	\int_0^{\frac{1}{L}} \bigg( \int_t^{\frac{1}{L}} \frac{c}{u^2} \,
	\rme^{-\frac{1}{c L u}}  \, \dd u \bigg) \dd t
	= \int_0^{\frac{1}{L}} \frac{c}{u} \,
	\rme^{-\frac{1}{c L u}} \, \dd u
	= \int_0^1 \frac{c}{v} \, \rme^{-\frac{1}{cv}} \, \dd v =: C  < \infty \,.
\end{equation*}
For the terms $n \ge 2$ in $\Delta_L(x)$, we simply use $|\partial_u g_u(x)| \le \frac{c}{u^2}$ to bound
$|g_{\frac{n}{L}}(x) - g_t(x)| \le \frac{c}{t^2} \frac{1}{L}$ (because $\frac{n}{L}-t \le \frac{1}{L}$)
and we obtain, uniformly over $x\in\R^2$,
\begin{equation*}
	\sum_{n=2}^L \, \int_{\frac{n-1}{L}}^{\frac{n}{L}} \,
	\big| g_{\frac{n}{L}}(x) - g_t(x) \big| \, \dd t
	\le \frac{c}{L}\, \sum_{n=2}^L \, \int_{\frac{n-1}{L}}^{\frac{n}{L}} \, \frac{1}{t^2} \, \dd t
	= \frac{c}{L} \, \int_{\frac{1}{L}}^1 \frac{1}{t^2} \,\dd t
	\le c < \infty \,.
\end{equation*}
This completes the proof
that $|\Delta_L(x)| \le C+c$ uniformly for $L\in\N$ and $|x| \ge \frac{1}{\sqrt{L}}$.

\smallskip

We finally prove \eqref{eq:intR}, for which we may assume that $z \ne 0$ and
$L$ is large enough.
By the second line of \eqref{eq:llt}, uniformly over $|z|^2 \le 2 (1+  t^2) n$ we can write,
for a suitable $c > 0$,
\begin{equation*}
	q_{2n}(z) = g_{n}(z) \, \rme^{O(\frac{(1+t^2)^2}{n})} \, 2 \, \ind_{\Z^2_\even}(z)
	\ge \frac{c \, \rme^{-\frac{t^4}{c}}}{n} \, \ind_{\Z^2_\even}(z) \,,
\end{equation*}
hence, restricting the sum in \eqref{eq:RL2} to $n \ge
\bar{n}(z) :=  \lceil \frac{1}{2} ( \frac{|z|^2}{1+t^2}  + 1 ) \rceil $, we get
\begin{equation*}
	R_L(z) \ge \sum_{n=\bar{n}(z)}^L \frac{c \, \rme^{-\frac{t^4}{c}}}{n}
	\ge \int_{\bar{n}(z)}^{L+1} \frac{c \, \rme^{-\frac{t^4}{c}}}{u} \, \dd u
	= c \, \rme^{-\frac{t^4}{c}} \, \log \frac{L+1}{\bar{n}(z)} \,.
\end{equation*}
For $|z| \le t \sqrt{L}$ we have $\bar{n}(z) \le \frac{1}{2} (L+1)$, equivalently
$\frac{L+1}{\bar{n}(z)} \ge 1 + \frac{L+1}{2 \bar{n}(z)}$, and since $\bar{n}(z) \le \frac{|z|^2+1}{2}$
we finally obtain $\log \frac{L+1}{\bar{n}(z)} \ge \log (1 + \frac{L+1}{1+|z|^2})$, hence
\eqref{eq:intR} holds with $c_t = c \, \rme^{-\frac{t^4}{c}}$.
\end{proof}

\medskip

The key tool for the proof of Proposition~\ref{th:2mb} is the following renewal
representation of the second moment of the partition function developed in \cite{CSZ19a}.

\begin{remark}[Renewal interpretation]
Given $N\in\N$,
we define the integer valued renewal process
\begin{equation*}
	\tau^{(N)}_k = T^{(N)}_1 + \ldots + T^{(N)}_k
\end{equation*}
with $\tau^{(N)}_0 := 0$ and i.i.d.\ increments $(T^{(N)}_i)_{i\in\N}$ with distribution (recall \eqref{eq:rw} and \eqref{eq:RL})
\begin{equation} \label{eq:T1}
	\P\big(T^{(N)}_i = n\big) = \frac{1}{R_N} \, q_{2n}(0), \qquad 1\leq i\leq N.
\end{equation}
For $\beta = \beta_N$ in the quasi-critical regime \eqref{eq:quasi-crit}, we can then write \eqref{eq:U} as follows: for every $n\in\N_0$,
\begin{equation}\label{eq:U-renewal}
	U_{\beta_N}(n) = \sum_{k \ge 0} \pig( 1 - \tfrac{|\theta_N|}{\log N}\pig)^k
	\, \P\big(\tau_k^{(N)} = n \big)
	= \frac{\log N}{|\theta_N|} \, \P\big(\tau_{\cK_N}^{(N)} = n \big) \,,
\end{equation}
where $\cK_N$ is an independent Geometric random variable with mean $\frac{\log N}{|\theta_N|}$.
\end{remark}

\begin{proof}[Proof of Proposition~\ref{th:2mb}]
Note that (recall \eqref{eq:T1} and \eqref{eq:RL})
\begin{equation*}
\P\big(\tau_k^{(N)} \le L \big) \le \P\big(T_i^{(N)} \le L \ \forall i=1,\ldots, L \big)
= \P\big(T_1^{(N)} \le L \big)^k = \Big(\frac{R_L}{R_N}\Big)^k.
\end{equation*}
Then it follows by \eqref{eq:U} and \eqref{eq:barU} that
\begin{equation*}
	\bbE\big[ Z_L^{\beta_N}(0)^2 \big] = \overline{U}_{\beta_N}(L)
	= \sum_{k \ge 0} \pig( 1 - \tfrac{|\theta_N|}{\log N}\pig)^k \, \P\big(\tau_k^{(N)} \le L \big)
	\le \sum_{k \ge 0} \pig( 1 - \tfrac{|\theta_N|}{\log N}\pig)^k\,
	\pig(\tfrac{R_L}{R_N} \pig)^k \,,
\end{equation*}
which yields the RHS of \eqref{eq:point-to-plane-2ndmom} as an upper bound.

To get a lower bound, note that we can write
$\P\big(\tau_k^{(N)} \le L \big) = (\frac{R_L}{R_N})^k \,\P\big(\tau_k^{(L)} \le L \big)$
because the law of $T^{(N)}_i$ conditionally on $T^{(N)}_i \le L$ is just the law
of $T^{(L)}_i$, see \eqref{eq:T1}. By Markov's inequality, we have
$\P\big(\tau_k^{(L)} \le L \big) \ge 1 - \frac{1}{L}\E[\tau_k^{(L)}]$, where
\begin{equation*}
\E[\tau_k^{(L)}] = k \, \E[\tau_1^{(L)}] = \frac{k}{R_L} \sum_{n=1}^L n \, q_{2n}(0)
\le c \, \frac{k}{R_L} \, L \quad \mbox{for some} \quad c<\infty.
\end{equation*}
Using $\sum_{k=0}^\infty k \, x^k = \frac{x}{(1-x)^2}$, we obtain
\begin{equation*}
\begin{split}
	\bbE\big[ Z_L^{\beta_N}(0)^2 \big]
	& \ge \sum_{k \ge 0} \pig( 1 - \tfrac{|\theta_N|}{\log N}\pig)^k\,
	\pig(\tfrac{R_L}{R_N} \pig)^k \, \pig(1 - c\,\tfrac{k}{R_L}\pig) \\
	& = \frac{1}{1 - \frac{R_L}{R_N}(1-\frac{|\theta_N|}{\log N})}
	\bigg( 1 -  \frac{c}{R_L} \frac{\frac{R_L}{R_N}(1
	-\frac{|\theta_N|}{\log N})}{1-\frac{R_L}{R_N}(1-\frac{|\theta_N|}{\log N})}\bigg) \\
	& \ge \frac{1}{1 - \frac{R_L}{R_N}(1-\frac{|\theta_N|}{\log N})}
	\bigg( 1 -  \frac{c'}{|\theta_N|} \bigg)\,,
\end{split}
\end{equation*}
where the last inequality is obtained using $(1-\frac{|\theta_N|}{\log N}) \le 1$ in the numerator,
$\frac{R_L}{R_N} \le 1$ in the denominator, and $R_N\sim \frac{\log N}{\pi}$ by \eqref{eq:RL}.
This completes the proof of \eqref{eq:point-to-plane-2ndmom}.
\end{proof}

\smallskip

\begin{proof}[Proof of Theorem~\ref{th:vargen}]
We are going to exploit the upper and lower bounds in \eqref{eq:boundsvar}.

We recall from \eqref{eq:theta} that $\log \frac{1}{\delta_N^2} = |\theta_N|$,
hence by \eqref{eq:RL},
\begin{equation*}
	\frac{R_{L_N}}{R_N} = \frac{\log N - \ell \, \log \frac{1}{\delta_N^2} + O(1)}{\log N + O(1)}
	= 1 - \ell  \, \frac{|\theta_N| + O(1)}{\log N}  \,.
\end{equation*}
Applying \eqref{eq:point-to-plane-2ndmom} then gives
\begin{equation} \label{eq:barUB}
	\bbE\big[ Z_{L_N}^{\beta_N}(0)^2 \big]
	\sim \frac{1}{1-\frac{R_{L_N}}{R_N} (1-\frac{|\theta_N|}{\log N})}
	\sim \frac{1}{1+\ell} \, \frac{\log N}{|\theta_N|} \,,
\end{equation}
and the same holds for
$\bbE\big[ Z_{\frac{1}{2}L_N}^{\beta_N}(0)^2 \big]$
since $\frac{1}{2}L_N = N \, (\delta_N^2)^{\ell+o(1)}$, just as~$L_N$.

The difference $R_{L} (\varphi, \varphi) - R_{\frac{1}{2}L} (\varphi, \varphi)$
is uniformly bounded by a constant, see \eqref{eq:doubling}.
Then assumption \eqref{eq:phi-cond0} on $\varphi_N$ implies, since $|\theta_N| \to \infty$,
\begin{equation} \label{eq:Ras}
	R_{\frac{1}{2} L_N} (\varphi_N, \varphi_N)
	= R_{L_N} (\varphi_N, \varphi_N) + O(1)
	= \frac{1}{\pi} \, (w-\ell) \, |\theta_N| + o(|\theta_N|) \,.
\end{equation}
Since $\sigma_{\beta_N}^2 \sim \frac{1}{R_N} \sim \frac{\pi}{\log N}$, see \eqref{eq:quasi-crit}
and \eqref{eq:RL}, the bounds in \eqref{eq:boundsvar} yield \eqref{eq:2nd-mom}.
\end{proof}

\begin{remark}[Sub-critical regime] \label{rem:subcrapp}
The proof of Theorem~\ref{th:vargen} also applies to the sub-critical regime \eqref{eq:sub-critical} if we take $\theta_N \sim -(1-\hat\beta^2) \log N$ with $\hat\beta^2 \in (0,1)$. In this case $\frac{|\theta_N|}{\log N} \to (1-\hat\beta^2) > 0$ and we must take into account second order terms in \eqref{eq:barUB}, namely
\begin{equation} \label{eq:barUBsub}
	\bbE\big[ Z_{L_N}^{\beta_N}(0)^2 \big]
	\sim \frac{1}{1-(1-\ell\,(1-\hat\beta^2))\,(1-(1-\hat\beta^2))}
	\sim \frac{1}{(1+\ell\,\hat\beta^2)\,(1-\hat\beta^2)} \,.
\end{equation}
Since $\sigma_{\beta_N}^2 \sim \frac{\hat\beta^2}{R_N} \sim \frac{\pi \, \hat\beta^2}{\log N}$, the bounds in \eqref{eq:boundsvar} yield \eqref{eq:2nd-momsub}.
\end{remark}

\smallskip

\begin{proof}[Proof of Proposition~\ref{th:initial}]
Recalling \eqref{eq:Rfg}, we rewrite condition \eqref{eq:phi-cond0} as
\begin{equation}\label{eq:Ras2}
	\sum_{x, y \in \Z^2_\even} \varphi_N(x) \, \varphi_N(y) \,
	R_{L_N}(x-y) = \frac{1}{\pi} \, \log \frac{L_N}{W_N} + o(|\theta_N|) \,.
\end{equation}
To obtain our goal  \eqref{eq:phi-cond}, we simply replace $R_{L_N}(x-y)$ in this sum by
$\frac{1}{\pi} \log (1 + \frac{L_N}{1 + |x-y|^2})$ because their difference
is uniformly bounded, see \eqref{eq:asR}, and hence their contributions to the sum
differ by $O(1) = o(|\theta_N|)$, which is
negligible for the RHS of \eqref{eq:Ras2}.
\end{proof}

\smallskip

\begin{proof}[Proof of Proposition~\ref{th:initial2}]
We need to show that condition \eqref{eq:conditions-phi} implies \eqref{eq:phi-cond}.

We assume for simplicity of notation that $z_N = 0$.
We then rewrite \eqref{eq:conditions-phi} as follows:
\begin{equation} \label{eq:conditions-phi-bis}
	\sum_{|x| \le \sqrt{W_N^+}}
	\varphi_N(x) = 1 - o(1)\,, \qquad
	\sup_{x\in\Z^2} \varphi_N(x) \le \frac{1}{W_N^-}\,,
	\qquad
	\text{with } \ W_N^\pm := W_N \, \rme^{\pm t_N} \,,
\end{equation}
where we recall that $0 \le t_N = o(|\theta_N|)$.
Henceforth

If $\varphi_N$ satisfies \eqref{eq:conditions-phi-bis}, then we can get a lower bound on
the sum in \eqref{eq:phi-cond}  by restricting to
the ranges $|x|, |y| \le \sqrt{W_N^+}$,
which have probability $1 - o(1)$ by the first relation in \eqref{eq:conditions-phi-bis}.
For such values of $x,y$ we have $|x-y| \le 2 \, \sqrt{W_N^+}$,
so the logarithm in the LHS of \eqref{eq:phi-cond} is bounded from below by
\begin{equation} \label{eq:ab}
	\log\Big(1 + \tfrac{L_N}{1+4 \, W_N^+} \Big)
	\ge \log \tfrac{L_N}{5 \, W_N^+}
	= \log \tfrac{L_N}{W_N} - t_N - \log 5
	= \log \tfrac{L_N}{W_N} - o(|\theta_N|) \,,
\end{equation}
where the first inequality holds for large~$N$
because $W_N^+ \ge W_N \to \infty$, see \eqref{eq:W}.
This yields the RHS of \eqref{eq:phi-cond}.

To get an upper bound, we fix $\xi_N \to \infty$ with $\xi_N = o(|\theta_N|)$ and we
define the scale $V_N := W_N^- \, \rme^{- \xi_N} = W_N \, \rme^{- t_N - \xi_N}$
which is smaller than $W_N^-$.
For $|x-y| \ge \sqrt{V_N}$, we can bound from above the logarithm in \eqref{eq:phi-cond}
by $\log(1 + \tfrac{L_N}{1+V_N} ) \le \log (1+ \tfrac{L_N}{V_N})$. In case $\tfrac{L_N}{V_N} \le 1$
this is at most $\log 2 = o(|\theta_N|)$, while in case $\tfrac{L_N}{V_N} > 1$ we obtain the upper bound
\begin{equation*}
	\log \Big(2 \, \tfrac{L_N}{V_N}\Big)
	= \log \tfrac{L_N}{W_N} + t_N + \xi_N + \log 2
	= \log \tfrac{L_N}{W_N} + o( |\theta_N| ) \,,
\end{equation*}
which agrees with the RHS of \eqref{eq:phi-cond}.
We are left with showing that the range $|x-y| < \sqrt{V_N}$
gives a negligible contribution of order $o(|\theta_N|)$ to the sum in \eqref{eq:phi-cond}.

For fixed~$y$, we apply the second relation in \eqref{eq:conditions-phi-bis} to
estimate the sum over $x$ in \eqref{eq:phi-cond}:
\begin{equation*}
\begin{split}
	\sum_{x \in B(y, \sqrt{V_N})} \varphi_N(x) \,
	\log \Bigl( 1 + \tfrac{L_N}{1+|x-y|^2} \Bigr)
	&\le \frac{1}{W_N^-} \, \sum_{x \in B(y, \sqrt{V_N})}
	\log \Bigl( 1 + \tfrac{L_N}{1+|x-y|^2} \Bigr) \\
	& \le C \, \frac{V_N}{W_N^-} \,
	\int_{|z| \le 1} \log \Bigl( 1 + \tfrac{L_N}{V_N} \, \tfrac{1}{|z|^2} \Bigr) \, \dd z
\end{split}
\end{equation*}
for some $C < \infty$, by Riemann sum approximation.
Since $\frac{V_N}{W_N^-} = \rme^{-\xi_N} = o(1)$ by definition
of $V_N$, it remains to show that the integral is $O(|\theta_N|)$.
If $\frac{L_N}{V_N} \le 1$ then the integral is at most
\begin{equation*}
	\int_{|z| \le 1} \log \Bigl( 1 + \tfrac{1}{|z|^2} \Bigr) \, \dd z
	= O(1) = o(|\theta_N|) \,,
\end{equation*}
while if $\frac{L_N}{V_N} > 1$ then,
recalling that $V_N = W_N \, \rme^{-o(|\theta_N|)}$, we can bound the integral by
\begin{equation*}
	\int_{|z| \le 1} \log \Bigl( \tfrac{L_N}{V_N} \, \tfrac{2}{|z|^2} \Bigr) \, \dd z
	= \log \tfrac{L_N}{V_N} + O(1)
	= \log \tfrac{L_N}{W_N} + o(|\theta_N|)
	= O(|\theta_N|) \,,
\end{equation*}
where the last equality holds by \eqref{eq:L} and \eqref{eq:W}.
The proof is completed.
\end{proof}

\normalcolor

\section{Auxiliary proofs for high moments bounds}
\label{sec:addpr}

In this appendix, we collect the proofs of some results from Section~\ref{sec:moments}.

\subsection{Proof of Lemma~\ref{th:locunif}}
\label{sec:locunif}
We can rewrite $\bbD[\varphi]$, see \eqref{eq:bbD}, as follows:
\begin{equation} \label{eq:bbD2}
	\bbD[f] = \frac{1}{2} \sum_{z, w\in\Z^2} |z-w|^2 \, f(z) \, f(w) \,.
\end{equation}
To prove \eqref{eq:rphibound0}, we first estimate
\begin{equation*}
	p_> := \sum_{|z - z_0| > s \sqrt{L/2}} \varphi(z)
	\le \rme^{- \sqrt{2} \,\hat{t}\, s} \sum_{z\in\Z^2} \varphi(z) \, \rme^{2\,\hat{t}\,\frac{|z-z_0|}{\sqrt{L}}}
	\le \sfc_1 \, \rme^{- \sqrt{2} \,\hat{t}\, s} \le \frac{1}{4}
	\qquad \text{for } s := \tfrac{\log (4 \sfc_1)}{\sqrt{2}\,\hat{t}\,}  \,,
\end{equation*}
therefore $p_\le := 1 - p_> = \sum_{|z - z_0| \le s \sqrt{L/2}} \varphi(z) \ge \frac{3}{4}$.
We now restrict the sums in the definition \eqref{eq:Rfg}
of  $R_L(\varphi,\varphi)$ to $|z - z_0| \le s \sqrt{L/2}$, $|w - z_0| \le s \sqrt{L/2}$:
defining the probability mass function
$\tilde \varphi(z) := p_\le^{-1} \, \varphi(z) \, \ind_{|z - z_0| \le s \sqrt{L/2}}$
and recalling  \eqref{eq:intR},
we can write
\begin{equation*}
	R_L(\varphi,\varphi)
	\ge R_{L/2}(\varphi,\varphi)
	\ge p_\le^2 \, R_{L/2}(\tilde\varphi, \tilde\varphi)
	\ge p_\le^2 \, c_s \log\pig(1 + \tfrac{L/2}{1+2 \bbD[\tilde\varphi]}\pig) \,,
\end{equation*}
where we applied \eqref{eq:bbD2} and Jensen's inequality,
since $x \mapsto \log(1 + \frac{L}{x})$ is convex for $x>0$.
Since $\bbD[\tilde\varphi] \le p_\le^{-2} \, \bbD[\varphi] \le 2 \, \bbD[\varphi]$,
the proof of \eqref{eq:rphibound0} is complete.

In order to prove that \eqref{eq:locunif1} implies \eqref{eq:locunif},
we first observe that
\begin{equation} \label{eq:phil2}
	\| \varphi\|_{\ell^2}^2
	= \sum_{z\in\Z^2} \varphi(z)^2
	\le \|\varphi\|_{\ell^\infty} \sum_{z\in\Z^2} \varphi(z)
	= \|\varphi\|_{\ell^\infty} \,,
\end{equation}
since $\varphi$ is a probability mass function. We next apply \eqref{eq:rphibound0}, that we rewrite for convenience:
\begin{equation}\label{eq:rphibound}
	R_L(\varphi,\varphi) \ge \mathfrak{c} \, \log\pig(1 + \tfrac{L/2}{1+4 \, \bbD[\varphi]}\pig) \,.
\end{equation}
We now assume that \eqref{eq:locunif1} holds and we distinguish two cases:
\begin{itemize}
\item if $\bbD[\varphi] \le 1$ then $R_L(\varphi,\varphi) \ge \mathfrak{c}  \log \frac{L}{10}$ by
\eqref{eq:rphibound},
while $\log( L\, \| \varphi\|_{\ell^2}^2)
\le \log( L)$ by \eqref{eq:phil2} and $\|\varphi\|_{\ell^\infty} \le 1$,
hence \eqref{eq:locunif} is satisfied;

\item if $\bbD[\varphi] > 1$, then
$R_L(\varphi,\varphi) \ge \mathfrak{c}  \log ( \frac{1}{10} \, \frac{L}{\bbD[\varphi]})$ by
\eqref{eq:rphibound}
and $\log( L\, \| \varphi \|_{\ell^2}^2)
\le \log(\sfc_1 \sfc_2 \frac{L}{\bbD[\varphi]})$ by \eqref{eq:phil2} and \eqref{eq:locunif1},
so \eqref{eq:locunif} holds again.
\end{itemize}
The proof is completed.\qed

\smallskip

\subsection{Proof of Lemma~\ref{th:boundsphi}}
\label{sec:boundsphi}

Recalling \eqref{eq:bbD2}, by the (squared) triangle inequality $|z-w|^2 \le 2(|z-z_0|^2 + |w-z_0|^2)$
and $x^2 \le \rme^x$ for $x \ge 0$, we can bound
\begin{equation*}
	\bbD[\varphi] \le \sum_{z\in\Z^2} |z-z_0|^2 \, \varphi(z)
	\le \frac{L}{2\,\hat{t}\,} \sum_{z\in\Z^2} \rme^{\frac{2\,\hat{t}\,}{L}|z-z_0|^2} \, \varphi(z)
	\le \frac{\sfc_1}{2\,\hat{t}\,} \, L \,.
\end{equation*}
It then suffices to apply \eqref{eq:rphibound0} to prove \eqref{eq:Rub}
with $\eta = \mathfrak{c} \, \log\pig(1 + \tfrac{\hat{t}\,}{2(\,\hat{t}\, + 2 \sfc_1)}\pig)$.
\qed

\smallskip

\subsection{Proof of Lemma~\ref{th:sigmatheta}}
\label{sec:sigmatheta}

We fix $\beta \ge 0$, $L \in \N$
and  a probability mass function $\varphi$ satisfying \eqref{eq:concsqrtL+} and \eqref{eq:varub0+}.
By the first inequality in \eqref{eq:boundsvar}, in view of \eqref{eq:Rub} and \eqref{eq:varub0},
we obtain
\begin{equation} \label{eq:c4c2}
	\sigma_\beta^2 \,\, \bbE\big[Z_{L/2}^{\beta}(0)^2\big] \le \frac{\sfc_3}{\eta} \,.
\end{equation}
We are going to obtain a lower bound on $\bbE\big[Z_{L/2}^{\beta}(0)^2\big]$ which will yield
our goal \eqref{eq:sigmatheta}.

If $\sigma_{\beta}^2 R_L < 1$ then
\eqref{eq:sigmatheta} holds with $\bar\theta = 0$.
We then assume
$\sigma_{\beta}^2 R_L \ge 1$, which lets us write
\begin{equation} \label{eq:newpara}
	\sigma_\beta^2 \, R_L
	= \frac{1}{1-\frac{\theta}{R_L}}
	\qquad \text{for a suitable} \quad 0 \le \theta = \theta(\beta,L) < R_L \,.
\end{equation}
Our goal \eqref{eq:sigmatheta} is to show that $\theta \le \bar\theta$
for some $\bar\theta = \bar\theta(\,\hat{t}\,, \sfc_1, \sfc_3) \in [0,\infty)$.

Let $\tau_k^{(L)}$ be the random walk with
$\P(\tau_1^{(L)} = n) = \frac{1}{R_L} \, q_{2n}(0) \, \ind_{\{1 \le n \le L\}}$.
We can write
\begin{equation*}
\begin{split}
	\bbE\big[ Z_{L/2}^{\beta}(0)^2\big]
	&= 1 + \sum_{k=1}^\infty (\sigma_{\beta}^2)^k
	\sum_{0 < n_1 < \ldots < n_k \le L/2} \prod_{i=1}^k q_{2(n_i - n_{i-1})}(0) \\
	&= 1 + \sum_{k=1}^\infty (\sigma_{\beta}^2 R_L)^k
	\, \P(\tau_k^{(L)} \le L/2)
	\ge 1 + \sum_{k=1}^\infty (\sigma_{\beta}^2 R_L)^k
	\, \bigg(1 - \frac{\E[\tau_k^{(L)}]}{L/2}\bigg) \,.
\end{split}
\end{equation*}
Note that $\E[\tau_k^{(L)}] = k \, \E[\tau_1^{(L)}]
= \frac{k}{R_L} \sum_{n=1}^L n \, q_{2n}(0)
\le \frac{k}{R_L} \, \mathfrak{a}_+ \, L$ by \eqref{eq:barN}.
Restricting to $k \le \lceil \frac{1}{4 \mathfrak{a}_+} R_L \rceil$
we then have $\E[\tau_k^{(L)}] \le L/4$.
Since $\sigma_\beta^2 R_L \ge \rme^{\frac{\theta}{R_L}}$ by \eqref{eq:newpara}
and $\frac{1}{1-x} \ge \rme^x$ for $0 \le x \le 1$, we get
\begin{equation*}
	\bbE\big[ Z_{L/2}^{\beta}(0)^2\big] \ge
	\frac{1}{2} \, \sum_{k=1}^{K}
	(\sigma_{\beta}^2 R_L)^k
	\ge \frac{1}{2} \, \sum_{k=1}^{K} \rme^{\frac{\theta}{R_M} k}
	\qquad \text{with} \qquad
	K := \lceil \tfrac{1}{4 \mathfrak{a}_+} R_L \rceil \,.
\end{equation*}
Applying Jensen we then obtain
\begin{equation*}
	\bbE\big[ Z_{L/2}^{\beta}(0)^2\big]
	\ge \frac{K}{2}   \,
	\rme^{\frac{\theta}{R_L} \, \frac{1}{K}
	\sum_{k=1}^{K} k}
	\ge \frac{K}{2}   \,
	\rme^{\frac{\theta}{R_L} \frac{K}{2}}
	\ge \frac{R_L}{8\mathfrak{a}_+}   \,
	\rme^{\frac{\theta}{8\mathfrak{a}_+}}
	\ge \frac{1}{\sigma_\beta^2} \, \frac{1}{8\mathfrak{a}_+}   \,
	\rme^{\frac{\theta}{8\mathfrak{a}_+}} \,,
\end{equation*}
where in the last inequality we used $\sigma_\beta^2 \, R_L \ge 1$.
Recalling \eqref{eq:c4c2}, we finally obtain
\begin{equation*}
	\rme^{\frac{\theta}{8\mathfrak{a}_+}}
	\le 8\mathfrak{a}_+ \, \frac{\sfc_3}{\eta} \,,
	\qquad \text{hence} \qquad
	\theta \le
	\bar\theta := 8\mathfrak{a}_+ \, \log^+ \frac{8\mathfrak{a}_+\,\sfc_3}{\eta}\,,
\end{equation*}
which completes the proof of \eqref{eq:sigmatheta}.\qed

\smallskip

\subsection{Proof of Lemma~\ref{th:2lb}}
\label{sec:2lb}
Recalling \eqref{eq:Rlambda},
we denote by $\tau^{(L,\hat\lambda)}_k$ the random walk with step distribution
$\P(\tau^{(L,\hat\lambda)}_1 = n) = \frac{1}{R^{(\hat\lambda)}_L} \rme^{-\frac{\hat\lambda}{L} n}
\, q_{2n}(0) \, \ind_{\{1,\ldots, L\}}(n)$.
Then we can write
\begin{equation*}
\begin{split}
	\bbE\big[ Z_{L/2}^{\beta}(0)^2\big]
	&= 1 + \sum_{k=1}^\infty (\sigma_{\beta}^2)^k
	\sum_{0 < n_1 < \ldots < n_k \le L/2} \prod_{i=1}^k q_{2(n_i - n_{i-1})}(0) \\
	&\ge 1 + \sum_{k=1}^\infty (\sigma_{\beta}^2)^k
	\sum_{0 < n_1 < \ldots < n_k \le L/2} \prod_{i=1}^k \rme^{-\frac{\hat\lambda}{L} (n_i - n_{i-1})} \,
	q_{2(n_i - n_{i-1})}(0) \\
	&= 1 + \sum_{k=1}^\infty (\sigma_{\beta}^2 R_L^{(\hat\lambda)})^k
	\, \P(\tau_k^{(L,\hat\lambda)} \le L/2)
	\ge 1 + \sum_{k=1}^\infty (\sigma_{\beta}^2 R_L^{(\hat\lambda)})^k
	\, \bigg(1 - \frac{\E[\tau_k^{(L,\hat\lambda)}]}{L/2}\bigg)
\end{split}
\end{equation*}
and note that $\E[\tau_k^{(L,\hat\lambda)}] = k \, \E[\tau_1^{(L,\hat\lambda)}]
= \frac{k}{R_L^{(\hat\lambda)}} \sum_{n=1}^L n \, \rme^{-\frac{\hat\lambda}{L} n} \, q_{2n}(0)
\le \frac{k}{R_L^{(\hat\lambda)}} \, \mathfrak{a}_+ \, L$ by \eqref{eq:barN}.
Note that assumption \eqref{eq:cond-for-LB} ensures that $\sigma_{\beta}^2 R_L^{(\hat\lambda)} < 1$.
By $\sum_{k=1}^\infty k \, x^k = \frac{x}{(1-x)^2}$ we then obtain
\begin{equation*}
	\bbE\big[ Z_{L/2}^{\beta}(0)^2\big] \ge
	\frac{1}{1-\sigma_\beta^2 R_L^{(\hat\lambda)}}
	- \frac{2\mathfrak{a}_+}{R_L^{(\hat\lambda)}}
	\, \frac{\sigma_\beta^2 R_L^{(\hat\lambda)}}{(1-\sigma_\beta^2 R_L^{(\hat\lambda)})^2}
	= \frac{1}{1-\sigma_\beta^2 R_L^{(\hat\lambda)}} \bigg(1 -
	\frac{2\mathfrak{a}_+ \, \sigma_\beta^2}{1-\sigma_\beta^2 R_L^{(\hat\lambda)}} \bigg) \,.
\end{equation*}
Note that assumption \eqref{eq:cond-for-LB}
is equivalent to
$\frac{2\mathfrak{a}_+ \, \sigma_\beta^2}{1-\sigma_\beta^2 R_L^{(\hat\lambda)}} \le \frac{1}{2}$,
which proves \eqref{eq:LBvar}.\qed

\smallskip

\subsection{Proof of Lemma~\ref{th:choice}}
\label{sec:choice}

Since $1-\rme^{- x} \ge \frac{1}{2}$ for $x \ge 1$,
recalling \eqref{eq:barN} and bounding $\sum_{n=a}^b \frac{1}{n} \ge \log \frac{b}{a}$,
for $\hat\lambda \ge 1$ we can write
\begin{equation*}
\begin{split}
	R_L^{(\hat\lambda)}
	&:= R_L - \sum_{n=1}^L (1-\rme^{-\hat\lambda \frac{n}{L}} ) \, q_{2n}(0)
	\le R_L - \frac{1}{2} \sum_{n=\lceil \hat\lambda^{-1} L \rceil}^L q_{2n}(0)
	\le R_L - \frac{\mathfrak{a}_-}{2}\log \frac{L}{\lceil \hat\lambda^{-1} L \rceil} \,.
\end{split}
\end{equation*}
We finally note that
\begin{equation*}
	\frac{L}{\lceil \hat\lambda^{-1} L \rceil}
	\ge
	\frac{L}{\hat\lambda^{-1}L+1}
	= \hat\lambda \, \frac{L}{L + \hat\lambda}
	\ge \frac{\hat\lambda}{2}
	\qquad \text{for } L \ge \hat\lambda  \,,
\end{equation*}
which completes the proof of \eqref{eq:choice}.
\qed

\smallskip

\subsection{Proof of Theorem~\ref{th:moment-bound-CCR}}
\label{sec:mb}

For the final relation \eqref{eq:geosu}, we simply note that
\begin{equation*}
	\forall 0 \le x < 1\,, \ h \ge 2 \colon \qquad
	\sum_{r=1}^\infty x^{\max\{r,\frac{h}{2}\}} \le
	\bigg( \frac{h}{2} + \frac{1}{1-x}\bigg) \, x^{\frac{h}{2}}
	\le \rme^h \, \frac{x^{\frac{h}{2}}}{1-x} \,.
\end{equation*}

It remains to show that \eqref{eq:genbo} holds. This bound is already proved in \cite{CCR23}, though it is not stated in this form. For this reason, in the next lines we state the needed results from \cite[Section~4]{CCR23} and we put them together to deduce \eqref{eq:genbo}. \emph{The purpose is to provide a roadmap for an interested reader to check that \eqref{eq:genbo} is a direct consequence of the results in \cite{CCR23}}. We refrain from introducing the notation involved, see \cite{CCR23} for details.

\smallskip

We start from \cite[Theorems~4.8 and~4.11]{CCR23}: by equations (4.18)-(4.19) and (4.24)-(4.25) with $\lambda = \hat\lambda / L$, the following inequality holds:
\begin{equation}\label{eq:moment-bound}
\begin{split}
	\big| \bbE \big[ \big( Z_{L}^{\beta}(\varphi,\psi)- \bbE[Z_{L}^{\beta}(\varphi,\psi)] \big)^h \big] \big|
	&\le \Big( \max_{I \ne *} \big\Vert \widehat \sfq_{L}^{|\varphi|,I} \tfrac{1}{\cW}
	\big\Vert_{\ell^p} \Big)
	 \Big( \max_{J \ne *} \big\Vert \cW \,
	\widebar{\sfq}_{L}^{|\psi|,J} \big\Vert_{\ell^q} \Big) \, \rme^{\hat\lambda}
	\sum_{r=1}^\infty \Xi^\bulk(r)
\end{split}
\end{equation}
where
\begin{equation}\label{eq:Xibulk}
	\Xi^\bulk(r) \,:= \!\!\!
	\sum_{\substack{I_1, \ldots, I_r \vdash\{1, \ldots, h\} \\
	\text{with full support} \\
	\text{and } I_i \ne I_{i-1} , \,\, I_i \ne * \ \forall i}}
	\!\!\! \bigg\{ \prod_{i=1}^r \big| \bbE[\xi_{\beta}^{I_i}] \big|  \bigg\} \,
	\big( \|\widehat \sfQ_{L}\|_{\ell^q\to\ell^q}^{\cW} \big)^{r-1}
	\, \big( \| |\widehat\sfU|_{L,\hat\lambda,\beta}\|_{\ell^q \to \ell^q}^{\cW} \big)^{r}  \,.
\end{equation}
The terms in these expressions are defined in \cite[Section~4]{CCR23}.
In a nutshell, $\cW$ is a weight function\footnote{We choose $\cW = \cW_t$
with $t = \hat t/\sqrt{L}$ as in \cite[Remark~4.12]{CCR23}.}
and $\widehat \sfq_{L}^{|\varphi|,I}$,
$\widebar \sfq_{L}^{|\varphi|,I}$ denote suitable averages of the boundary conditions $\varphi$,
$\psi$ with respect to the random walk kernel, while $\widehat\sfQ$ and $\widehat\sfU$
denote linear operators acting on functions on $(\Z^2)^h$.
\emph{We will next state estimates on each term
in \eqref{eq:moment-bound}-\eqref{eq:Xibulk}, taken verbatim from \cite{CCR23}, which
combined together will lead to \eqref{eq:genbo}.}

\smallskip

We denote by $\scrC_{\hat t}$, $\overline{\scrC_{\hat t}}$,
$\widehat{\scrC}_{\hat t}$, $\widecheck{\scrC}_{\hat t}$ suitable constants.
Let us first bound the terms in \eqref{eq:moment-bound}:
\begin{itemize}
\item by \cite[Proposition~4.19]{CCR23}, equation (4.44) with $r=1$ gives
\begin{equation}\label{eq:bo1}
	\Big(\max_{I \ne *} \big\Vert \widehat \sfq_{L}^{|\varphi|,I} \tfrac{1}{\cW} \big\Vert_{\ell^p}\Big)
	\le 4\, \scrC_{\hat t}^h \, q	\, \pigl\| \varphi(\cdot) \, \rme^{\frac{\hat t}{\sqrt{L}}|\cdot|} \pigr\|_{\ell^1}
	\, \pigl\| \varphi(\cdot) \, \rme^{\frac{\hat t}{\sqrt{L}}|\cdot|} \pigr\|_{\ell^p}^{h-1} \,;
\end{equation}

\item by \cite[Proposition~4.21]{CCR23}, equation (4.49) with $w_t(\cdot) = \rme^{-\frac{\hat{t}}{\sqrt{L}}|\cdot|}$ gives
\begin{equation}\label{eq:bo2}
	\Big(\max_{J \ne *} \big\Vert \cW \,
	\widebar{\sfq}_{L}^{|\psi|,J} \big\Vert_{\ell^q} \Big)
	\le \overline{\scrC_{\hat t}}^h \, p
	\, \pigl\|  \psi(\cdot) \, \rme^{-\frac{\hat t}{\sqrt{L}}|\cdot|} \pigr\|_{\ell^\infty}
	\, \pigl\|  \psi(\cdot) \, \rme^{-\frac{\hat t}{\sqrt{L}}|\cdot|} \pigr\|_{\ell^q}^{h-1}  \,.
\end{equation}
\end{itemize}
Plugging these bounds into \eqref{eq:moment-bound}, we obtain the second and third line in \eqref{eq:genbo}.

\smallskip

We next attach the factors $p$ and $q$ from \eqref{eq:bo1} and \eqref{eq:bo2} to $\Xi^\bulk(r)$. To complete the proof of \eqref{eq:genbo}, recalling \eqref{eq:Gammapar}, it suffices to show that for a suitable constant $\mathtt{K}_{h}^{(\hat t)} < \infty$
\begin{equation} \label{eq:goto}
	(p\, q) \, \Xi^\bulk(r) \le
	\bigl( p\, q \, \Gamma \bigr)^{\max\{r, \frac{h}{2}\}}
	\qquad \text{with} \quad
	\Gamma = \mathtt{K}_{h}^{(\hat t)} \,
	\frac{\sigma_\beta^2}{1-\sigma_{\beta}^2 \, R_{L}^{(\hat\lambda)}} \,.
\end{equation}
We bound the terms in \eqref{eq:Xibulk} as follows:\footnote{Propositions~4.13 and~4.24 in \cite{CCR23} require that $\beta \le \beta_0(h)$ for a suitable $\beta_0(h) > 0$ depending on the disorder distribution (to ensure that $\bbE[(\rme^{\beta \omega - \lambda(\beta)})^p] \le \bbE[(\rme^{\beta \omega - \lambda(\beta)})^2] = \sigma_\beta^2$ for all $3 \le p \le h$). This is why the same requirement appears in Theorem~\ref{th:moment-bound-CCR}.}
\begin{itemize}
\item by \cite[Proposition~4.23]{CCR23}, equation (4.45) gives
\begin{equation*}
	\|\widehat \sfQ_{L}\|_{\ell^q\to\ell^q}^{\cW}
	\le h! \, \widehat{\scrC}_{\hat t}^{\,h} \, p \, q \,;
\end{equation*}

\item by \cite[Proposition~4.24]{CCR23}, equation (4.58) gives
\begin{equation*}
	\| |\widehat\sfU|_{L,\hat\lambda,\beta}\|_{\ell^q \to \ell^q}^{\cW}
	\le 1 + \widecheck{\scrC}_{\hat t}^{\,h} \,
	\frac{\sigma_{\beta}^2 R_{L}^{(\hat\lambda)}}{1-\sigma_{\beta}^2  R_{L}^{(\hat\lambda)}}
	\le \frac{\widecheck{\scrC}_{\hat t}^{\,h}}{1-\sigma_{\beta}^2  R_{L}^{(\hat\lambda)}};
\end{equation*}

\item by \cite[Proposition~4.13]{CCR23}, there is $C(h)<\infty$ such that
\begin{gather*}
	\forall I_1, \ldots, I_r \vdash \{1,\ldots, h\} \ \emph{with full support} :
	\qquad
	\prod_{i=1}^r \big| \bbE[\xi_{\beta}^{I_i}] \big|
	\le C(h)^r \,
	\big( \sigma_{\beta}^2 \, \big)^{\max\{r,\frac{h}{2}\}} \,.
\end{gather*}
\end{itemize}
Overall, applying \eqref{eq:Xibulk} we can bound
\begin{equation*}
	(p \, q) \, \Xi^\bulk(r)
	\le (p \, q)^r \, \big( h! \, \widehat{\scrC}_{\hat t}^{\,h} \big)^{r-1} \,
	\bigg( \frac{\widecheck{\scrC}_{\hat t}^{\,h}}{1-\sigma_{\beta}^2  R_{L}^{(\hat\lambda)}} \bigg)^r
	\, C(h)^r \,
	\big( \sigma_{\beta}^2 \, \big)^{\max\{r,\frac{h}{2}\}} \,.
\end{equation*}
We increase the RHS replacing $r$ by $\max\{r,\frac{h}{2}\}$, which shows that \eqref{eq:goto} holds with
\begin{equation}\label{eq:K}
	\mathtt{K}_{h}^{(\hat t)}
	:= h! \, \widehat{\scrC}_{\hat t}^{\,h} \, \widecheck{\scrC}_{\hat t}^{\,h}\, C(h) \,.
\end{equation}
This completes the proof.\qed

\smallskip

\section*{Acknowledgements}
We thank the referees for helpful comments that improved the exposition of the paper. 
F.C.~is supported by INdAM/GNAMPA. R.S.~is supported by NUS grant A-8001448-00-00 and NSFC grant 12271475.


\end{document}